\documentclass{article}
\usepackage[T1]{fontenc}
\usepackage{lmodern}
\usepackage{mathtools, amsmath, amssymb, xcolor, amsthm, float, mathrsfs, algorithm2e, csquotes, enumerate}
\usepackage{algpseudocode}
\usepackage{booktabs}
\usepackage{pdfpages}
\usepackage[normalem]{ulem}
\mathtoolsset{showonlyrefs}
\usepackage{url}
\usepackage{hyperref}
\usepackage[export]{adjustbox}
\usepackage{geometry}
\usepackage{verbatim}
\usepackage{float}

\usepackage{bbm} 
\usepackage{pgfplots} 
\pgfplotsset{compat=1.17} 

\newcounter{dummy} 
\numberwithin{dummy}{section}

\newtheorem{remark}[dummy]{Remark}
\newtheorem{theorem}[dummy]{Theorem}
\newtheorem{lemma}[dummy]{Lemma}
\newtheorem{example}[dummy]{Example}

\newtheorem{proposition}[dummy]{Proposition}

\usepackage{appendix}

\DeclareMathOperator{\R}{\mathbb R}
\DeclareMathOperator{\N}{\mathbb N}
\let\P\relax
\DeclareMathOperator{\P}{\mathcal P}
\DeclareMathOperator{\C}{\mathcal C}
\DeclareMathOperator{\F}{\mathcal F}

\newcommand{\lebesgue}{\Lambda}
\DeclareMathOperator{\dom}{dom}

\DeclareMathOperator{\supp}{supp}

\DeclareMathOperator*{\argmin}{arg\,min}

\renewcommand{\d}{\mathop{}\!\mathrm{d}}

\usepackage[draft, todonotes={textsize=tiny}, commandnameprefix=always]{changes}
\setuptodonotes{color=red, backgroundcolor=red!60!white,
  bordercolor=red, tickmarkheight=10pt}
  
\definechangesauthor[name=Robert, color=cyan]{RB}

\definechangesauthor[name=Richard, color=orange]{RD}

\title{Wasserstein Gradient Flows of MMD Functionals with Distance Kernels under Sobolev Regularization
}

\author{
Richard Duong
\and
Nicolaj Rux
\and
Viktor Stein
\and
Gabriele Steidl
}

 \date{\today}
\footnotetext[1]{Institute of Mathematics,
TU Berlin,
Stra{\ss}e des 17. Juni 136, 
10587 Berlin, Germany,
\{duong, rux, stein, steidl\}@math.tu-berlin.de
}

\begin{document}

\maketitle

\begin{abstract}
    We consider Wasserstein gradient flows of  maximum mean discrepancy (MMD) functionals $\text{MMD}_K^2(\cdot, \nu)$   for positive and negative distance kernels
    $K(x,y) \coloneqq \pm |x-y|$ and given target measures $\nu$ on $\R$.    
    Since in one dimension the Wasserstein space can be isometrically embedded
    into the cone $\mathcal C(0,1) \subset L_2(0,1)$ of quantile functions, 
    Wasserstein gradient flows can be characterized by
    the solution of an associated Cauchy problem on $L_2(0,1)$.
        While for the negative kernel, the MMD functional is geodesically convex, this is not the case for the positive kernel, which needs to be handled to ensure the existence of the flow. We propose to add a regularizing Sobolev term $|\cdot|^2_{H^1(0,1)}$ corresponding to the Laplacian with Neumann boundary conditions
    to the Cauchy problem of quantile functions. Indeed, this ensures the existence of a generalized minimizing movement for the positive kernel. Furthermore, for the negative kernel, we demonstrate by numerical examples
    how the Laplacian rectifies a "dissipation-of-mass" defect of the MMD gradient flow. 
\end{abstract}

\section{Introduction}
Wasserstein gradient flows have been a research topic in stochastic analysis for a long time
and have recently received much interest in machine learning, among the huge amount of papers see, e.g., \cite{AKSG2019,CHHRS2023,HHABCS2023,LBADDP2022}.
In this paper, we study Wasserstein gradient flows of probability measures in $\P_2(\R)$ with respect to the maximum mean discrepancy (MMD) given by the positive or negative distance kernel $K(x,y) \coloneqq \pm |x-y|$. Such \emph{non-smooth} kernels exhibit a flexible flow behavior: while the MMD flow for smooth kernels is rigid in the sense that empirical measures stay empirical, for non-smooth kernels like the negative distance kernel, point measures can become absolutely continuous and vice-versa, see Figure \ref{fig:Dirac_to_Dirac_Intro} and \cite[Section 6.1]{HBGS2023}.

Given the interesting properties of MMD flows with non-smooth kernels, their analysis in the multivariate case $\P_2(\R^d)$  is quite difficult. In case of the negative distance kernel, the MMD func\-tional becomes geodesically non-convex, see \cite{HGBS2022}, and the mere existence of such Wasserstein gradient flows is unclear in general.

The situation becomes more tangible in one dimension: in \cite{BCDP15}, the \emph{interaction energy} part of the MMD was considered, also for both the negative and positive distance kernel. On the real line, Wasserstein gradient flows can be equivalently described by the flow of \emph{quantile functions} with respect to an associated functional on $L_2(0,1)$. For both kernels, this associated functional has the form $F = F_1 + F_2$, where $F_1$ is \emph{linear} and $F_2$ is a {convex indicator function}. It is hence convex on $L_2(0,1)$, and the existence of the quantile-/Wasserstein gradient flow is established by standard semigroup theory.

The problem gets more involved when considering the \emph{whole} MMD functional, including the interaction energy and \emph{the potential energy} part with respect to a fixed reference measure $\nu \in \P_2(\R)$. For the negative distance kernel, a comprehensive analysis of the MMD flow, including general existence,  properties and numerical approximation, was given  in \cite{DSBHS2024}. Here, the functional $F_1$ becomes \emph{nonlinear}, but still convex, and again, nonlinear semigroup theory can be employed by the convexity of $F = F_1 + F_2$. 

The situation becomes worse in case of the whole MMD with the \emph{positive} distance kernel. Here, the associated $L_2$-functional has the form $F = -F_1 + F_2$, where $-F_1$ is \emph{nonlinear and non-convex}. Now, the semigroup theory of subdifferentials of convex functionals cannot be applied to deduce the general existence of the Wasserstein gradient flow.

This theoretical shortfall also applies for the negative kernel in multiple dimensions, as already stated above: the MMD becomes geodesically non-convex. More practically speaking, the repulsive nature of the negative distance kernel leads to \emph{mass spreading to infinity}. Concerning another Riesz kernel, the {Coulomb kernel}, this problem was already observed in \cite{BV2023}. But also in the one-dimensional case of the negative distance kernel, a slightly toned-down version of this problem can be observed. More precisely, in \cite{HBGS2023}, the explicit MMD flow starting in a Dirac measure $\delta_{-1}$ towards a Dirac target $\delta_0$ was computed. The flow has the "\emph{dissipation-of-mass}" defect visualized in the top row of Figure \ref{fig:Dirac_to_Dirac_Intro}:
first, the mass extends until it reaches the target at $0$. 
Then, a Dirac point grows at $0$, 
while the mass on $(-1,0)$ slowly dissipates and gets arbitrarily slim.

\begin{figure}[t]
    \centering
    \includegraphics[width=.16\textwidth]{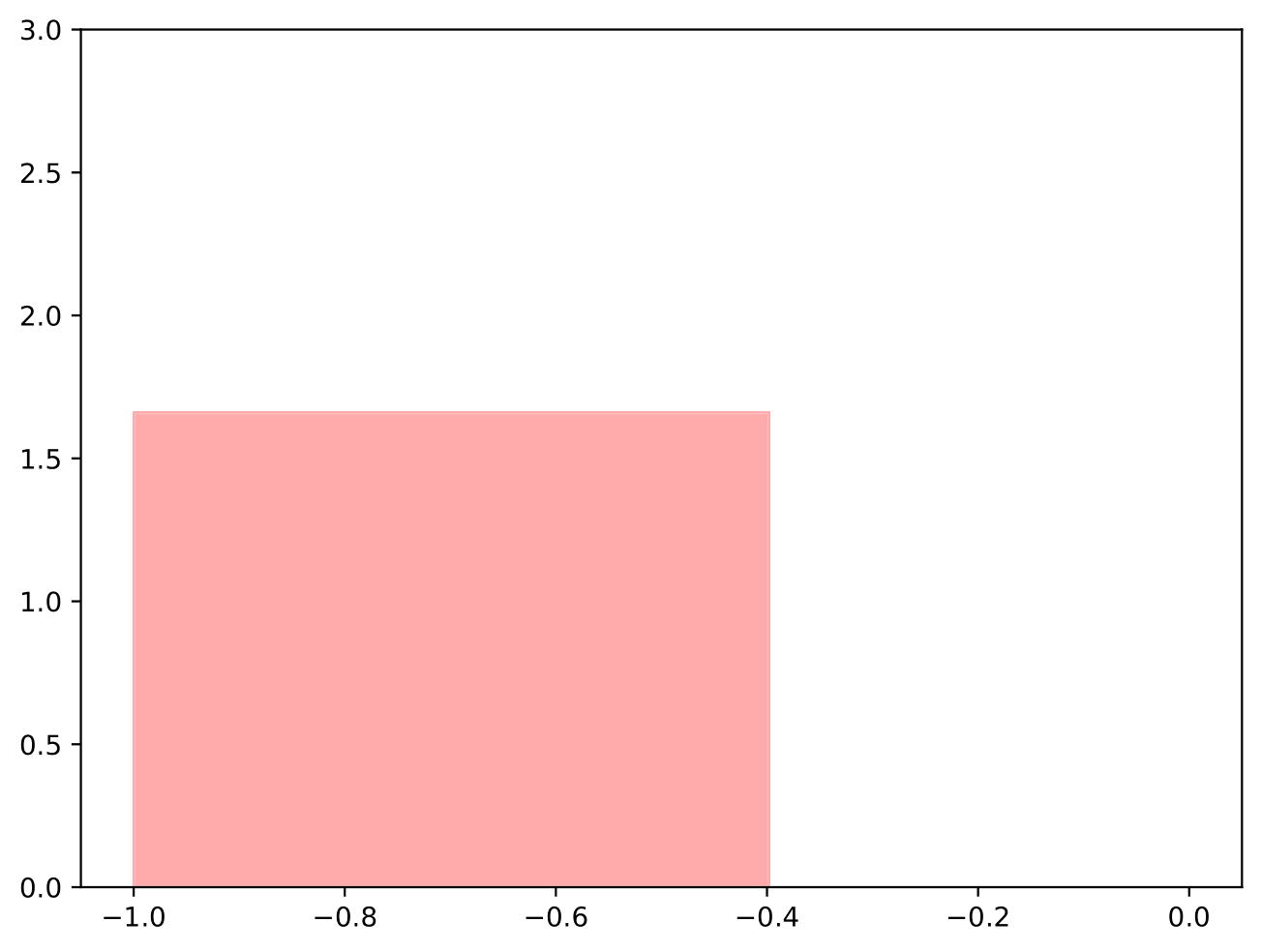}
    \includegraphics[width=.16\textwidth]{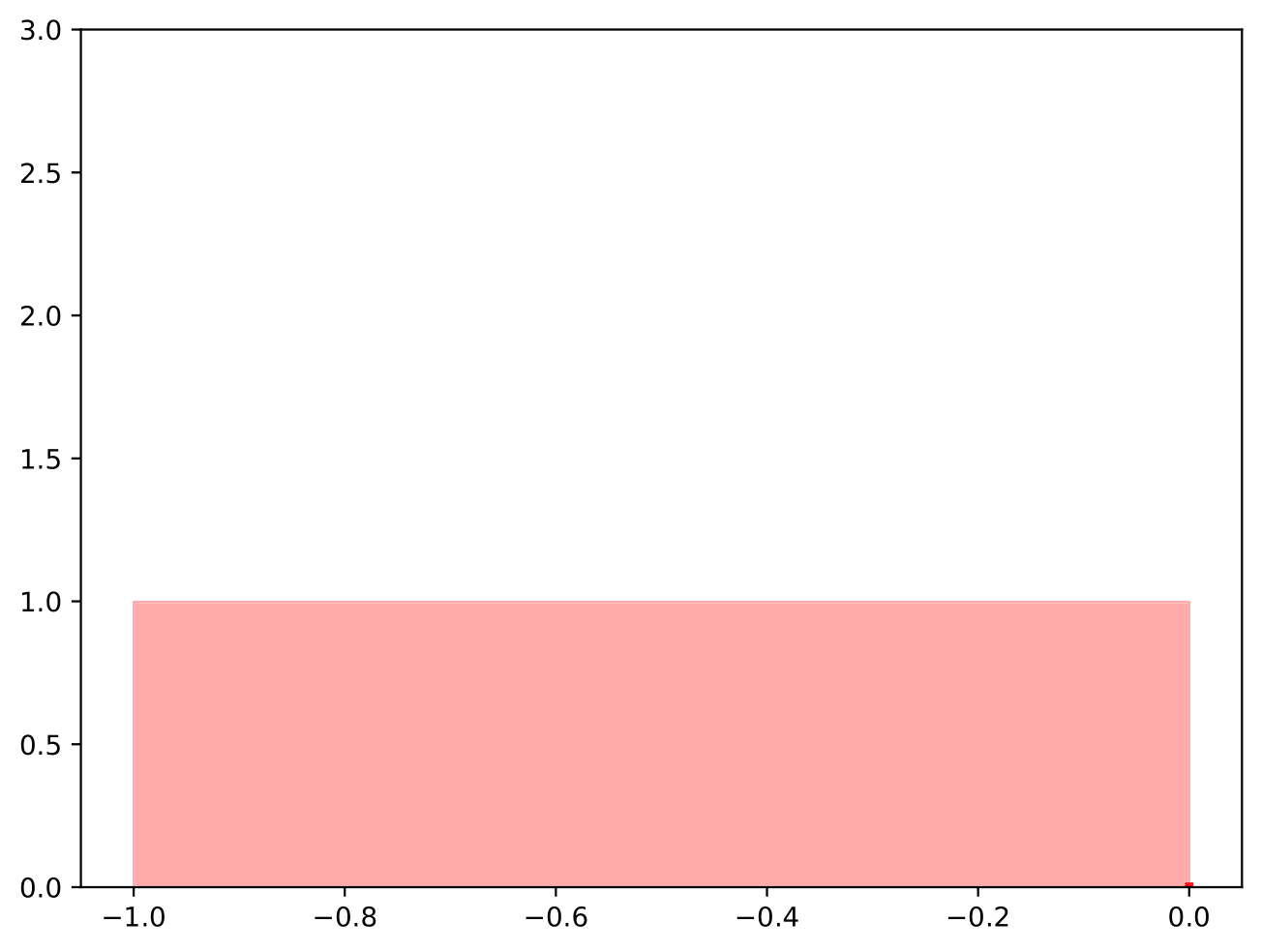}
    \includegraphics[width=.16\textwidth]{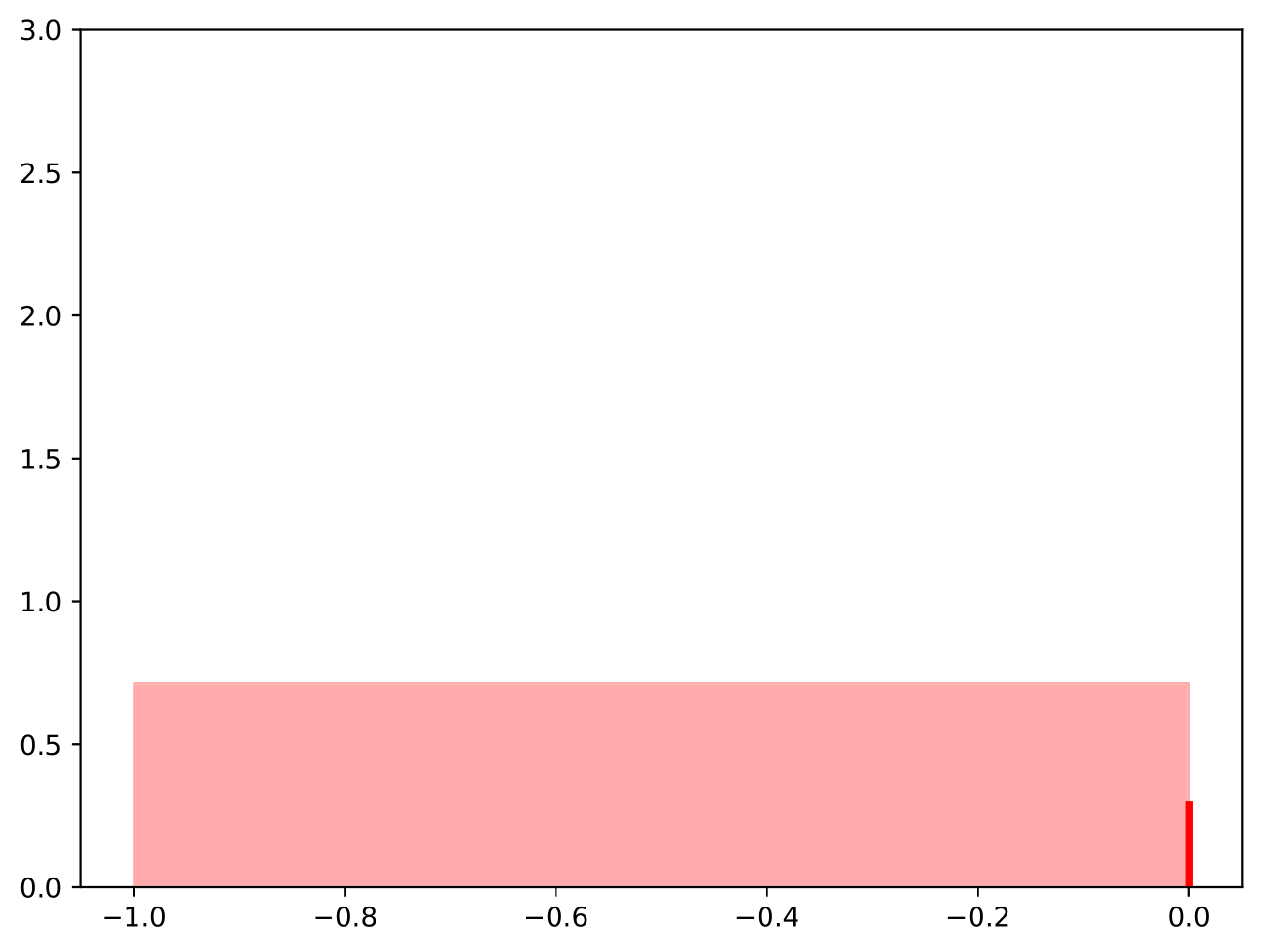}
    \includegraphics[width=.16\textwidth]{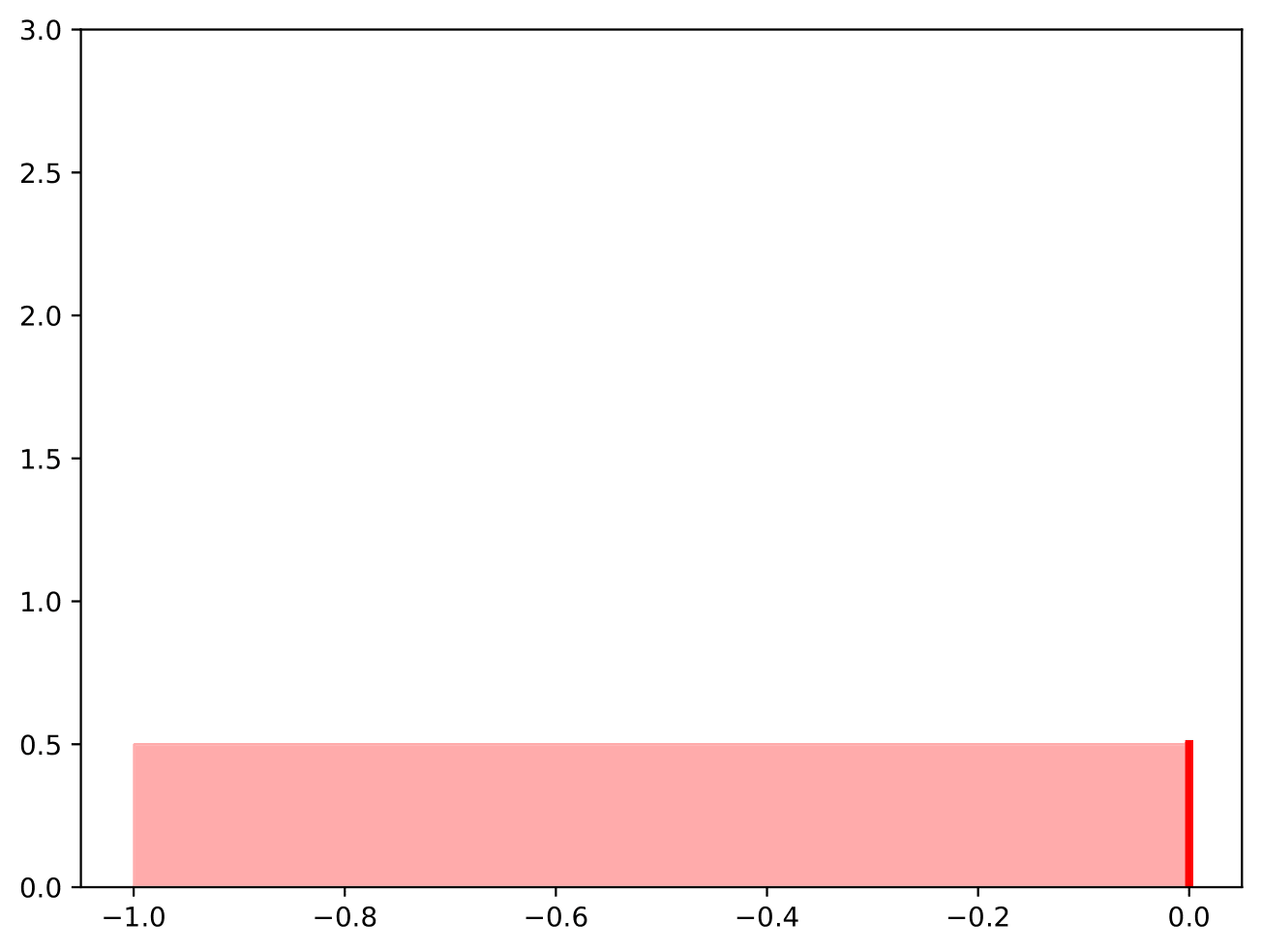}
    \includegraphics[width=.16\textwidth]{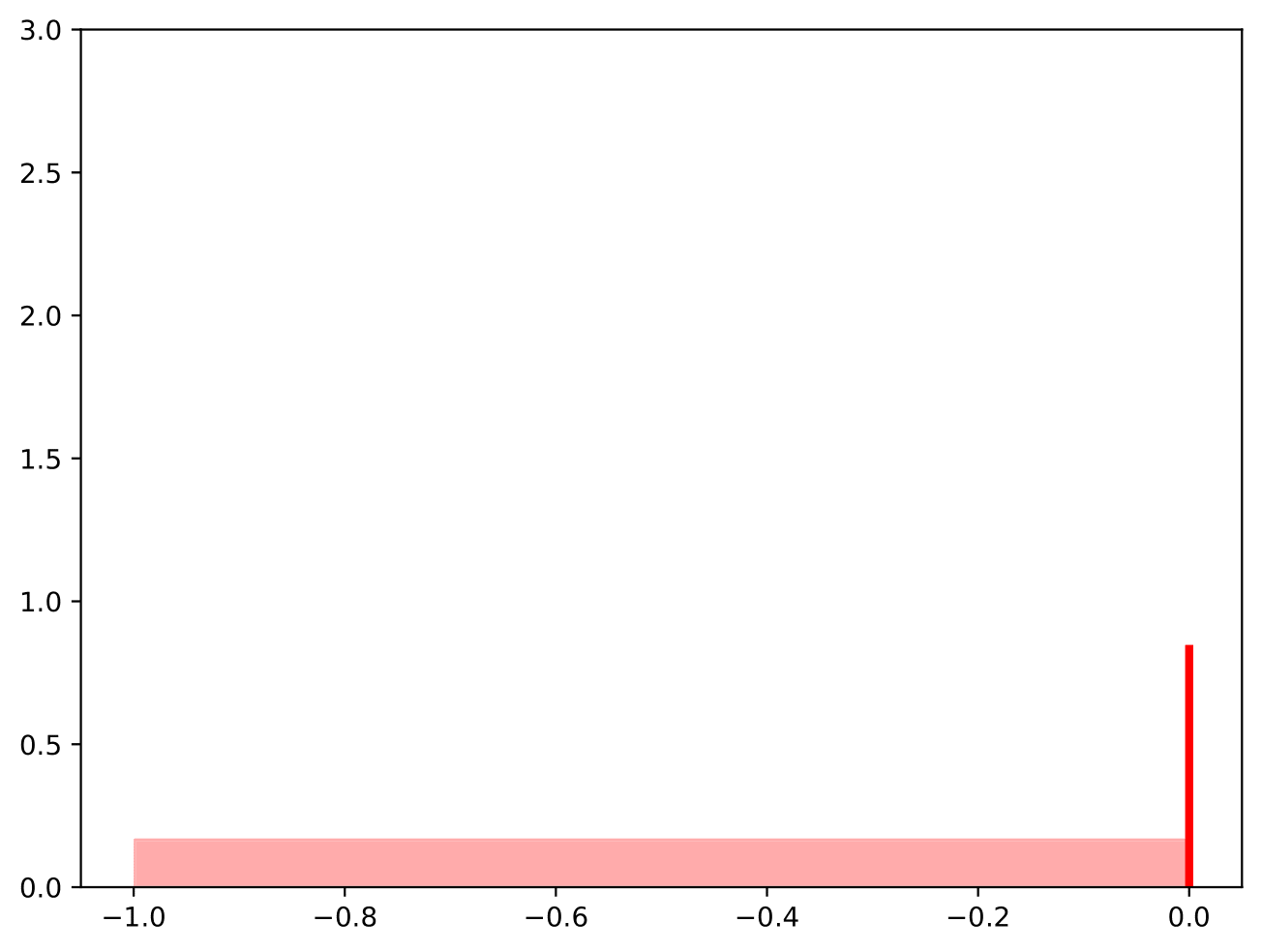}
    \includegraphics[width=.16\textwidth]{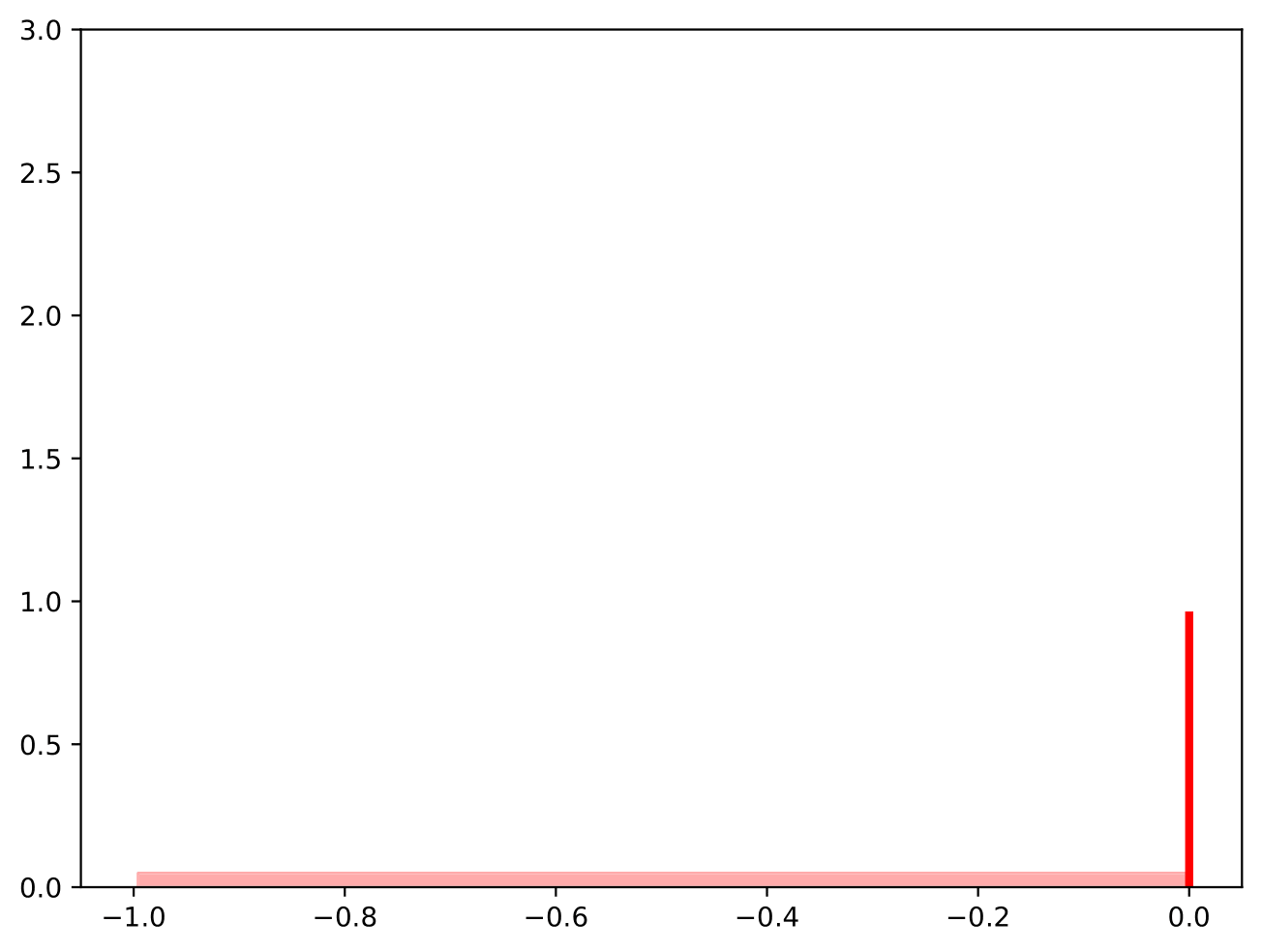}
\\[2ex]
    \includegraphics[width=.16\textwidth]{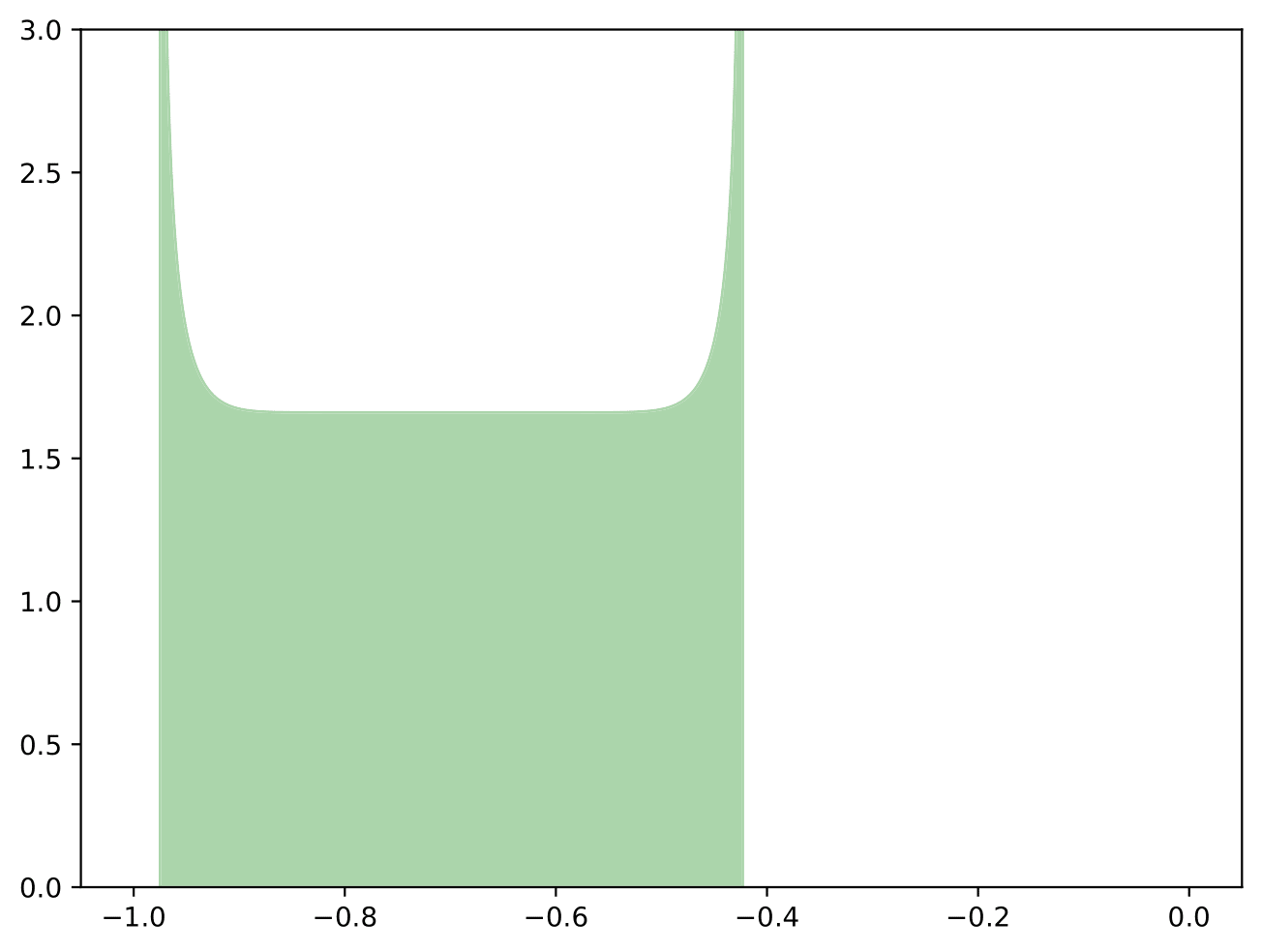}
    \includegraphics[width=.16\textwidth]{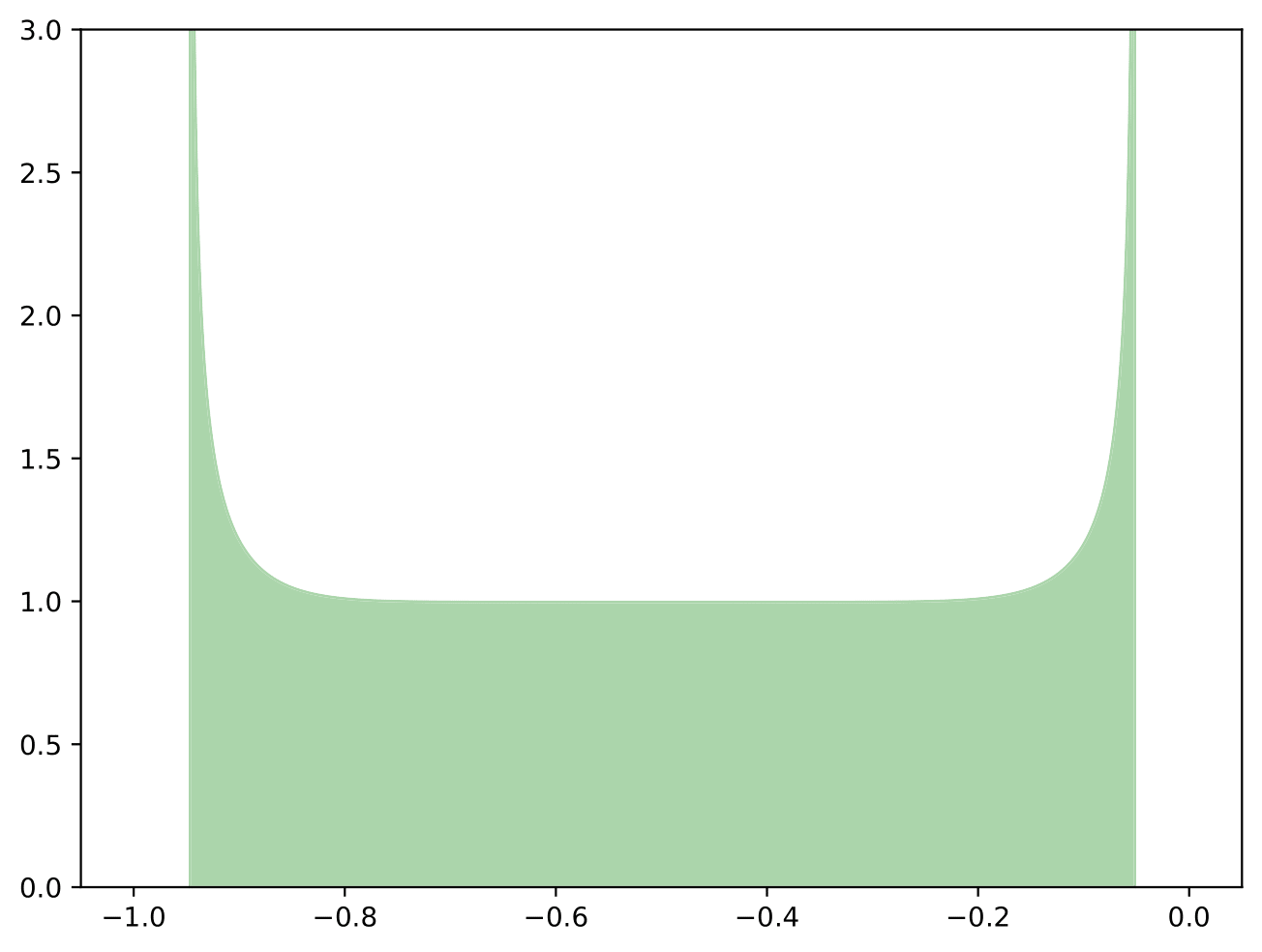}
    \includegraphics[width=.16\textwidth]{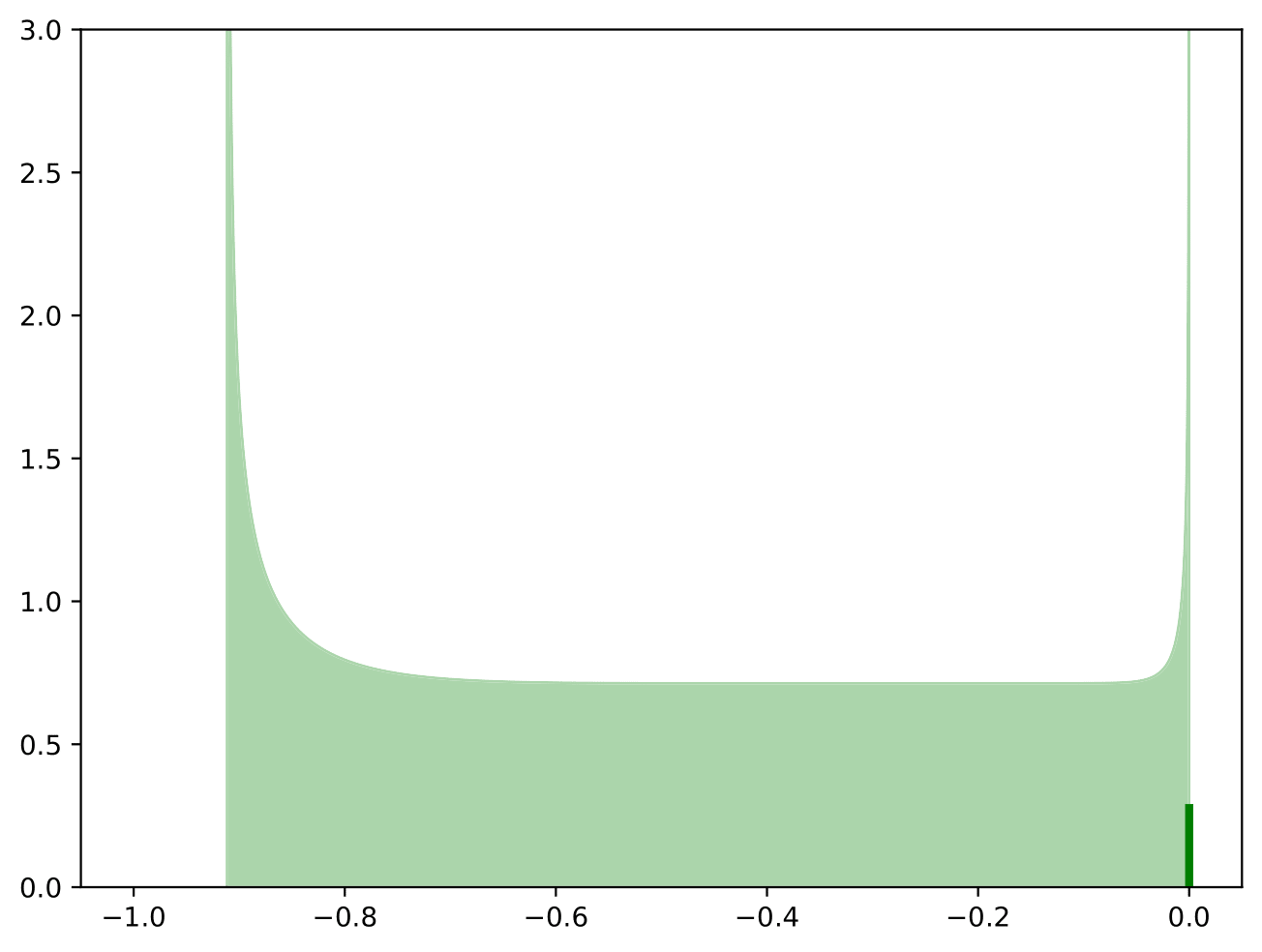}
    \includegraphics[width=.16\textwidth]{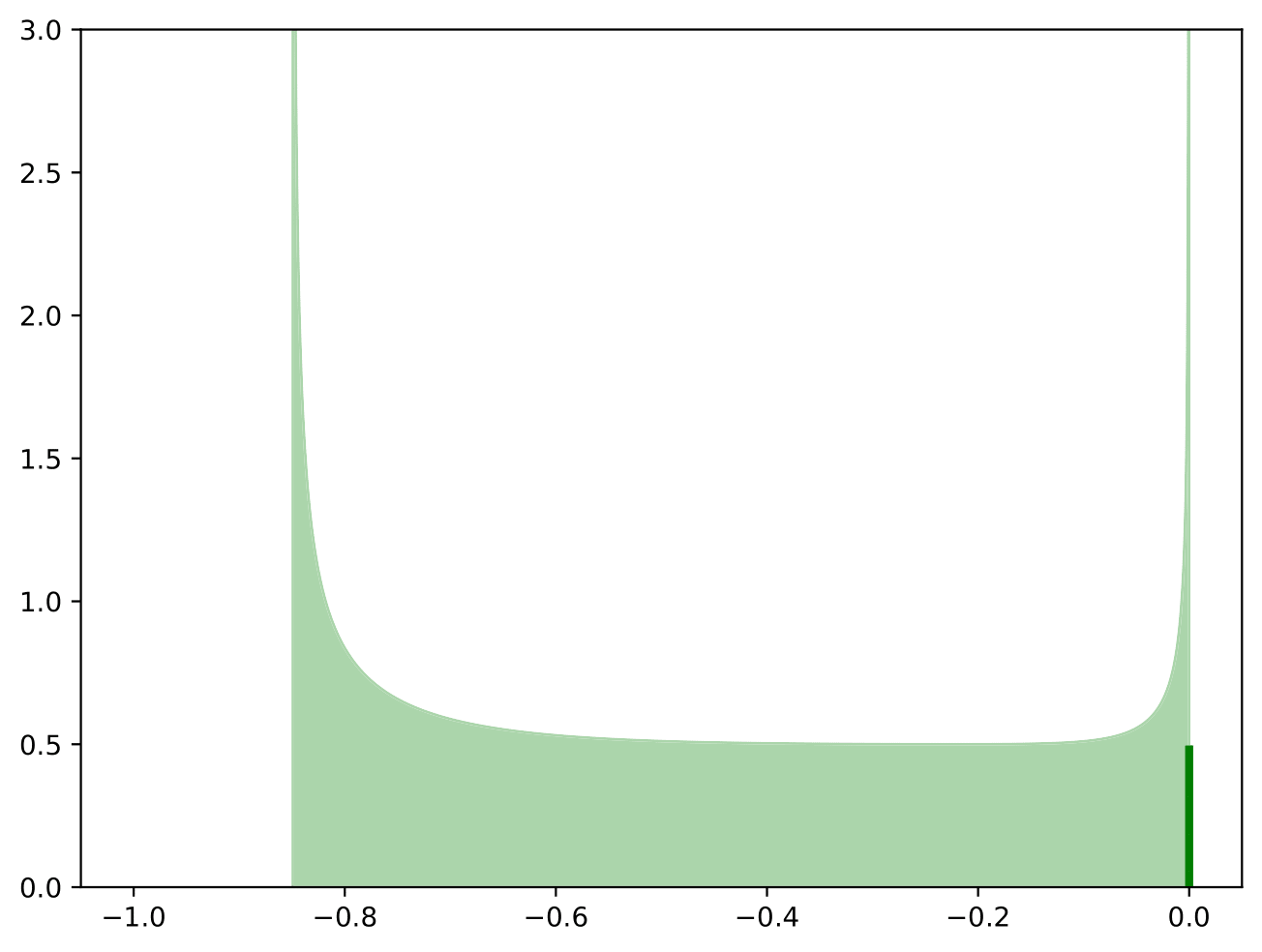}
    \includegraphics[width=.16\textwidth]{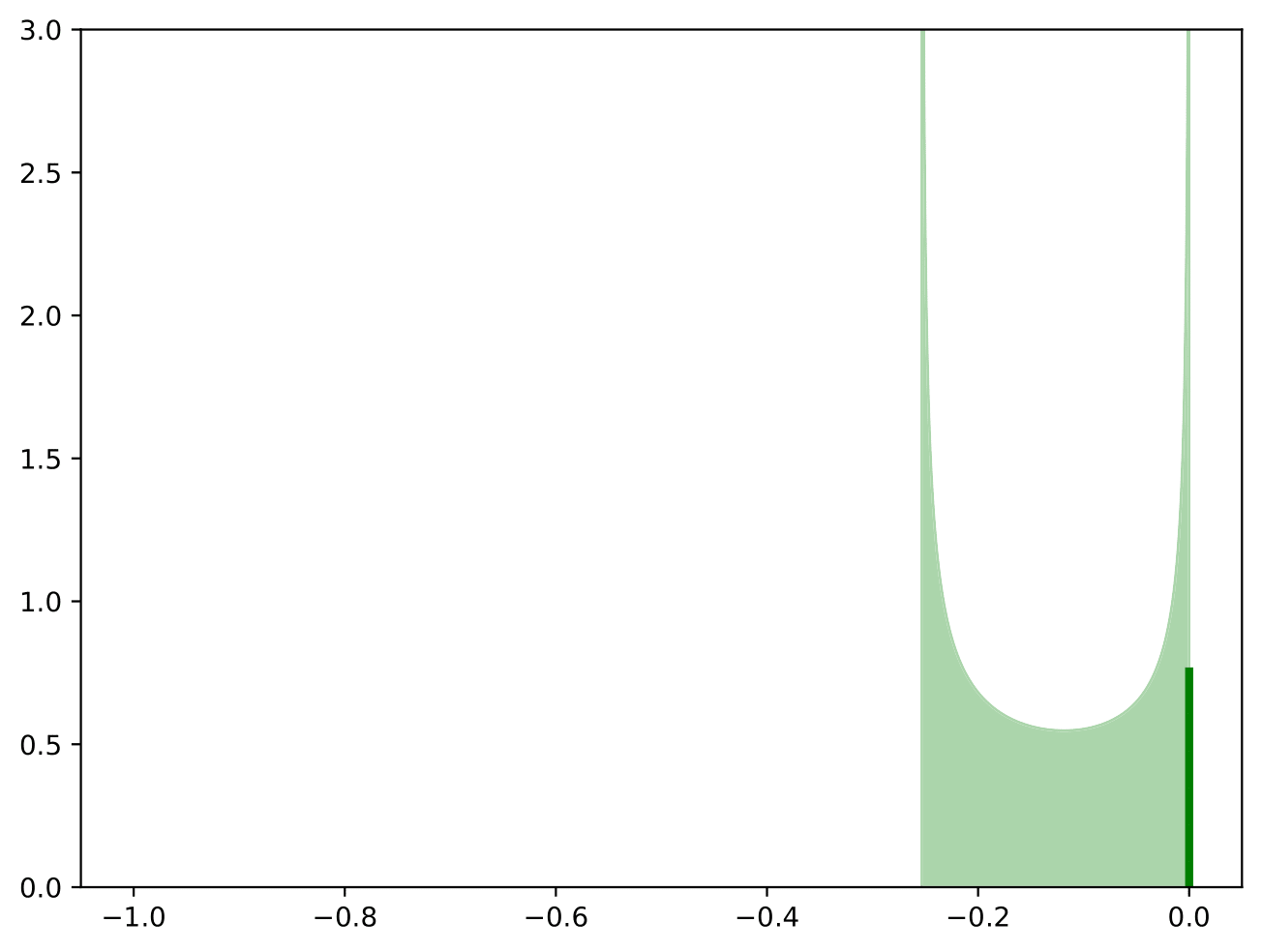}
    \includegraphics[width=.16\textwidth]{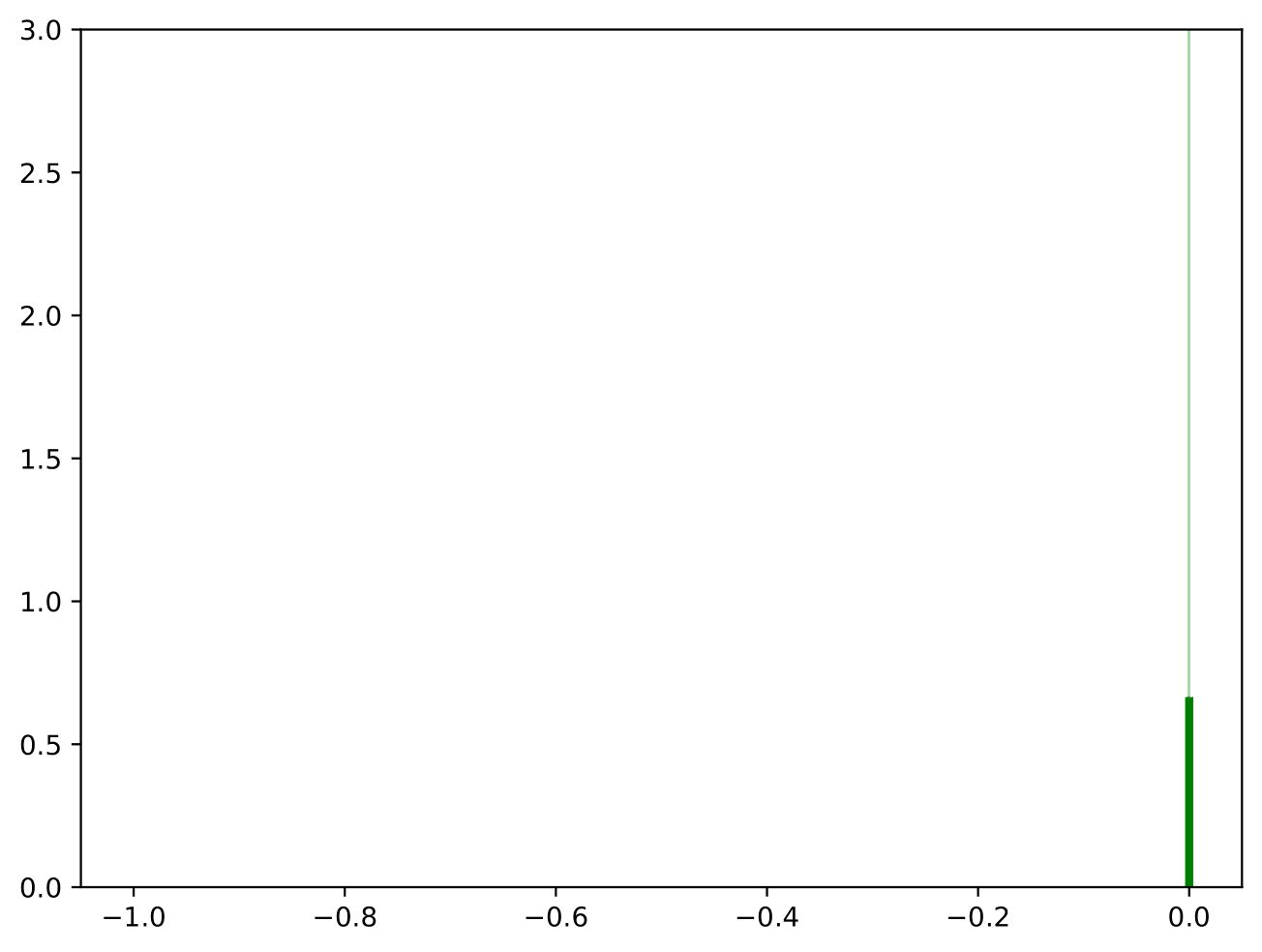}

    \caption{\textbf{Top}: MMD flow with negative distance kernel from $\delta_{-1}$ to $\delta_0$ has the "dissipation-of-mass" defect that the mass in $(-1,0)$ slowly dissipates towards zero, but does not vanish on any part of the interval.
    \textbf{Bottom}: Correction of the "dissipation-of-mass" defect by adding a Sobolev regularization term. 
        }
    \label{fig:Dirac_to_Dirac_Intro}
\end{figure}

In this paper, we show how to tackle the above shortcomings by introducing a regularized version of the associated $L_2$-functional by adding a Sobolev term $|\cdot|^2_{H^1(0,1)}$.
This entails multiple advantages:
\begin{itemize}
    \item Existence of the (whole) regularized MMD gradient flow in 1D for the \emph{negative} distance kernel, see Theorem \ref{main_mmd}; and existence of a generalized minimizing movement for the \emph{positive} kernel, see Theorem \ref{main_mmd2}. 
    In the latter case, the measure $\nu$ plays the role of a repulsion point, while in the former one, it attracts as a target.
    \item Circumvention of the dissipation-of-mass defect of the negative distance kernel as in the bottom row of Figure \ref{fig:Dirac_to_Dirac_Intro},
    see Example \ref{ex:dirac-to-dirac-reg} and Section \ref{sec:numerics-examples}. 
    This might be a promising approach in the multi-dimensional case to prevent the spreading of mass to infinity, or at least, to control it by the regularization in future work.
    \item Capitalizing on the strong repulsive behavior of the unregularized MMD with negative kernel, allowing initial Dirac points to instantly become absolutely continuous and resulting in a quick expansion of mass towards the target. The interplay of repulsion and regularization can be controlled via a regularization/diffusion constant $\lambda > 0$.
\end{itemize}
 
\paragraph{Outline of the paper.}
In Section~\ref{sec:prelim_1}, we recall all the notions needed to generally define Wasserstein gradient flows on the Wasserstein space $\P_2(\R^d)$. Starting in Subsection~\ref{subsec:1D-flows}, we restrict ourselves to the one-dimensional case $\P_2(\R^1)$, and recall the equivalence of Wasserstein gradient flows to the $L_2$-gradient flows of quantile functions. Section~\ref{sec:mmd} introduces the MMD functional, its associated $L_2$-functional as well as its newly proposed Sobolev regularization. Also, the main existence results, Theorem \ref{main_mmd} and \ref{main_mmd2}, next to explicit Examples \ref{ex:dirac-to-dirac-reg} (Dirac-to-Dirac) and \ref{ex:dirac-from-dirac} (Dirac-away-from-Dirac), are featured there.
The final Section \ref{sec:numerics} deals with the numerical approximation of the regularized MMD flow with the negative kernel, including several visualizations of the (quantile) flow, which illustrate its advantage over the unregularized MMD.

\section{Wasserstein Gradient Flows}\label{sec:prelim_1}
In this section, we  briefly introduce the notation of Wasserstein gradient flows for measures on $\R^d$
and  its simplifications for $\R$. For a more comprehensive introduction, we refer to 
\cite{BookAmGiSa05}, and for the one-dimensional case to \cite{DSBHS2024}.

\subsection{Flows on \texorpdfstring{$\mathcal P_2(\R^d)$}{P2(Rd)}}
The \emph{Wasserstein space}
$\P_2(\R^d)$  consists of all probability measures  on $\R^d$  with finite second moments together with
the \emph{Wasserstein distance} 
$W_2\colon\P_2(\R^d) \times \P_2(\R^d) \to [0,\infty)$ defined by
\begin{equation}        \label{def:W_2}
  W_2^2(\mu, \nu)
  \coloneqq 
  \min_{\pi \in \Gamma(\mu, \nu)} 
  \int_{\R^d\times \R^d}
  \|x - y\|_2^2
  \d \pi(x, y),
  \qquad \mu,\nu \in \P_2(\R^d),
\end{equation}
where 
$
\Gamma(\mu, \nu)
\coloneqq
\{ \pi \in \P_2(\R^d \times \R^d):
(\pi_{1})_{\#} \pi = \mu,\; (\pi_{2})_{\#} \pi = \nu\}
$ 
and $\pi_i(x) \coloneqq x_i$, $i = 1,2$ for $x = (x_1,x_2)$, see, e.g., \cite{Vil03,BookVi09}.
Let $L_{2,\mu}$ denote the Bochner space of (equivalence classes of) 
functions $\xi:\R^d \to \R^d$ with
$\|\xi \|_{L_{2,\mu}} ^2 \coloneqq \int_{\R^d} \|\xi(x)\|_2^2 \d \mu(x) < \infty$.

A curve $\gamma\colon (a,b) \to \P_2(\R^d)$ is \emph{absolutely continuous}\footnote{Here, we already consider the case ${AC}^2((a,b); \P_2(\R^d))$ according to \cite{BookAmGiSa05}.}, if there exists a Borel velocity field $v$ of functions $v_t \colon \R^d \to \R^d$ with $\int_a^b \| v_t \|_{L_{2,\gamma_t}}^2 \d t < \infty$ such that the \emph{continuity equation}
\begin{equation} \label{eq:CE}
  \partial_t \gamma_t + \nabla_x \cdot ( v_t \, \gamma_t) = 0
\end{equation}
holds on $(a,b) \times \R^d$ in the sense of distributions \cite[Thm.~8.3.1]{BookAmGiSa05}.
The velocity field $v_t$ in \eqref{eq:CE} can be chosen uniquely by demanding that it has minimal norm $\|v_t\|_{L_2, \gamma_t}$.
A locally absolutely continuous curve 
  $\gamma \colon (0, \infty) \to \P_2(\R^d)$ 
  with minimal norm velocity field $v_t$ 
  is called \emph{Wasserstein gradient flow 
  with respect to} a lower semicontinuous (lsc) function $\F\colon \P_2(\R^d) \to (-\infty, \infty]$ 
  if 
  \begin{equation}\label{wgf}
      v_t \in - \partial \F(\gamma_t) \quad \text{for a.e. } 0 < t < \infty,
  \end{equation}
  with the reduced Fr\'echet subdifferential $\partial \F(\gamma_t)$.

The next theorem from \cite[Thm.~11.2.1]{BookAmGiSa05}  ensures the existence of Wasserstein gradient flows via so-called minimizing movement schemes for 
$\lambda$-convex functions along (generalized) geodesics.
To this end, recall that a curve $\gamma \colon [a,b] \to \P_2(\R^d)$ is called a \emph{geodesic} 
if there exists a constant $C \ge 0$ 
such that 
$    W_2(\gamma_{t_1}, \gamma_{t_2}) = C |t_2 - t_1|$ 
for all $t_1, t_2 \in [a,b]$. For $\lambda \in \R$, a function $\F\colon \P_2(\R^d) \to (-\infty,\infty]$ is called 
\emph{$\lambda$-convex along geodesics} if, for every 
$\mu, \nu  \in \P_2(\R^d)$ with finite $\mathcal F$ value,
there exists at least one geodesic $\gamma \colon [0, 1] \to \P_2(\R^d)$ 
between $\mu$ and $\nu$ such that
\begin{equation} \label{def:lambda_convex}
    \F(\gamma_t) 
    \le
    (1-t) \, \F(\mu) + t \, \F(\nu) 
    - \tfrac{\lambda}{2} \, t (1-t) \, W_2^2(\mu, \nu), 
    \qquad t \in [0,1].
\end{equation}
Further, recall the notion of the {minimizing movement (MM) scheme}: 
Given $\mu_0 \in \P_2(\R^d)$, 
consider piecewise constant curves $\gamma_\tau$ 
constructed from the minimizers (assuming existence) 
    \begin{equation} \label{mms}
        \mu_{n+1} \in  \argmin_{\mu \in \P_2(\R^d)} \Big\{ \F(\mu) + \frac{1}{2 \tau} W_2^2(\mu_{n}, \mu) \Big\}, \qquad \tau > 0,
    \end{equation}
    via $\gamma_\tau|_{(n \tau,(n+1)\tau]} \coloneqq \mu_n, ~n \in \N_0$. We say that $\gamma:[0,\infty) \to \P_2(\R^d)$ is a \emph{Minimizing Movement}
    if for \emph{any} sequence of step sizes $\tau_k \downarrow 0$, the discrete curve $(\gamma_{\tau_k})$ satisfies
    \begin{equation}\label{eq:mms-convergence}
       \lim_{k \to \infty} \gamma_{\tau_k}(t) = \gamma(t) \quad \text{for all } t \in [0,\infty),
    \end{equation}
    where the limit is taken in $\P_2(\R^d)$,
    and we write $\gamma \in {\rm MM}(\F, \mu_0)$.

\begin{theorem}[Existence and uniqueness of Wasserstein gradient flows] \label{thm:WGF}
    Let $\F \colon \P_2(\R) \to (- \infty, \infty]$ be bounded from below, lsc and $\lambda$-convex along geodesics for some $\lambda \in \R$. 
    Let $\mu_0 \in \P_2(\R)$ with $\mathcal F(\mu_0) < \infty$.
    Then,  there exists a unique Wasserstein gradient flow $\gamma \colon (0, \infty) \to \P_2(\R)$ with respect to $\F$ with $\gamma(0+) = \mu_0$.
            Furthermore, the minimizers \eqref{mms} are unique, the gradient flow $\gamma$ is given by the Minimizing Movement $\gamma \in {\rm MM}(\F, \mu_0)$ and the convergence \eqref{eq:mms-convergence} is {uniform on} $(0,\infty)$.
    
    If $\lambda > 0$, then $\F$ admits a unique minimizer $\bar \mu$ and we observe exponential convergence of $\gamma_t$ to $\bar \mu$ as $t\to \infty$.
    If $\lambda = 0$ and $\bar{\mu}$ is a minimizer of $\F$, then we still have
    $
        \F(\gamma_t) - \F(\bar{\mu})
        \le \frac{1}{2 t} W_2^2(\mu_0, \bar{\mu}).
    $
    
The above also holds true in $\R^d$ when switching to so-called generalized geodesics, see \cite[Sections 9.2 and 11.2]{BookAmGiSa05}.
\end{theorem}

\subsection{Flows on \texorpdfstring{$\mathcal P_2(\R)$}{P2(R)}}\label{subsec:1D-flows}
For $d =1$, Wasserstein gradient flows can be handled via gradient flows in the Hilbert space $L_2(0, 1)$. To this end, we have to introduce
the \emph{cumulative distribution function} (CDF) of $\mu \in \P(\R)$,
\begin{equation*}
    R_{\mu}^+ \colon \R \to [0, 1], \quad
    R_{\mu}^+(x) \coloneqq \mu\big( (-\infty, x] \big), 
\end{equation*}
and 
the related function
\begin{equation*}
    R_{\mu}^{-} \colon \R \to [0, 1], \quad
    R_{\mu}^{-}(x) \coloneqq \mu\big( (- \infty, x) \big) .
\end{equation*}
Then, the \emph{quantile function} of $\mu$ is given by 
\begin{equation*}
    Q_{\mu} \colon (0, 1) \to \R, \quad
    Q_{\mu}(s) \coloneqq \min\{ x \in \R: R_{\mu}^+(x) \ge s \}.
\end{equation*}
The functions are visualized for an empirical measure in Figure \ref{fig:R_and_Q}. For an overview on quantile functions in convex analysis, see also \cite{RoRo14}.
\begin{figure}[ht]
    \centering
     \includegraphics[height=3.5cm,valign=c]{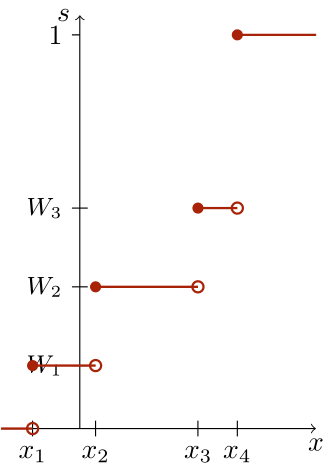}
     \hspace{0.5cm}
        \includegraphics[height=3.5cm,valign=c]{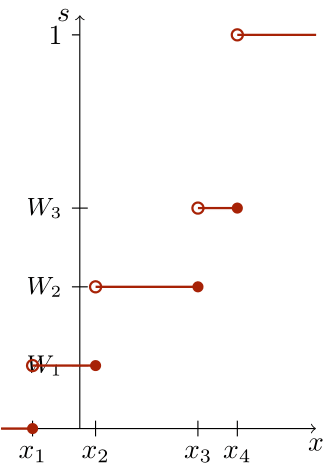}
     \hspace{1.5cm}
    \includegraphics[height=3.5cm,valign=c]{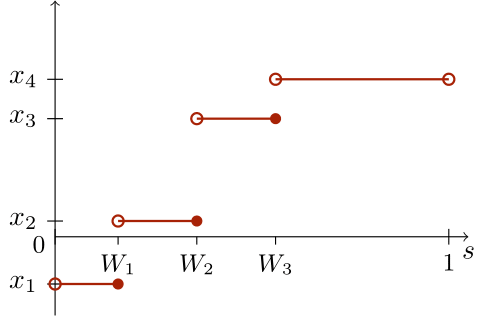}
    \caption{Functions $R_{\mu}^{+}$ (left), $R_{\mu}^{-}$ (middle) and $Q_{\mu}$ (right) for a discrete measure $\mu \coloneqq \sum_{k = 1}^{4} w_k \delta_{x_k}$ with $W_k \coloneqq \sum_{j = 1}^{k} w_j$. Source: see \cite{DSBHS2024}.}
    \label{fig:R_and_Q}
\end{figure}

\begin{remark}[Properties of $R_{\mu}^{\pm}$ and $ Q_{\mu} $]\label{l:DLOW-eq-lem}
The functions $R_\mu^{\pm}$ and $Q_\mu$ are monotonically increasing with only countably many discontinuities.
The function $Q_\mu$ is strictly increasing 
if and only if $R_\mu^+$ is continuous, and continuous if and only if $R_\mu^+$ is strictly increasing on $\overline{(R_\mu^+)^{-1}(0,1)}$. 
The points of continuity $x \in \R$ of $R_\mu^+$ and $R_\mu^-$ coincide and there, it holds $R_\mu^-(x) = R_\mu^+(x)$. 
We have $R_\mu^- \le R_\mu^+$ such that 
$$
[R_\mu^-(x), R_\mu^+(x)] \ne \emptyset \hspace*{4mm} \text{ for all } x \in \R.
$$
Both $R_\mu^-$ and $Q_\mu$ are left-continuous and, since they are increasing, also lsc, whereas $R_\mu^+$ is right-continuous and, since it is increasing, upper semicontinuous. If $\mu(\{x\}) = 0$ for all $x \in \R$, then the functions 
 $R_{\mu}^{\pm}$ are continuous and they coincide.
\end{remark}

The quantile functions of measures in $\mathcal P_2(\R)$ form the closed, convex cone
$$
\mathcal C (0,1) := \{ u \in L_2(0,1): u ~ \text{is increasing (a.e.)} \}.
$$

By the following theorem, see, e.g., \cite[Thm.~2.18]{Vil03}, the mapping 
$$\mathcal I \colon \mu \mapsto Q_\mu$$ 
is an isometric embedding 
of $\P_2(\R)$ into $L_2( 0,1 )$.

\begin{theorem}\label{prop:Q}
    For $\mu, \nu \in \P_2(\R)$,
    the quantile function $Q_{\mu} \in \mathcal C(0,1)$ 
		satisfies $\mu = (Q_{\mu})_{\#} \lebesgue_{(0,1)} \coloneqq \lebesgue_{(0,1)} \circ Q_{\mu}^{-1}$
  with the Lebesgue measure $\lebesgue_{(0,1)}$ on $(0,1)$,
   	and $\mathcal{I}$ is an isometry, i.e.,
    \begin{equation}
        W_2^2(\mu, \nu) = \int_{0}^1 |Q_{\mu}(s) - Q_{\nu}(s)|^2 \d s.
    \end{equation}    
\end{theorem}

The theorem gives rise to the following definition.
We call a functional $F \colon L_2( 0,1 ) \to (-\infty,\infty]$ \emph{associated with
a functional}
$\F \colon \P_2(\R) \to (-\infty,\infty]$ if
\begin{equation}\label{eq:assoc}
F (Q_\mu) = \F(\mu) \quad \text{for all } \mu \in \P_2(\R).  
\end{equation}
Clearly, there are many functions $F \colon L_2( 0,1 ) \to (-\infty,\infty]$ associated to a function $\F$,
since \eqref{eq:assoc} determines $F$ only on $\mathcal C(0,1) \subset L_2(0,1)$.
Recall that 
$F$
is called \emph{proper} if 
$\dom(F) \coloneqq \{u \in L_2(0,1): F(u) < \infty\} \ne \emptyset$.
Further, the  \emph{regular subdifferential} 
of a function $F \colon L_2( 0,1 ) \to (-\infty,\infty]$ 
is defined by 
\begin{equation} \label{eq:subdiff_1}
    \partial F(u) \coloneqq 
    \bigl\{ v \in L_2( 0,1 ): F(w) \ge F(u) + \langle v, w-u \rangle +o(\|w-u\|) \; ~\text{for all } w\in L_2(0,1)\bigr\}.
\end{equation}
If $F$ is convex, the $o$-term in \eqref{eq:subdiff_1} can be skipped. 
If $F$ is proper, then it holds 
$$
\dom(\partial F) \coloneqq \{u \in L_2(0,1) : \partial F(u) \ne \emptyset \}\subseteq \dom(F),
$$
and if $F$ is in addition convex and lsc, we have $\overline{\dom(\partial F)} = \dom(F)$.

Wasserstein gradient flows can be characterized by the strong solution of a corresponding Cauchy problem in $L_2(0,1)$.
To this end, recall that for an
operator $A \colon \dom(A) \subseteq L_2(0,1) \to 2^{L_2(0,1)}$ 
and an initial function $g_0 \in L_2(0,1)$, 
a \emph{strong solution}
$g \colon [0, T) \to L_2(0,1)$ of the Cauchy problem
\begin{align}\label{eq:cauchy-strong}
    \begin{cases}
    \partial_t g(t) + Ag(t) ~\ni~ 0,  \quad 0 < t < T \le \infty, \\
    g(0) ~=~ g_0,
    \end{cases}
\end{align}
is a locally absolutely continuous function $g \in H^{1}_{{\rm loc}} \big((0,T);L_2(0,1) \big)$
which is continuous in $t=0$, meets the initial condition and solves\footnote{Implicitly, it is required that $g(t) \in \dom(A)$ for a.e. $t$, and that the strong time derivative $\partial_t g(t)$ exists for a.e. $t$.} the differential inclusion in \eqref{eq:cauchy-strong} pointwise for a.e. $0 < t < T$.

\begin{theorem}\label{thm:L2_representation}
    Let $\F \colon \P_2(\R) \to (-\infty,\infty]$ and let
    $F \colon L_2(0,1) \to (-\infty,\infty]$ be its associated functional as in \eqref{eq:assoc}.
    Let $g_0 \in \mathcal C(0,1)$ be an initial datum. 
    Assume, there exists a strong solution $g \colon [0, \infty) \to \mathcal C(0, 1)$ 
    of the Cauchy problem 
    \begin{align}\label{eq:cauchy}
    \begin{cases}
    \partial_t g(t) + \partial F (g(t)) ~\ni~ 0, \quad t \in (0, \infty), \\
    g(0) ~=~ g_0.
    \end{cases}
    \end{align}
    Then, there exists a Wasserstein gradient flow $\gamma_t$ of $\F$, and it is given by
    $\gamma_t = (g(t))_{\#} \lebesgue_{(0,1)}$ and $\gamma(0+) = (g_0)_{\#} \lebesgue_{(0,1)}$. Furthermore, the curve $\gamma_t$ has quantile functions $Q_{\gamma_t} = g(t)$.
\end{theorem}

With  minor adjustments, the proof can be taken from \cite[Theorem 3.5, part 2]{DSBHS2024}. For completeness, we provide it in the appendix.
   
\section{Distance Kernel MMDs with Sobolev Regularization}\label{sec:mmd} 
In this section, we introduce the MMD functionals for the negative as well as the positive distance kernel and their Sobolev regularized versions $\widetilde{F}_\nu^-, \widetilde{F}_\nu^+$.  
Then, we consider flows of the negative kernel. We state the general existence of the Wasserstein gradient flow via minimizing movements, and derive a more explicit gradient flow equation \eqref{eq:cauchy3}. 
With it, we can give an explicit example of the $\widetilde{F}_\nu^-$-flow, and study the convergence of the gradient flow for vanishing regularization constants $\lambda \downarrow 0$. We also give a result concerning a boosted convergence towards the target $\nu$ in the asymptotic limit $t\to \infty$.
Finally, we show that there exists a so-called generalized minimizing movement in case of the positive kernel, and present an example of an $\widetilde{F}_\nu^+$-flow.

 \subsection{General modeling for both kernels}
We consider  MMD functionals given by
the positive \emph{and} negative distance kernel 
\begin{equation} \label{eq:riesz}
K(x,y) \coloneqq \pm |x-y|, \quad x,y \in \R.
\end{equation} 
The MMD  between $\mu,\nu \in \mathcal P_2(\R)$ is defined by
    \begin{align}         \label{eq:DK2}
        \text{MMD}_K^2(\mu, \nu) &\coloneqq 
        \frac12 \int_{\R^2} K(x,y) \d \left(\mu(x) - \nu(x) \right) \d \left(\mu(y) - \nu(y) \right ).       
\end{align}
For the negative distance kernel, which is \emph{conditionally positive definite}, the square root of the above formula defines a distance on $\P_2(\R)$, see also \cite{NS2023,MD2024} (and \cite{DSBHS2024} and the references therein for the relation to the \emph{energy} and \emph{Cramer distance}), so that
$\text{MMD}_K(\mu, \nu) \ge 0$ 
with equality if and only if $\mu = \nu$.
Fixing the target measure $\nu$ and skipping the arising constant term  in \eqref{eq:DK2}, 
the MMD functionals  $\F_{\nu}^{\pm} \colon \mathcal P_2(\R) \to (-\infty,+\infty]$
are defined according to the distance kernels \eqref{eq:riesz} by
\begin{equation}  \label{eq:dis-decomp_all}
    \F_{\nu}^{\pm}(\mu)    
    \coloneqq \frac12 \int_{\R^2} \pm |x-y| \d \mu(x) \d \mu(y) 
    -
    \int_{\R^2} \pm |x-y| \d \mu(x) \d \nu(y). 
\end{equation}
The first summand is known as \emph{interaction energy}, while
the second one is called \emph{potential energy} of 
$V_\nu \coloneqq \int_{\R} K(\cdot,y) \d \nu(y)$. 
Functionals of this kind were,
e.g., applied in \cite{TSGSW2011}
for function dithering.
In general, MMDs have seen use in statistics and machine learning \cite{mmd_energy_eq, BGRKSS06, GBRSS13}, 
and as loss functions in generative adversarial networks \cite{BSAG2018, DRG2015}, while their gradient flows gained popularity in machine learning \cite{AKSG2019} and image processing \cite{EGNS2021}.
The next result from \cite{DSBHS2024,HBGS2023}
determines  special associated functionals for $\F_{\nu}^{\pm}(\mu)$
with convenient properties. 

\begin{lemma}\label{lem:extended_fun}
Let $\F_\nu^{\pm}$ be given by \eqref{eq:dis-decomp_all}, 
and let  $F_\nu \colon L_2(0,1) \to \R$ be defined by 
\begin{align} \label{eq:Fnu}
    F_{\nu}(u)
    &\coloneqq \int_{0}^{1} \left( (1 - 2 s)  u(s)
    + \int_{0}^{1} | u(s) - Q_{\nu}(t) | \d{t} \right) \d{s}.
\end{align}
Then, it holds for all $u \in L_2(0, 1)$ that
 \begin{align} \label{eq:subdiff}
    \partial F_{\nu}(u) 
    = \big\{ f \in L_2(0,1): 
    f(s) \in 2 \big[ R_{\nu}^-\big(u(s)\big), R_{\nu}^+\big(u(s)\big)\big] - 2s
    ~\text{for a.e.} \; s \in (0, 1) \big\}.
    \end{align} 
For all $\mu\in\P_2(\R)$, we have the association
\begin{equation} \label{def_kern}
    \F_{\nu}^{\pm}(\mu)=  \mp  F_\nu(Q_\mu) \eqqcolon F_\nu^{\pm}(Q_\mu).
\end{equation}
The functional $F_{\nu}^{-} = F_\nu$ corresponding to the negative kernel is convex and continuous.
In particular, we have $\dom(\partial F_\nu^-) = \dom(F_\nu^-) = L_2(0,1)$.
The functional $F_{\nu}^{+} = - F_\nu$ is in general \emph{not} $\lambda$-convex for any $\lambda \in \R$.
\end{lemma}

\begin{proof}
Except for the last part, the result was shown in \cite{DSBHS2024}.
Concerning the last part, 
let $\lambda \in \R$ and consider $\nu \coloneqq \delta_0$, so that $Q_\nu \equiv 0$.
Assume that $F_{\nu}^+$ is $\lambda$-convex. 
    Then, we have for any $u_0, u_1 \in L_2(0, 1)$ that
    \begin{equation*}
         F_{\nu}^+ \left(\frac{u_0 + u_1}{2}\right)
        \le 
				\frac{F_{\nu}^+(u_0) + {F}_{\nu}^+(u_1)}{2}
        - \frac{1}{8} \lambda \| u_0 - u_1 \|_{L_2(0, 1)}^2.
    \end{equation*}
    Since the first summand of $F_{\nu}^+$ is linear and since $Q_\nu \equiv 0$,
    this implies
    \begin{align*}
           \int_{0}^{1} \left|  u_0+  u_1 \right| \d{s}
        & \ge   \int_{0}^{1} | u_0  | \d{s}
        +  \int_{0}^{1} | u_1|  \d{s} 
        +  \frac{1}{4} \lambda \int_{0}^{1} | u_0 - u_1 |^2 \d{s}.
    \end{align*}
     For $\varepsilon > 0$ and $u_0\equiv \varepsilon$, $u_1 \equiv - \varepsilon$ this implies
		$
        0
        \ge 2 + \lambda \varepsilon
        $,
		which is a contradiction for $\varepsilon$ sufficiently small.
		\end{proof}

In order to compensate for the missing convexity of $F_\nu^+$, we propose to  regularize with the function 
$F_H \colon L_2(0,1) \to [0, \infty]$ given by
\begin{equation}\label{eq:sobolev-functional}
    F_H (u) \coloneqq 
    \begin{cases}
        \frac{\lambda}{2}\int_0^1 |u'(s)|^2 \d s, & \text{if } u \in H^1(0,1), \\
        \infty, & \text{else},
    \end{cases}
\end{equation}
with the  (diffusion) constant $\lambda > 0$.
The domain $\dom (F_H) = H^1(0,1) \subset L_2(0,1)$ is the  Sobolev space of square integrable functions $u$, whose weak first derivative $u'$ is also square integrable.
The following lemma  is a classical result of the PDE theory and can be found, e.g., in \cite[Proposition 2.9]{B2010}.
Note that the homogeneous Neumann boundary conditions naturally arise from the definition \eqref{eq:sobolev-functional} of the functional $F_H$ on $H^1(0,1)$.
\begin{lemma}\label{l:laplacian}
 The functional $F_H$ in \eqref{eq:sobolev-functional} is proper, convex and lsc. 
 The subdifferential $\partial F_H(u)$ of $F_H$ at $u \in L_2(0,1)$ consists of all $f \in L_2(0,1)$  such that $u$ is the weak solution of the problem
 \begin{equation}
     \lambda \int_0^1 u'(s) v'(s) \d s = \int_0^1 f(s) v(s) \d s \quad \text{for all } v \in H^1(0,1).
 \end{equation} 
 Further, $\dom (\partial F_H)$ is dense in $L_2(0,1)$ and is given by
 \begin{align}
  \dom (\partial F_H) &= \{ u \in H^1(0,1) : u'' \in L_2(0,1) \text{ and } \, u'(0) = u'(1) = 0 \}   \\
  &= \{ u \in H^2(0,1) : u'(0) = u'(1) = 0 \}
 \end{align}
 where $u''$ denotes the second weak derivative of $u$.
 For $u \in \dom (\partial F_H)$,  it holds 
 $\partial F_H (u) = -\lambda u''$.
\end{lemma}
Regularizing $F_\nu ^-$ directly with the functional $F_H$ leads to a distortion of the target $\nu$. 
Therefore, we consider for $Q_\nu \in H^1(0,1)$ the function  
$F_{H, \nu} \colon L_2(0,1) \to [0,\infty]$ with
\begin{equation}\label{eq:sobolev-functional-target}
    F_{H, \nu} (u ) \coloneqq F_H(u - Q_\nu).
\end{equation}
It is also proper, convex, lsc and
$\dom (F_{H, \nu}) = \dom(F_{H}) = H^1(0,1)$. Further, it holds 
$\partial  F_{H, \nu} (u) = \partial F_H (u - Q_\nu)$ 
for all $u \in H^1(0,1)$ and 
$\dom (\partial F_{H, \nu}) = \dom (\partial F_{H}) + Q_\nu$. 
If  $Q_\nu \in H^2(0,1)$, then we have
$$\dom (\partial F_{H, \nu}) = \{ u \in H^2(0,1) : u'(0) = Q_\nu'(0), ~ u'(1) = Q_\nu'(1)\}.$$
Note how the Neumann boundary conditions are now prescribed by the target $Q_\nu$.
Moreover,
$$\partial F_{H, \nu} (u) = -\lambda u'' + \lambda Q_\nu'' \qquad \text{for all }
 u \in \dom (\partial F_{H, \nu}).$$

To ensure that the gradient flow stays in $\C(0,1)$, we use the indicator function \begin{equation}
    I_{\C(0,1)} (u) \coloneqq
    \begin{cases}
        0, \quad ~\,\text{if } u \in \C(0,1), \\
        \infty, \quad \text{else},
    \end{cases}
\end{equation}
which is proper, convex and lsc, since 
$\C(0,1)$ is nonempty, convex and closed. 
Further, $\dom (\partial I_{\C(0,1)}) = \C(0,1)$, and it holds $0 \in \partial I_{\C(0,1)}(u)$ for all $u \in \C(0,1)$.
Now, we can introduce the  regularized versions $\widetilde{F}_{\nu}^\pm \colon L_2(0,1) \to (-\infty, \infty]$ of $F_\nu^\pm$ given by
\begin{align}\label{eq:regular-1}
   \widetilde{F}_{\nu}^- &\coloneqq F_{\nu}^{-} + F_{H,\nu} + I_{\C(0,1)}, 
\qquad
   \widetilde{F}_{\nu}^+ \coloneqq F_{\nu}^{+} + F_H + I_{\C(0,1)}.
\end{align}
It holds that $\dom(\widetilde{F}_{\nu}^{\pm}) = \C(0,1) \cap H^1(0,1)$. 

\begin{remark}[$\widetilde{F}_{\nu}^\pm$ as associated functions of $\widetilde{\mathcal F}_{\nu}^\pm$]\label{assoc}
   Using Lemma \ref{lem:extended_fun} and the isometry $\mathcal{I}$ from Theorem \ref{prop:Q}, we see that  
	$\widetilde{F}_{\nu}^\pm \colon L_2(0,1) \to (-\infty, \infty]$
	are associated functionals of $\widetilde{\F}_{\nu}^\pm \colon \mathcal P_2(\R) \to (-\infty, \infty]$ given by
\begin{equation}\label{eq:regular-assoc-1}
\widetilde{\F}_{\nu}^-(\mu) \coloneqq \F_{\nu}^{-}(\mu) + F_{H,\nu}(\mathcal{I}(\mu))  
\quad \text{and} \quad
\widetilde{\F}_{\nu}^+(\mu) \coloneqq \F_{\nu}^{+}(\mu) + F_H(\mathcal{I}(\mu)) .
\end{equation}
If $\mu \in \P_2(\R)$ has a density $f_{\mu} \in L^1(\R)$, 
then the squared $H^1$-seminorm of $Q_{\mu}$ becomes
\begin{align*}
    | Q_{\mu} |_{1, 2}^2
        & = \int_{\supp(\mu)} \frac{1}{f_{\mu}(x)} \d{x},
\end{align*}
which delivers the second summand in the above functionals.
Note that this expression would also make sense in higher dimensions which we will consider in our future work.
\end{remark}

\subsection{Flows of the negative kernel}

We consider the function $\widetilde{F}_{\nu}^-$, i.e., $K(x,y) = -|x-y|$.
Let us start with the existence result.

\begin{theorem} \label{main_mmd}
For $Q_\nu \in \C(0,1)\cap H^1(0,1)$, let $\widetilde{F}_{\nu}^-$ be given by \eqref{eq:regular-1}, and suppose an initial datum $g_0 \in \dom(\widetilde{F}_{\nu}^{-}) = \C(0,1) \cap H^1(0,1)$.
Then, there exists a unique strong solution $g \in H^1_{\rm loc}([0,\infty); L_2(0,1))$ of
the Cauchy problem
\begin{align}\label{eq:cauchy2}
    \begin{cases}
        \partial_t g(t) \in -  \partial \widetilde{F}_{\nu}^{-} (g(t)),   \quad  t \in (0, \infty), \\
         g(0) = g_0.
    \end{cases}
\end{align}
The associated curve
$\gamma_t \coloneqq (g(t))_\# \Lambda_{(0,1)}$
is the unique Wasserstein gradient flow of $\widetilde{\F}_{\nu}^{-}$ on $(0,\infty)$ with $\gamma(0+) = (g_0)_{\#} \lebesgue_{(0,1)}$.
Both flows $g$ and $\gamma_t$ are given by minimizing movements $g \in {\rm MM}(\widetilde{F}_\nu^-, g_0)$ and $\gamma \in {\rm MM}(\widetilde{\F}_\nu^-, \gamma(0+))$, respectively. The minimizers \eqref{mms} are unique and the convergence \eqref{eq:mms-convergence} is globally uniform in $t \in [0,\infty)$ with step size $\tau \downarrow 0$.
\end{theorem}

\begin{proof}
Since $\widetilde{F}_{\nu}^-$ and $\widetilde{\F}_{\nu}^-$ are proper, convex, lsc and bounded below, the result follows from combining Remark \ref{assoc}, Theorem \ref{thm:L2_representation} and the standard literature of maximal monotone/subdifferential operators, see, e.g., \cite{B1973}, and also \cite[Theorem 11.2.1]{BookAmGiSa05} for the Wasserstein setting.
\end{proof}

\begin{remark}[Purpose of the $F_{H, \nu}$-regularization]\label{rem:sobolev-role}
  
  To be clear, the functional $F_\nu^- = F_\nu$ associated with the {negative kernel} is already convex (in one dimension), and technically, the existence of strong solutions $g \in {\rm MM}({F}_{\nu}^-, g_0)$ is already known. Nevertheless,  
  the regularization via $F_{H, \nu}$ rectifies a \emph{"dissipation-of-mass" defect} of the $F_\nu$-flow as demonstrated in Section \ref{sec:numerics-examples}. This might  be useful when working on $\P_2(\R^d)$ since there, the non-convexity of the negative kernel MMD paired with this defect\footnote{In the multi-dimensional case, mass might not only dissipate in a bounded region, but even spread to infinity due to the repulsiveness of the negative distance kernel.} 
  seems to complicate the proof of the existence of Wasserstein gradient flows. 
\end{remark}

By the following proposition, under mild assumptions on the target, 
the indicator function $I_{\C(0,1)}$ is actually not needed when considering the gradient flow of $\widetilde{F}_\nu^-$. These assumptions are fulfilled,
e.g., if $Q_\nu \in C^3([0,1])$ and $\lambda > 0$ is sufficiently small; or more generally, if $Q_\nu \in H^3(0,1)$ and 
\begin{equation}
    Q_\nu'''(s) \le \frac{2}{\lambda} \quad \text{for all } s\in (0,1).
\end{equation}
Together with Lemma \ref{lem:extended_fun} and \ref{l:laplacian}, we obtain an \emph{explicit} form of the gradient flow equation \eqref{eq:cauchy2}.

\begin{proposition}\label{prop:reduced-cauchy}
Assume that $Q_\nu \in H^2(0,1)$ and $\lambda > 0$ such that the function
\begin{equation}\label{eq:assumption-quantile}
    (0,1) \to \R, ~~ s\mapsto 2s -\lambda Q_\nu''(s)  \quad \text{is non-decreasing (except for a nullset)}.
\end{equation}
Then, for any $g_0 \in H^1(0,1) \cap \C(0,1)$, the unique strong solution of the Cauchy problem
\begin{align}\label{eq:cauchy3-reduced}
    \begin{cases}
        \partial_t g(t) + \partial F_\nu (g(t)) + \partial F_{H, \nu}(g(t)) \ni 0,   \quad  t \in (0, \infty), \\
         g(0) = g_0,
    \end{cases}
\end{align}
exists, and it coincides with the unique strong solution $g$ of \eqref{eq:cauchy2}. In particular, it satisfies the invariance property $g(t) \in \C(0,1)$ for all $t \ge 0$.
\end{proposition}

\begin{proof}
 Note that by \cite[Theorem 3.16]{P1993}, it holds the equality
 $\partial F_\nu + \partial F_{H, \nu} = \partial(F_\nu + F_{H, \nu})$, 
 since the \emph{Rockafellar interior condition} 
 \begin{equation}
     {\rm int} \big(\dom (F_\nu) \big) \cap \dom (F_{H, \nu}) = L_2(0,1) \cap H^1(0,1) \ne \emptyset
 \end{equation}
 is fulfilled by Lemma \ref{lem:extended_fun}. Now, $F_\nu + F_{H, \nu}$ is proper, convex and lsc with $\dom(F_\nu + F_{H, \nu}) = H^1(0,1)$. Then, by standard results about subdifferential operators, see \cite[Theorem 3.6]{B1973}, there exists a unique strong solution $g$ of \eqref{eq:cauchy3-reduced}. 
 
 To show that $g$ also solves \eqref{eq:cauchy2}, we only need to prove that $g(t) \in \C(0,1)$ for all $t \ge 0$, since this implies
 \begin{align}
 \partial F_\nu (g(t)) + \partial F_{H, \nu}(g(t)) + \{0\}  &\subseteq  \partial F_\nu (g(t)) + \partial F_{H, \nu}(g(t)) +  \partial I_{\C(0,1)} (g(t))\\
 &\subseteq \partial (F_\nu + F_{H, \nu} + I_{\C(0,1)}) (g(t))
 = \partial \widetilde{F}_{\nu}^{-} (g(t)).
 \end{align}
 To this end, let $\tau > 0$ and consider the implicit Euler scheme with step size $\tau > 0$  of \eqref{eq:cauchy3-reduced} given by
 \begin{equation}\label{eq:implicit-Euler-negative}
  g_{n+1}(s) + 2 \tau [R_\nu^-(g_{n+1}(s)), R_\nu^+(g_{n+1}(s))] - 2\tau s - \tau \lambda g_{n+1}''(s) + \tau \lambda Q_{\nu}'' (s) \ni g_n(s) \quad \text{for a.e. } s \in (0,1),
 \end{equation}
 where the existence of $g_{n+1} \in \dom(\partial F_{H,\nu})$ is obtained via the maximal monotonicity of $\partial F_{H,\nu} + \partial F_{\nu}$. 
 We need to show that $g_n \in \C(0,1)$ implies $g_{n+1} \in \C(0,1)$. Then, the assertion follows by the arguments in \cite[Corollary 3.6]{DSBHS2024}.

 So let $g_n \in \C(0,1)$, and assume that there exists $s_0 \in (0,1)$ such that $g_{n+1}'(s_0) < 0$. Note that by assumption, $Q_\nu \in H^2(0,1)$ is continuously differentiable on $[0,1]$. Since $Q_\nu$ is increasing, it must hold $Q_\nu'(0),\, Q_\nu'(1) \ge 0$. 
 Now, since $g_{n+1} \in \dom(\partial F_{H,\nu}) = \{ u \in H^2(0,1) : u'(0) = Q_\nu'(0), ~ u'(1) = Q_\nu'(1)\}$, 
 we have that $g_{n+1}$ is continuously differentiable with $g_{n+1}'(s_0) < 0 \le \min\{g_{n+1}'(0), \,g_{n+1}'(1) \}$. Hence, by the continuity of $g_{n+1}'$, there exists an interval $\emptyset \ne (a,b) \subset (0,1)$ around $s_0$ such that $g_{n+1}'(s) < 0$ for all $s \in (a,b)$, and $g_{n+1}'(a) = g_{n+1}'(b) = 0$. In particular, $g_{n+1}$ is strictly decreasing on $(a,b)$. Now, it holds
 \begin{equation}
 - \tau \lambda g_{n+1}''(s) 
  \in g_n(s) - g_{n+1}(s) + 2 \tau [-R_\nu^+(g_{n+1}(s)), -R_\nu^-(g_{n+1}(s))]
 + \tau (2 s -  \lambda Q_{\nu}'' (s))  
 \end{equation}
 for a.e. $s \in (a,b)$. 
 Since $g_n$ is increasing, $g_{n+1}$ is strictly decreasing\footnote{Note that the \emph{strict} monotonicity is needed to estimate $-R_\nu^-(g_{n+1}(s_1))$ by $-R_\nu^+(g_{n+1}(s_2))$.} and by  assumption \eqref{eq:assumption-quantile},
 we have for a.e. $s_1, s_2 \in (a,b)$ with $s_1 < s_2$ that
 \begin{align}
- \tau \lambda g_{n+1}''(s_1)
&\le g_n(s_1) - g_{n+1}(s_1) - 2 \tau R_\nu^-(g_{n+1}(s_1))
 + \tau (2 s_1 -  \lambda Q_{\nu}'' (s_1)) \\
&< g_n(s_2) - g_{n+1}(s_2) - 2 \tau R_\nu^+(g_{n+1}(s_2))
 + \tau (2 s_2 -  \lambda Q_{\nu}'' (s_2))   \\
&\le  - \tau \lambda g_{n+1}''(s_2).
 \end{align}
  This means that, outside a nullset, $-g_{n+1}''$ is increasing on $(a,b)$, or $g_{n+1}''$ is \emph{decreasing} on $(a,b)$. Finally, note that $g_{n+1}'(a) = 0 > g_{n+1}'(s_0)$, and by the absolute continuity of $g_{n+1}'$, there must exist $A_1 \subseteq (a, s_0)$ with positive Lebesgue measure such that $g_{n+1}''(\xi_1) < 0$ for all $\xi_1 \in A_1$. 
 At the same time, since $g_{n+1}'(s_0) < 0 = g_{n+1}'(b)$, there must exist $A_2 \subseteq (s_0, b)$ with positive Lebesgue mass such that $g_{n+1}''(\xi_2) > 0$ for all $\xi_2 \in A_2$. Hence, $g_{n+1}''(\xi_1) < 0 < g_{n+1}''(\xi_2)$ for all $\xi_1 \in A_1, \, \xi_2 \in A_2$, but $g_{n+1}''$ is \emph{decreasing} on $(a,b)$ outside a nullset. 
 
 This contradiction implies $g_{n+1}'(s) \ge 0$ for all $s \in (0,1)$, from which we infer that $g_{n+1}$ is increasing on $(0,1)$, i.e., $g_{n+1} \in \C(0,1)$, and we are done.
\end{proof}

With  Proposition \ref{prop:reduced-cauchy}, we can study the $\widetilde{F}_\nu^-$-gradient flow between two Dirac measures.

\begin{example}[Flow of $\widetilde{F}_\nu^-$ from $\delta_{-1}$ to $\delta_{0}$]\label{ex:dirac-to-dirac-reg}
  We consider as target the Dirac measure $\nu \coloneqq \delta_0$, and as initial measure $\gamma_0 \coloneqq \delta_{-1}$. 
 The corresponding quantiles are $Q_\nu \equiv 0$ and $g_0 \equiv -1$. Let $g(t,s) \coloneqq \left(g(t) \right)(s)$. 
 Then, by Proposition \ref{prop:reduced-cauchy} and Lemma \ref{lem:extended_fun}, for any $\lambda > 0$
 the $ \widetilde{F}_{\nu}^{-}$-flow is given by the differential inclusion
  \begin{align}\label{eq:dirac-to-dirac-inclusion}
        \partial_t g(t,s)  + \partial F_{H, \nu}(g(t))(s) \in 2s - 2 H(g(t,s)),   \quad  \text{for a.e. } t \in (0, \infty), \, \text{a.e. } s\in (0,1),
\end{align}
with  the set-valued Heaviside function
$$
H(x) \coloneqq \begin{cases}
    0, & x < 0,\\
    1, & x > 0,\\
    [0,1], & x = 0.
\end{cases}
$$
Hence, for all $0 < t < t^*$, where $t^* > 0$ is the first point in time when $g(t^*, \cdot)$ touches the constant $Q_\nu \equiv 0$, the inclusion \eqref{eq:dirac-to-dirac-inclusion} becomes the heat equation with homogeneous Neumann boundary conditions and inhomogeneity $2s$, i.e.,
\begin{align}\label{eq:dirac-to-dirac-heat}
    {\rm (HEAT)} \begin{cases}
        &\partial_t g(t,s)  - \lambda \partial^2_{ss} g(t,s) = 2s,   \quad  t \in (0, t^*), ~ s\in (0,1),\\
        &\partial_s g(t,0) = \partial_s g(t,1) = 0, \qquad  t \in (0, t^*),\\
        &g(0, \cdot) \equiv -1.
    \end{cases}    
\end{align}
Using Fourier series, the solution to \eqref{eq:dirac-to-dirac-heat} can be explicitly computed: inserting the Ansatz 
\begin{equation}
    u(t,s) = \frac{c_0(t)}{2} + \sum_{n=1}^\infty c_n(t) \cos(n \pi s)
\end{equation}
into \eqref{eq:dirac-to-dirac-heat} yields
\begin{equation}
    \frac{c_0'(t)}{2} + \sum_{n=1}^\infty [c_n'(t) + \lambda (n\pi)^2 c_n(t)]\cos(n \pi s) = 2s \eqqcolon 
    \frac{b_0}{2} + \sum_{n=1}^\infty b_n \cos(n \pi s),
\end{equation}
where the right-hand side defines the Fourier series of $f(s) = 2s$ evenly extended to $(-1, 1)$. Hence, 
\begin{equation}
    c_0'(t) = b_0 := 2 \int_0^1 f(s) \d s = 2,
\end{equation}
and for $n \ge 1$, we obtain
\begin{equation}
    c_n'(t) + \lambda (n\pi)^2 c_n(t) = b_n := 2 \int_0^1 f(s) \cos(n \pi s) \d s = \frac{4(\cos(n\pi) -1)}{(n\pi)^2}.
\end{equation}
Using the initial condition $u(0, \cdot) \equiv -1$, we get
\begin{equation}
    c_0(t) = 2(t - 1)
    \quad \text{and} \quad 
    c_n(t) = \frac{b_n }{\lambda (n\pi)^2} (1 - e^{-\lambda (n\pi)^2 t}).
\end{equation}
Altogether, the explicit solution of \eqref{eq:dirac-to-dirac-heat} up to the time $t < t^*$ reads as
\begin{equation}\label{eq:dirac-to-dirac-explicit}
    u(t,s) = (t - 1) + \frac{1}{\lambda}
    \sum_{n=1, n \,{\rm odd}}^\infty (-8) (n\pi)^{-4}
    (1 - e^{-\lambda (n\pi)^2 t}) \cos(n \pi s), \quad t \in (0, t^*), \, s\in (0,1). 
\end{equation}
For $t \ge t^*$, the solution is given by the following \emph{free boundary problem}, or \emph{Stefan problem},
\begin{align}\label{eq:dirac-to-dirac-stefan}
    {\rm (STEFAN)} \begin{cases}
        &\partial_t g(t,s)  - \lambda \partial^2_{ss} g(t,s) = 2s,   \quad  t \in (t^*, \infty), \,  s\in (0, \beta(t)),\\
        &\partial_s g(t,0) = \partial_s g(t, \beta(t)) = 0, \quad  t \in (t^*, \infty),\\
        &g(t, s) = 0, \quad  s \in [\beta(t), 1) , ~t \in (t^*, \infty),\\
        &g(t^*, \cdot) = u(t^*, \cdot),
    \end{cases}    
\end{align}
where the right moving boundary function $\beta \colon [t^*, \infty) \to (0, 1]$ with $\beta(t^*) = 1$ is part of the solution $(g, \beta)$. Since this problem is more involved, and no explicit representation via Fourier series  can be expected, we will approximate the solution in the numerical Section \ref{sec:numerics-examples}. \hfill $\Box$
\end{example}

Let us quickly examine the behavior of the flow \eqref{eq:cauchy3-reduced} of the {negative} kernel for vanishing diffusion constants $\lambda$, see also the Figure \ref{fig:Uniform_to_Uniform_small} in the appendix.
\begin{remark}[Convergence for $\lambda \downarrow 0$]\label{rem:semigroup-convergence}
Let us explicitly denote the dependence on the diffusion constant $\lambda > 0$ by ${F}_{\nu}^{\lambda} = F_{H,\nu}^{\lambda} + F_\nu$, where again, the functional $F_{H,\nu}^{\lambda} \colon H^1(0,1) \subset L_2(0,1) \to [0,\infty]$ is given by
\begin{equation}
 F_{H,\nu}^{\lambda}(w) =  
 \begin{cases}
        \displaystyle\frac{\lambda}{2}\int_0^1 |w'(s) - Q_\nu'(s)|^2 \d s, & \text{if } w \in H^1(0,1), \\
        \infty, & \text{else}.
    \end{cases}
\end{equation}
Let $(\lambda_n)$ be an arbitrary zero sequence of positive numbers, then the Mosco convergence of $F_{\nu}^{\lambda_n}$ to $F_\nu$ for $n \to \infty$ is immediate: Let $u \in L_2(0,1)$, then for any sequence $u_n \rightharpoonup u$ in $L_2(0,1)$, it clearly holds by the nonnegativity of $F_{H, \nu}^{\lambda_n}$ and weak lower semicontinuity of $F_\nu$ that
\begin{equation}
 F_\nu(u) \le  \liminf_{n \to \infty} F_{H,\nu}^{\lambda_n}(u_n) + F_\nu(u_n) = \liminf_{n \to \infty} {F}_{\nu}^{\lambda_n} (u_n).
\end{equation}
Lastly, since $H^1(0,1)$ is dense in $L_2(0,1)$, there exists a sequence $v_n \in H^1(0,1)$ with $v_n \to u$ in $L_2(0,1)$ such that $\lambda_n \, |v_n|_{H^1} \to 0$, and hence by the continuity of $F_\nu$,
\begin{equation}
 F_\nu(u)=  \lim_{n \to \infty} F_{H,\nu}^{\lambda_n}(v_n) + F_\nu(v_n) = \lim_{n \to \infty} {F}_{\nu}^{\lambda_n} (v_n).
\end{equation}
Now, by classical convergence results of semigroup theory, see \cite{BP1972, A1979}, and also \cite{BR2020} for an extensive presentation, it holds the convergence
\begin{equation}
    \lim_{n \to \infty} \sup_{t \in [0,T]} \| g_{\lambda_n}(t) - g(t) \|_{L_2(0,1)} = 0,
\end{equation}
where $g_{\lambda_n}, g$ are the unique solutions of \eqref{eq:cauchy3-reduced} with respect to diffusion constants $\lambda_n$ and $\lambda = 0$, respectively.
\hfill $\Box$
\end{remark}

Finally, we want to remark how the addition of the Sobolev term $F_{H,\nu}$ boosts the convergence to the target $\nu$ for $t \to \infty$, see also the Figure \ref{fig:Uniform_to_Uniform_large} in the appendix.
\begin{remark}[Long-time asymptotics for $t \uparrow \infty$]\label{rem:long-time}
Consider the sum of the original MMD functional \eqref{eq:DK2} with the Sobolev term \eqref{eq:sobolev-functional-target}, i.e., 
$\phi \coloneqq \text{MMD}_K^2(\cdot, \nu) + F_{H,\nu}^{\lambda}(\mathcal{I}(\cdot))$, where $\lambda=1$ for simplicity. By \cite[Eq. 11.2.6]{BookAmGiSa05}, it holds for the Wasserstein gradient flow $\mu_t$ of $\phi$ starting in $\mu_0$ that (notice $\phi(\nu) = 0$)
    $$
    \text{MMD}_K^2(\mu(t), \nu) + |Q_{\mu(t)} - Q_{\nu}|_{H^1(0,1)}^2
    = \phi({\mu(t)}) \le \frac{W_2^2(\mu_0, \nu)}{2t}, \quad \text{for all } t\in (0,\infty),
    $$
    showing the bound $\text{MMD}_K^2(\mu(t), \nu) \le \frac{W_2^2(\mu_0, \nu)}{2t} - |Q_{\mu(t)} - Q_{\nu}|_{H^1(0,1)}^2$ instead of only $\text{MMD}_K^2(\mu(t), \nu) \le \frac{W_2^2(\mu_0, \nu)}{2t}$.
\end{remark}

\subsection{Flows of the positive kernel}
Since for the positive kernel, the functional $F_\nu^+$ is in general not $\lambda$-convex due to Lemma \ref{lem:extended_fun}, its treatment is more involved, and in general, the existence of a Wasserstein gradient flow cannot be expected. Yet, the proposed regularization allows for a description of the flow via 
the generalized minimizing movement (GMM) scheme, see also \cite{RS2006}: 

Given $g_0 \in L_2(0,1)$ and $F \colon L_2(0,1) \to (-\infty, \infty]$, consider piecewise constant curves $g_\tau$ 
constructed from the (possibly non-unique) minimizers
    \begin{equation} \label{gmm}
        g_{n+1} \in \argmin_{v \in L_2(0,1)} \Big\{ F(v) + \frac{1}{2 \tau} \|g_n - v\|_{L_2(0,1)}^2 \Big\}, \quad \tau > 0,
    \end{equation}
    via $g_\tau|_{(n \tau,(n+1)\tau]} \coloneqq g_n, ~n \in \N_0$. 
    Fixing $T > 0$, we say that $g$ is a \emph{Generalized Minimizing Movement},
    if there \emph{exists} a subsequence of step sizes $\tau_k \downarrow 0$ and a corresponding family of discrete curves $(g_{\tau_k})$ such that
    \begin{equation}\label{eq:gmm-convergence}
       \lim_{k \to \infty} g_{\tau_k}(t) = g(t) \quad \text{for all } t \in [0,T],
    \end{equation}
    where the limit is taken in $L_2(0,1)$,
    and we write $g \in {\rm GMM}(F, g_0)$.

With this concept, we can describe a flow of $\widetilde{F}_\nu^+$ via generalized minimizing movements.  We will make use of the results in \cite{RS2006}. For convenience we provide them in the appendix.

\begin{theorem} \label{main_mmd2}
For $Q_\nu \in \C(0,1)$, let $\widetilde{F}_{\nu}^+$ be given by \eqref{eq:regular-1}. Suppose that the initial datum fulfills $g_0 \in \dom(\widetilde{F}_{\nu}^{+}) = \C(0,1) \cap H^1(0,1)$.
Then, for any $0 < T < \infty$, there exist functions $g \in {\rm GMM}(\widetilde{F}_{\nu}^{+}, g_0)$,
and they are strong solutions $g \in H^1(0,T; L_2(0,1))$ of
the Cauchy problem
\begin{align}\label{eq:cauchy3}
    \begin{cases}
        \partial_t g(t) \in -  \partial_l \widetilde{F}_{\nu}^{+} (g(t)),   \quad  t \in (0, T), \\
         g(0) = g_0.
    \end{cases}
\end{align}
Here, $\partial_l \widetilde{F}_{\nu}^{+}$ denotes the limiting subdifferential, see \cite{RS2006}.
The associated curve
$\gamma_t \coloneqq (g(t))_\# \Lambda_{(0,1)}$
is a Wasserstein absolutely continuous flow on $(0,T)$ with $\gamma(0+) = (g_0)_{\#} \lebesgue_{(0,1)}$, and it is given by $\gamma \in {\rm GMM}(\widetilde{\F}_\nu^+, \gamma(0+))$.
\end{theorem}

\begin{proof}
    Note that $\widetilde{F}_\nu^+$ and $g_0$ satisfy the compactness criterion of \cite[Lemma 1.2]{RS2006} by the compact Sobolev embedding $H^1(0,1) \overset{c}{\hookrightarrow} L_2(0,1)$. Furthermore, since the functionals $F_H + F_{\C(0,1)}$ and $F_\nu$ are convex and lsc, and since $\dom(\partial F_\nu) = L_2(0,1)$ by Lemma \ref{lem:extended_fun}, the conditions of \cite[Remark 1.9]{RS2006} are fulfilled. Therefore, \cite[Theorem 3]{RS2006} yields the existence of a strong solution $g \in H^1(0,T; L_2(0,1))$ of \eqref{eq:cauchy3} via {generalized minimizing movements} $g \in {\rm GMM}(\widetilde{F}_{\nu}^{+}, g_0)$ for any $0 < T < \infty$.
    The remaining part follows directly by using the isometry of Theorem \ref{prop:Q}, and Remark \ref{assoc}.
\end{proof}

\begin{remark}[Purpose of the $F_H$-regularization]\label{rem:sobolev-role2}
  Since the associated functional $F_\nu^+ = -F_\nu$ may lack any form of $\lambda$-convexity by Lemma \ref{lem:extended_fun}, the existence of generalized minimizing movements $g \in {\rm GMM}(F_\nu^+, g_0)$ might not be given \emph{in general}.
  Although it can be shown that the minimizers in \eqref{gmm} exist, the convergence \eqref{eq:gmm-convergence} of the discrete solutions is unclear in general.  
  On a technical level,
  regularizing via the Sobolev norm $F_H$ gives control over the convergence \eqref{eq:gmm-convergence}, ensuring that $g \in {\rm GMM}(\widetilde{F}_{\nu}^{+}, g_0)$ always exists.
\end{remark}

As stated in Remark \ref{rem:sobolev-role2} the regularization via $F_H$ plays more of a technical role ensuring general existence of generalized minimizing movements.
Nevertheless, in the following example it is not needed, and we can explicitly calculate $g \in {\rm GMM}(\widetilde{F}_{\nu}^{+}, g_0)$. 
Note that this together with Example \ref{ex:dirac-to-dirac-reg} shows that $\widetilde{F}_{\nu}^{-}$-flows are \emph{not} the time reversal of $\widetilde{F}_{\nu}^{+}$-flows.

\begin{example}[Flow of $\widetilde F_\nu^+$ from $\delta_{-1}$ away from  $\delta_0$]\label{ex:dirac-from-dirac}
We consider again as target (or more precisely, \emph{repulsion point}) the Dirac measure $\nu \coloneqq \delta_0$, and as initial measure $\gamma_0 \coloneqq \delta_{-1}$. 
The corresponding quantiles are $Q_\nu \equiv 0$ and $g_0 \equiv -1$.

Let $\tau>0$ be a positive step size and $g<0$ some constant function, where we denote both the function and the constant by $g$.
The minimizer \eqref{gmm} of one step of the GMM scheme is given by
\begin{align}
\argmin _{f\in L_2(0,1)} &\left\{ \widetilde F_\nu^+(f)+\frac{1}{2\tau} \int _0^1 |g-f(x)|^2 \d x \right\}
\\
= \argmin_{f\in \mathcal C(0,1)\cap H^1(0,1)} &\Big\{ -\int_0^1\left( (1-2s)f(s)+|f(s)| \right) \d s +\frac{\lambda}{2}  \int_0^1 |f'(x)|^2\d x \\
 + &\frac{1}{2\tau} \int_0^1 |g-f(x)|^2\d x \Big\}.
\label{eq:minPosR}
\end{align}
Next we show that the minimizer in \eqref{eq:minPosR} is a constant function. 
Let $f\in \mathcal C(0,1)\cap H^1(0,1)$ and set $\bar f \coloneqq -\int _0^1 |f(s)|\d s\le 0$. 
We consider the summands of \eqref{eq:minPosR} separately.
Obviously, the Sobolev term favors the constant function $\bar f$, because 
\begin{align*}
    \int_0^1 |\bar f'|^2\d\lambda =0 \le \int_0^1 |f'|^2\d\lambda,
\end{align*}
where the inequality is strict if $f$ is not constant.
Further, we have for $g < 0$ that
\begin{align*}
    \int_0^1 |g-\bar f|^2 \d x &
    = |g-\bar f|^2
    = \left( \int_0^1 g+ |f(x)|\d x\right)^2
    \le \left( \int_0^1 \bigl|g+|f(x)|\bigr|\d x\right)^2
    \\&
    = \left( \int_0^1  \bigl||f(x)|-|g|\bigr|\d x\right)^2
    \le  \left(\int_0^1 |f(x)-g|\d x\right)^2
    \le \int_0^1 |g-f(x)|^2 \d x.
\end{align*}
Finally, 
\begin{align*}
- \int_0^1 |f(s)|\d s =  \bar f = - \int _0^1 |\bar f|\d s,
\end{align*}
and for the remaining summand we have by monotonicity of $f$ that 
\begin{align*}
    -\int_0^1 (1-2s)f(s)\d s&
    =-\int_0^1 (1-2s) \int_0^s f'(t)\d t \d s
    =-\int_0^1 f'(t)\int_t^1 (1-2s)\d s \d t\\&
    =-\int_0 ^1 f'(t) [s-s^2]_t^1\d t
     =\int_0^1 f'(t)t(1-t)\d t\ge 0.
\end{align*}
Plugging the constant function $\bar f$ in the above equation, we obtain 
\begin{align*}
    -\int_0^1 (1-2s)\bar f(s)\d s = \int_0^1 0\cdot t(1-t)\d t = 0.
\end{align*}
In summary, the minimizer must be a constant function $f$ with value $f\in \R$ and the minimization problem \eqref{eq:minPosR} reduces 
for $g <0$ to \begin{align*}
    \argmin_{f\in \R} \left\{ -|f|+\frac{1}{2\tau} |g-f|^2\right\}
    =\argmin_{f\le 0} \left\{ f+\frac{1}{2\tau} (g-f)^2\right\}.
\end{align*}
The unique minimizer is given by $f = g-\tau < 0$, so that for the next GMM step the same argument can be applied. 
Hence, the generalized minimizing movement scheme with initial $g_0 \equiv -1$ is given by $g_n \equiv -1-n\tau$, and
the limiting flow for $\tau \to 0$ is just
\begin{equation}
    g(t, s) = -1 - t,
\end{equation}
corresponding to a Dirac measure $\mu_t =\delta_{-1-t}$ which moves away from the target with constant speed. \hfill $\Box$
\end{example}

We end this section with a remark concerning more general initial datums $\gamma_0$ and targets $\nu$.
\begin{remark}[Relaxing the limitation of $H^1(0,1)$]\label{rem:relax}
 Note that the assumption $g_0 \in H^1(0,1)$,
 and in case of the negative kernel, also the assumption $Q_\nu \in H^1(0,1)$ are integrated into our approach. In particular, the initial quantile $g_0$ is assumed to be absolutely continuous on $[0,1]$, which implies that the initial measure $\gamma_0$ shall have \emph{compact and convex} support. This can be drastically relaxed if we, instead of working with $F_H$, consider fractional derivatives given by the squared \emph{Sobolev-Slobodeckij} seminorm
 \begin{equation}
     F_H^\sigma (u) \coloneqq
     \begin{cases}
        \displaystyle\frac{\lambda}{2}\int_0^1 \int_0^1 \frac{|u(s_1)-u(s_2)|^2}{|s_1 - s_2|^{d+2\sigma}} \d s_1 \d s_2, & \text{if } u \in H^\sigma(0,1), \\
        \infty, & \text{else},
    \end{cases}
 \end{equation}
defined on the \emph{fractional} Sobolev space $H^\sigma(0,1) \subset L_2(0,1)$ for $0 < \sigma < 1$. Note that we still have the compact embedding $H^\sigma(0,1) \overset{c}{\hookrightarrow} L_2(0,1)$ for any $\sigma \in (0,1)$, and our Theorems \ref{main_mmd} and \ref{main_mmd2} are still valid with the suitable adjustments. Importantly, choosing $\sigma < \frac{1}{2}$ small enough, we can allow discontinuous jump functions $g_0, Q_\nu \in H^\sigma(0,1)$, which further may admit singularities at the boundary of $(0,1)$. This translates to measures $\gamma_0, \nu$ possibly having \emph{disconnected and unbounded} support. \hfill $\Box$
\end{remark}

\section{Numerics}\label{sec:numerics}

In this section, we numerically explore the Wasserstein gradient flow of $\widetilde{\F}_{\nu}^{-}$ via the flow of quantile functions with respect to the associated functional $\widetilde{F}_{\nu}^{-}$. We derive the associated implicit Euler scheme, and demonstrate a crucial advantage over the unregularized ${\F}_{\nu}^{-}$-flow.

\subsection{The Implicit Euler Scheme}\label{sec:numerics-Euler}
Let $g_0 \in H^1(0,1)\cap \C(0,1)$ and assume that the condition \eqref{eq:assumption-quantile} is fulfilled. Then, by Proposition \ref{prop:reduced-cauchy}, the flow $g$ of $\widetilde{F}_{\nu}^{-}$  coincides with the flow of 
$F_\nu + F_{H, \nu}$ given by \eqref{eq:cauchy3-reduced}, and it holds 
\begin{equation}
  \partial(F_\nu + F_{H, \nu}) = \partial F_\nu + \partial F_{H, \nu}. 
\end{equation}
The associated MM scheme \eqref{gmm} is then given by
\begin{equation} \label{eq:gmm-reduced}
        g_{n+1}^\tau = \argmin_{v \in L_2(0,1)} \Big\{ F_\nu(v) + F_{H, \nu}(v) + \frac{1}{2 \tau} \|v - g_n^\tau\|_{L_2(0,1)}^2 \Big\}, \quad \tau > 0.
    \end{equation}
The Euler equation of \eqref{eq:gmm-reduced} leads to the implicit Euler scheme 
\begin{equation}
   \frac{g_{n+1}^\tau - g_{n}^\tau}{\tau} + \partial F_\nu(g_{n+1}^\tau) + \partial F_{H, \nu}(g_{n+1}^\tau) \ni 0.
\end{equation}
By Lemma \ref{lem:extended_fun}
and \ref{l:laplacian}, given $g_n \coloneqq g_n^\tau \in H^1(0,1)\cap \C(0,1)$, we need to solve
\begin{equation}\label{eq:implicit-Euler-num}
  g_{n+1}(s) + 2 \tau [R_\nu^-(g_{n+1}(s)), R_\nu^+(g_{n+1}(s))] - 2\tau s - \tau \lambda g_{n+1}''(s) + \tau \lambda Q_{\nu}'' (s) \ni g_n(s), \quad  s \in (0,1),
 \end{equation}
 for $g_{n+1} \in \dom(\partial F_{H,\nu}) = \{ u \in H^2(0,1) : u'(0) = Q_\nu'(0), ~ u'(1) = Q_\nu'(1)\}$.  The resulting piecewise constant curve $g_\tau|_{(n \tau,(n+1)\tau]} \coloneqq g_n^\tau, ~n \in \N_0$, satisfies 
 \begin{equation}
       \lim_{\tau \to 0} g_{\tau}(t) = g(t) \quad \text{for all } t \in [0,\infty),
    \end{equation}
 with uniform  convergence in $\tau$ since $g_0 \in \dom(F_{H,\nu})$, see \cite[Eq.(4.0.6)]{BookAmGiSa05}.
We need to solve the following second-order differential inclusion Neumann problem:
\begin{align}\label{eq:neumann-problem}
{\rm (NP)}
\begin{cases}
 & -\tau \lambda g_{n+1}''(s) 
  ~\in~ g_n(s) - g_{n+1}(s)+ 2 \tau [-R_\nu^+(g_{n+1}(s)), -R_\nu^-(g_{n+1}(s))]\\
 &\hspace{23mm}  + \,\tau (2 s -  \lambda Q_{\nu}'' (s)) \hspace{44mm} \text{on } s\in (0,1), \\
 &\hspace{2mm} g_{n+1}'(0) = Q_\nu'(0), \quad g_{n+1}'(1) = Q_\nu'(1) \hspace{36mm} \text{on } s \in \partial(0,1).   
\end{cases}
 \end{align}

\paragraph{Technical implementation of solving the Neumann problem \eqref{eq:neumann-problem}:} 
To solve the second-order differential equation \eqref{eq:neumann-problem}, we  employ the SciPy solver \texttt{scipy.integrate.solve\_bvp}. Since this solver can only solve boundary value problems for first-order differential equations with a single-valued right-hand side, we transform \eqref{eq:neumann-problem} into a system of first order problems as follows:
we substitute \(y_0 = g_{n+1}\) and \(y_1 = g_{n+1}'\) and define
\begin{align}\label{eq:dirichlet-problem}
{\rm (NP')}\begin{cases}
y_0'(s) &= y_1(s) \hspace{56.1mm}\text{on } s \in (0,1),\\
y_1'(s) &= \big( y_0(s) + 2\tau R_\nu(y_0(s)) - (g_n(s) - \lambda \tau Q_\nu''(s) + 2\tau s) \big)/(\lambda \tau),\\
y_1(0) &= Q_\nu'(0), \quad y_1(1) = Q_\nu'(1).    
\end{cases}
\end{align}
Note that if the transformed problem \eqref{eq:dirichlet-problem} has a solution \(y\colon (0,1) \to \mathbb{R}^2\), then \(y_0\colon (0,1) \to \mathbb{R}\) is a solution of the Neumann problem \eqref{eq:neumann-problem}. Technically, \eqref{eq:neumann-problem} and \eqref{eq:dirichlet-problem} are only equivalent in the case of a continuous CDF $R_\nu$,
because \eqref{eq:dirichlet-problem} only considers the single-valued right-hand side. \\
For the numerical computation of the flow, let \(\tau > 0\) be the step size and \(\lambda > 0\) be the regularization parameter. Given \(R_\nu\) and the Laplacian \(Q_\nu''\) of the target measure \(\nu\) together with its derivative at the boundary \(\{Q_\nu'(0), Q_\nu'(1)\}\), we compute the discrete flow \(\gamma_n\) from an initial measure \(\gamma_0\) given in the form of its quantiles \(g_n\).\\

\RestyleAlgo{ruled}

\begin{algorithm}[H]
\caption{Sobolev regularized Riesz MMD Flow}
\label{alg:unregularizedEuler}
\SetKwFunction{RHS}{RHS} 
\SetKwFunction{bvp}{solve\_bvp} 

\KwData{$R_\nu,\, Q_\nu'',\, [Q_\nu'(0),\, Q_\nu'(1)], \, g_0,\, \lambda,\, \tau,\, N$}
\KwResult{$g_N,\, g_N'$}

\For{$n=0,\ldots, N-1$}{ \tcp{Define right hand side of \eqref{eq:dirichlet-problem}}
\SetKwProg{Fn}{Function}{:}{end} 
\Fn{\RHS{$s,y$}}{
    $z_0 \gets y_1(s)$\;
    $z_1 \gets  \left( y_0(s) +2\tau R_\nu(y_0(s))- (g_n(s)-\lambda \tau Q_\nu''(s)+2\tau s) \right)/(\lambda \tau)$\;
    \Return $[z_0,z_1]$\; 
}

\tcp{Solve boundary value problem \eqref{eq:dirichlet-problem} with \bvp}
$g_n, g_n'\gets \bvp{$\RHS, [Q_\nu'(0), Q_\nu'(1)]$}$

}
\tcp{Return quantile and its derivative after $N$ steps at time $t= \tau \cdot N$.}
\Return{$g_N, \, g_N'$}\; 

\end{algorithm}

\newpage
\subsection{Numerical Examples}\label{sec:numerics-examples}
Based on the MM scheme \eqref{eq:gmm-reduced} and the resulting Neumann problem \eqref{eq:neumann-problem}, we calculate and visualize the $\widetilde{F}_{\nu}^{-}$-flow for several examples. In order to demonstrate the advantage of the regularization via $F_{H,\nu}$, we also plot the flows of the unregularized functional $F_\nu^-$, showing a "dissipation-of-mass" defect.

\paragraph{Dirac-to-Dirac:} Consider as target the Dirac measure $\nu \coloneqq \delta_0$, and as initial measure the Dirac point $\gamma_0 \coloneqq \delta_{-1}$. The corresponding quantiles satisfy $Q_\nu \equiv 0$ and $Q_0 \equiv -1$. Hence, the assumptions of Proposition \ref{prop:reduced-cauchy} are fulfilled for any $\lambda > 0$.\\
In Figure \ref{fig:Dirac_to_Dirac_unreg}, the unregularized $\F_\nu^-$-flow is depicted. We can observe that the support of the flow grows monotonically with time $t$, which was proven in \cite[Theorem 6.11]{DSBHS2024}. As a consequence, the mass expands and can never totally vanish once it touches a point. Instead, it slowly dissipates and becomes arbitrarily slim on $(-1,0)$ for $t \to \infty$.
In Figure \ref{fig:Dirac_to_Dirac_quantiles}, this can be seen as the corresponding quantiles of the $F_\nu^-$-flow become \emph{arbitrarily steep at $s = 0$}.
\\
On the contrary, the support of the regularized $\widetilde{\F}_\nu^-$-flow can actually \emph{move towards the target} as seen in Figure \ref{fig:Dirac_to_Dirac_reg}. Their quantile functions do flatten out and \emph{smoothly approximate} the target $Q_\nu \equiv 0$ as shown in Figure \ref{fig:Dirac_to_Dirac_quantiles}. Note that the Neumann boundary values $Q_\nu'(0), \, Q_\nu'(1)$ of the problem \eqref{eq:neumann-problem} correspond to the fact that the regularized $\widetilde{\F}_\nu^-$-flow immediately develops "horn-shaped barricades" at the boundary of its support with the height $Q_\nu'(0)^{-1}, \, Q_\nu'(1)^{-1}$, respectively. In this case of a Dirac target, since $Q_\nu'(0)= Q_\nu'(1)=0$, the height of the "horns" is $\infty$.\\
Moreover, due to the explicit representation \eqref{eq:dirac-to-dirac-explicit} of the flow, we know that the initial Dirac point instantly becomes absolutely continuous up to the time $t < t^*$, where the flow touches the target Dirac point, mirroring the unregularized case. From there on, we technically do not know whether a Dirac point at $0$ develops as in the MMD-case, but our numerical calculations strongly suggest this, also see Figure \ref{fig:Dirac_to_Dirac_quantiles} and the upcoming paragraph for the technical implementation.

\paragraph{Technical implementation in the Dirac case:} 
To split the quantile \(g_n\) into the part with density and its singular part, we find the first point \(s_0\) such that \(|g_n(s)| < 10^{-3}\) for all $s \ge s_0$. For the interval \([0, s_0]\), we compute the density, while the constant part \([s_0, 1]\) becomes a Dirac at \(0\) with mass \(1 - s_0\). In the unregularized case, we indeed have a closed form of the gradient flow, see \cite[Example 6.8]{DSBHS2024}. If the tolerance is chosen too small, parts of the singular component may be incorrectly identified as having density, causing the plot to display a false behavior. Based on the closed form, we found that a tolerance of \(10^{-3}\) works best in our numerical computations. In the regularized case, we do not have a closed form for the entire flow. Here, we also divide the measure into an absolutely continuous and a singular part in the same way as in the unregularized case. The numerical results strongly suggest that the regularized flow also accumulates to a Dirac at \(0\). 
Another interesting feature of Algorithm \ref{alg:unregularizedEuler} is that it returns both \(g_N\) and \(g_N'\), allowing the density to be computed directly instead of employing finite differences.

\begin{figure}[H]
    \centering
    \includegraphics[width=.32\textwidth]{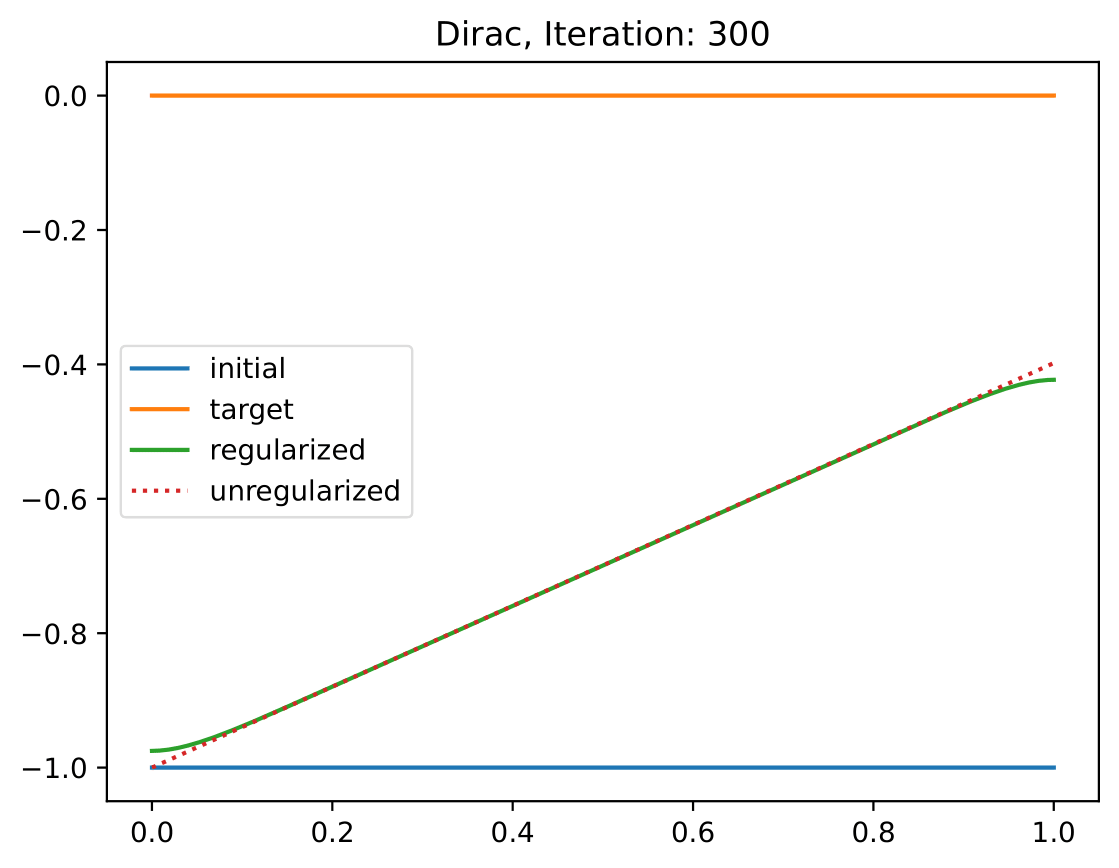}
    \includegraphics[width=.32\textwidth]{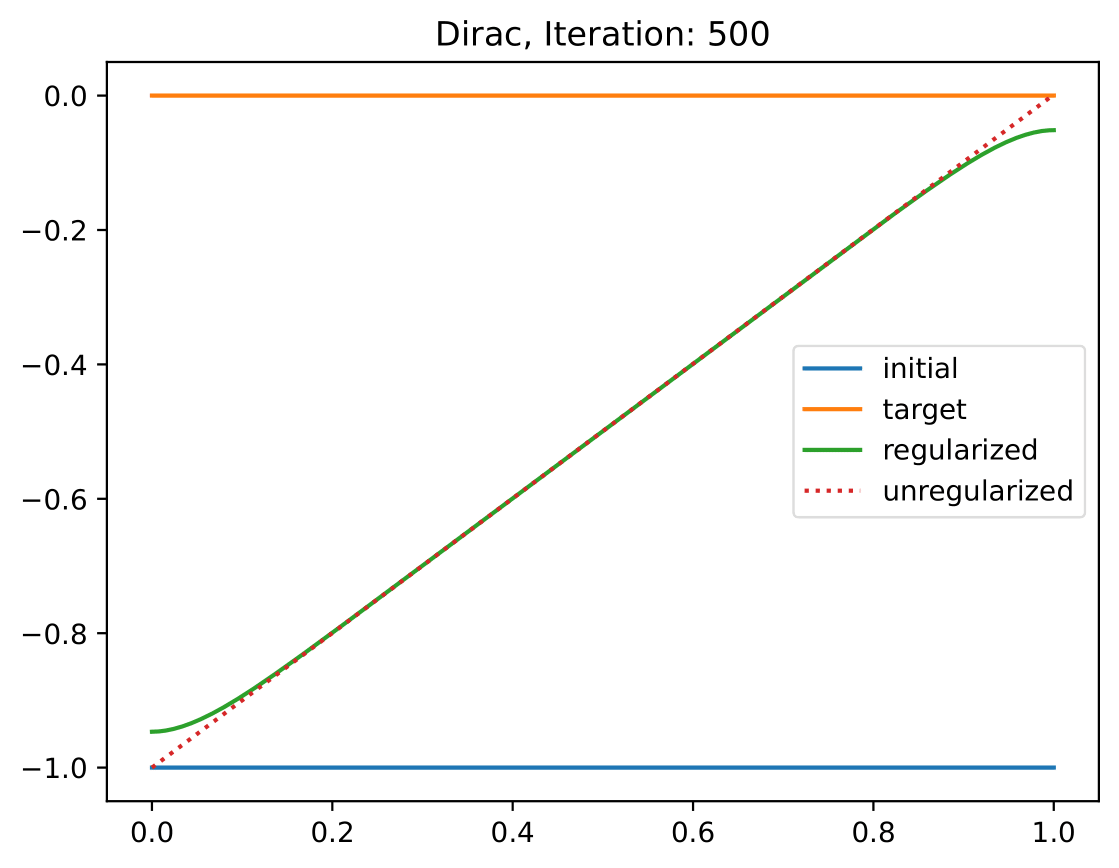}
    \includegraphics[width=.32\textwidth]{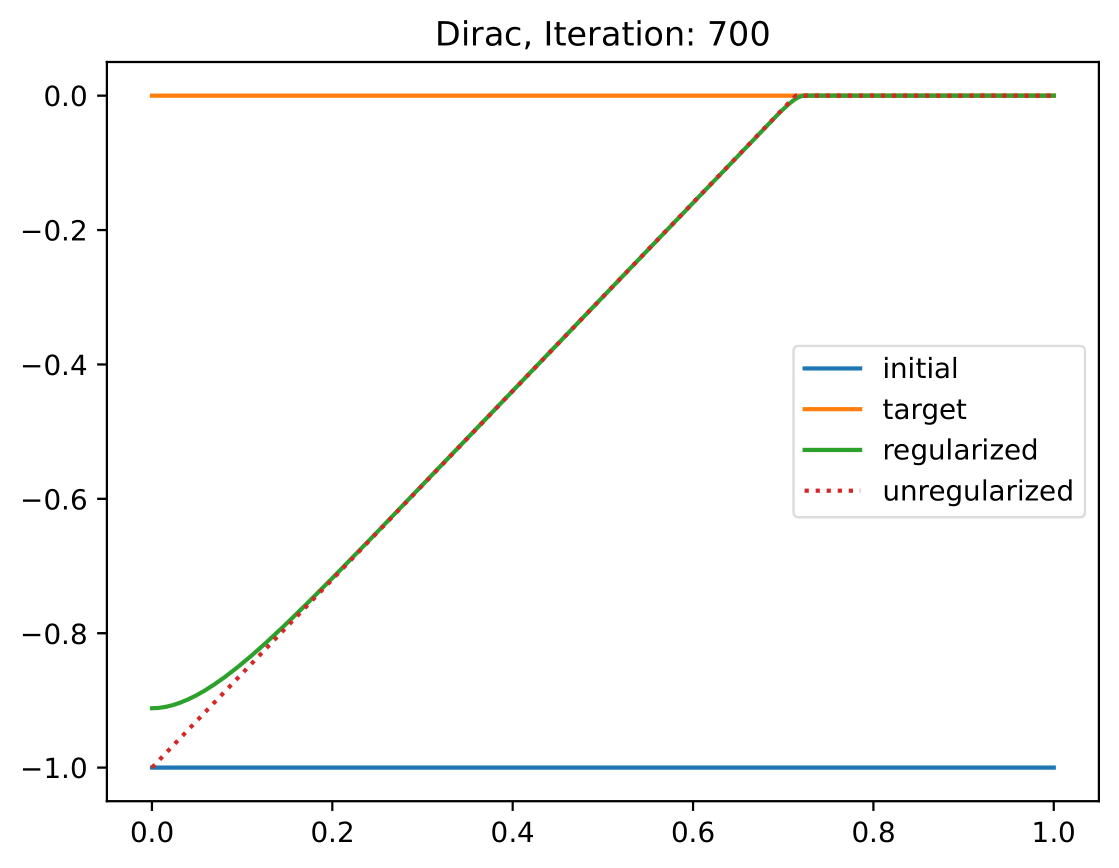}
    \\
    \includegraphics[width=.32\textwidth]{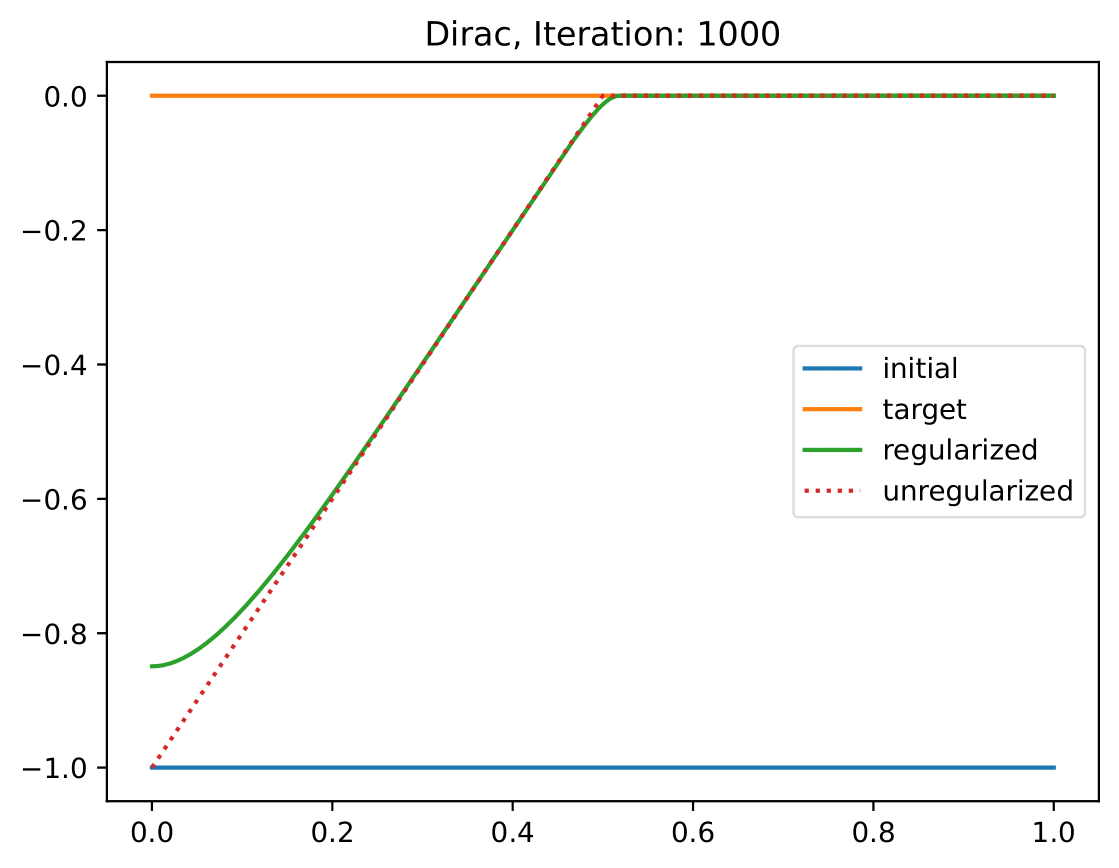}
    \includegraphics[width=.32\textwidth]{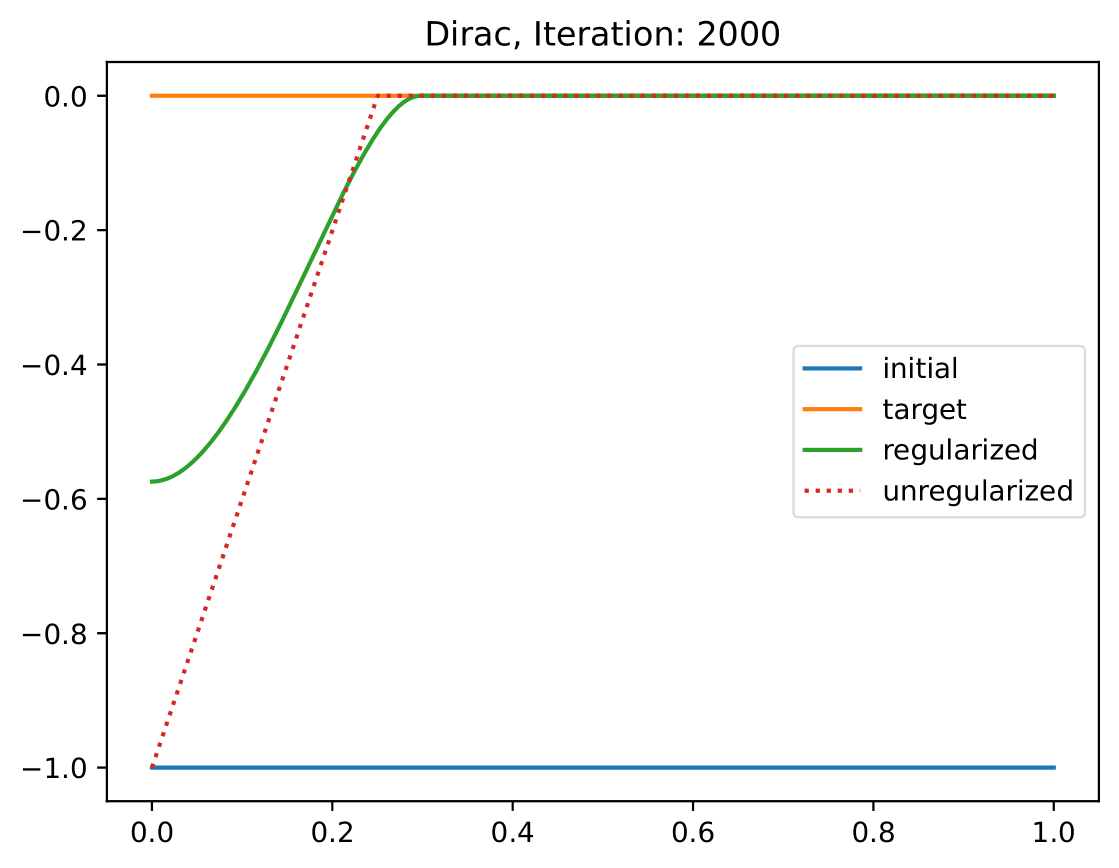}
    \includegraphics[width=.32\textwidth]{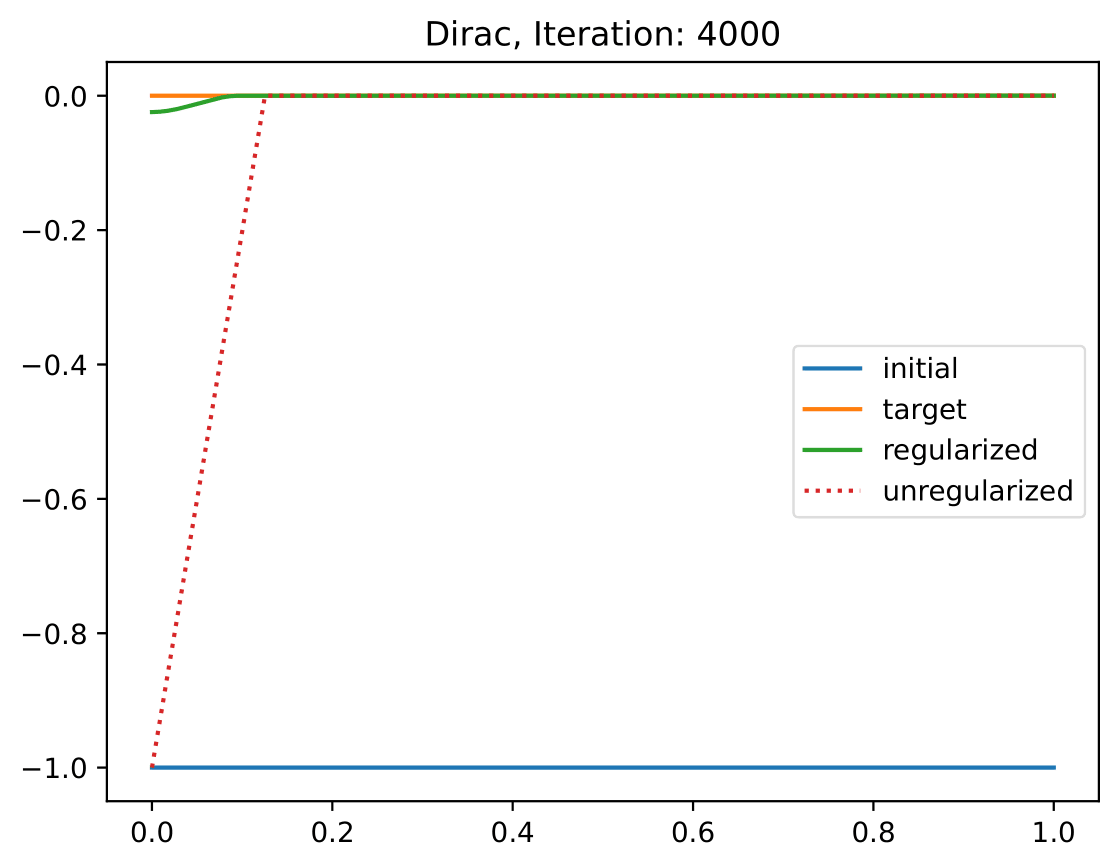}
    \caption{Quantiles of the regularized 
    and unregularized ${\F}_\nu^-$-flow from $\delta_{-1}$ ($Q_0 \equiv -1$) to $\delta_{0}$ ($Q_\nu \equiv 0$)
		for $\tau = 10^{-3}, \,\lambda = 10^{-2}$. The results strongly suggest an emergence of a Dirac at $0$ in the regularized flow. 
    }
    \label{fig:Dirac_to_Dirac_quantiles}
\end{figure}

The size of the diffusion constant $\lambda > 0$ naturally determines the effective contributions of the two functionals, which are the \emph{repulsiveness} of $F_\nu^-$ and the \emph{support-shifting} of $F_{H,\nu}$. For small $\lambda$, the mass expands as quickly as the $F_\nu^-$-flow, while the support moves very slowly. However, for large $\lambda$, the mass expands slowly and the support shifts very quickly; see Remark \ref{rem:semigroup-convergence}, \ref{rem:long-time}, and the Figures \ref{fig:Uniform_to_Uniform_large}, \ref{fig:Uniform_to_Uniform_small} in the appendix.

\begin{figure}[H]
    \centering
    \includegraphics[width=.32\textwidth]{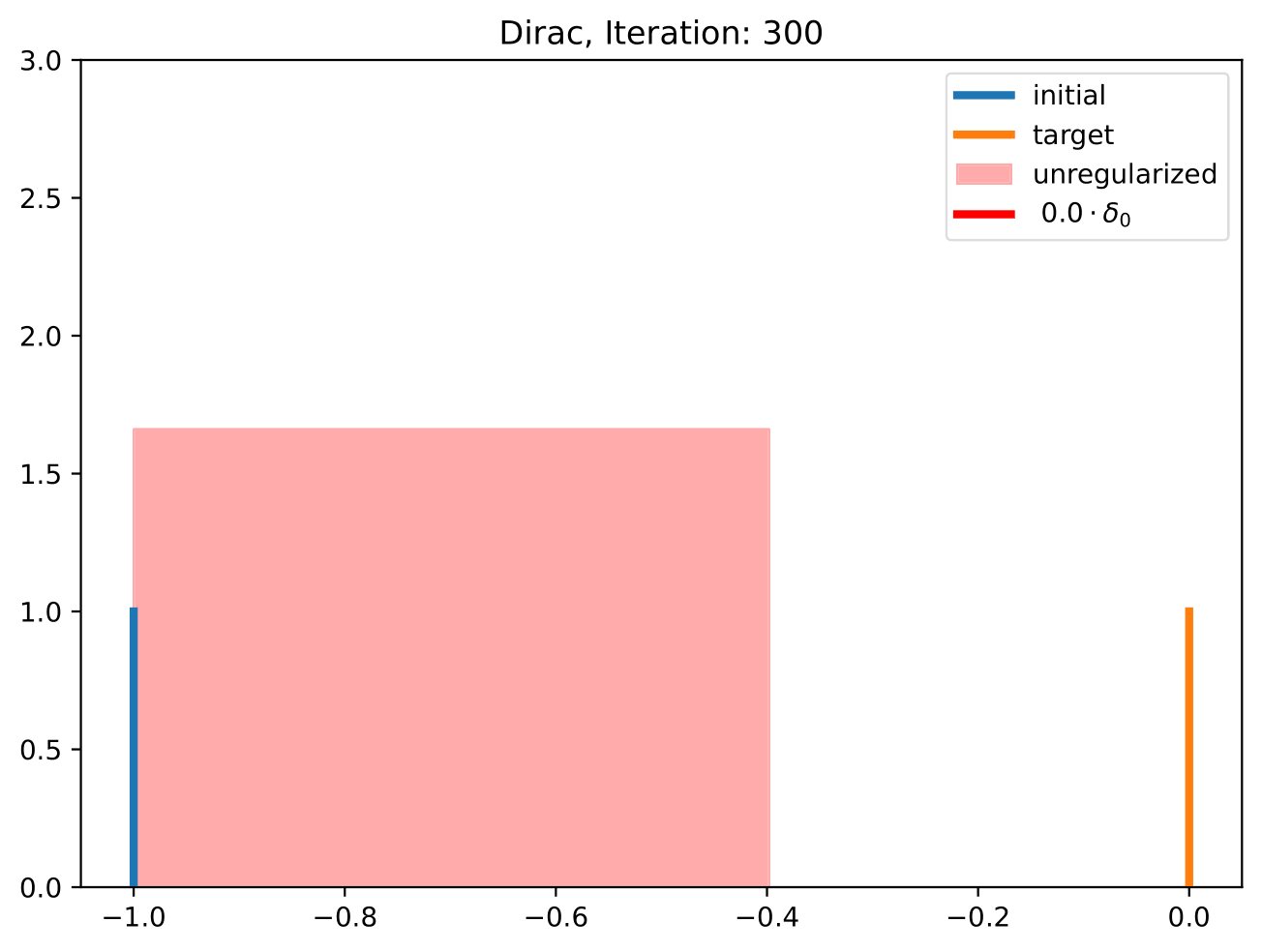}
    \includegraphics[width=.32\textwidth]{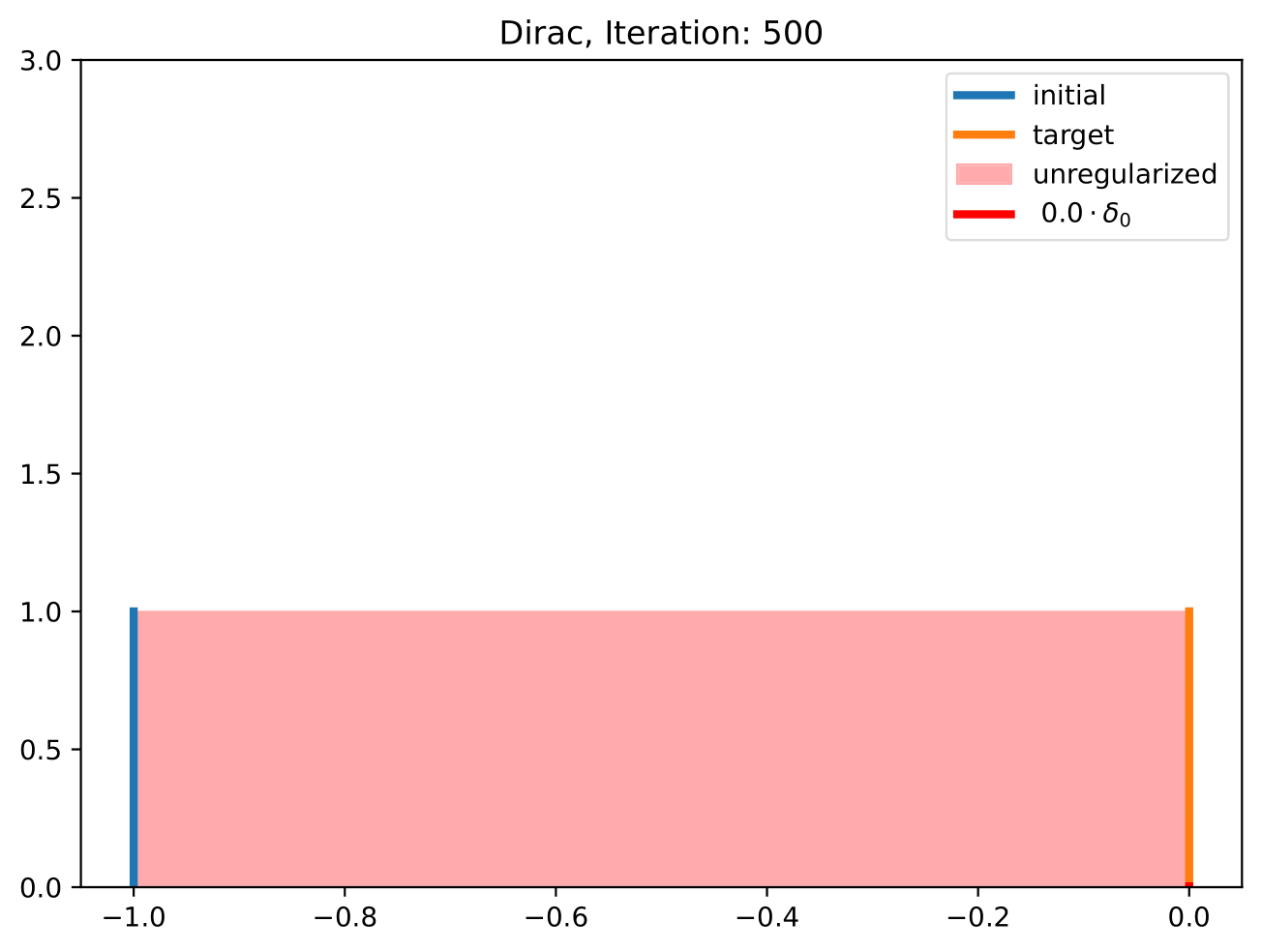}
    \includegraphics[width=.32\textwidth]{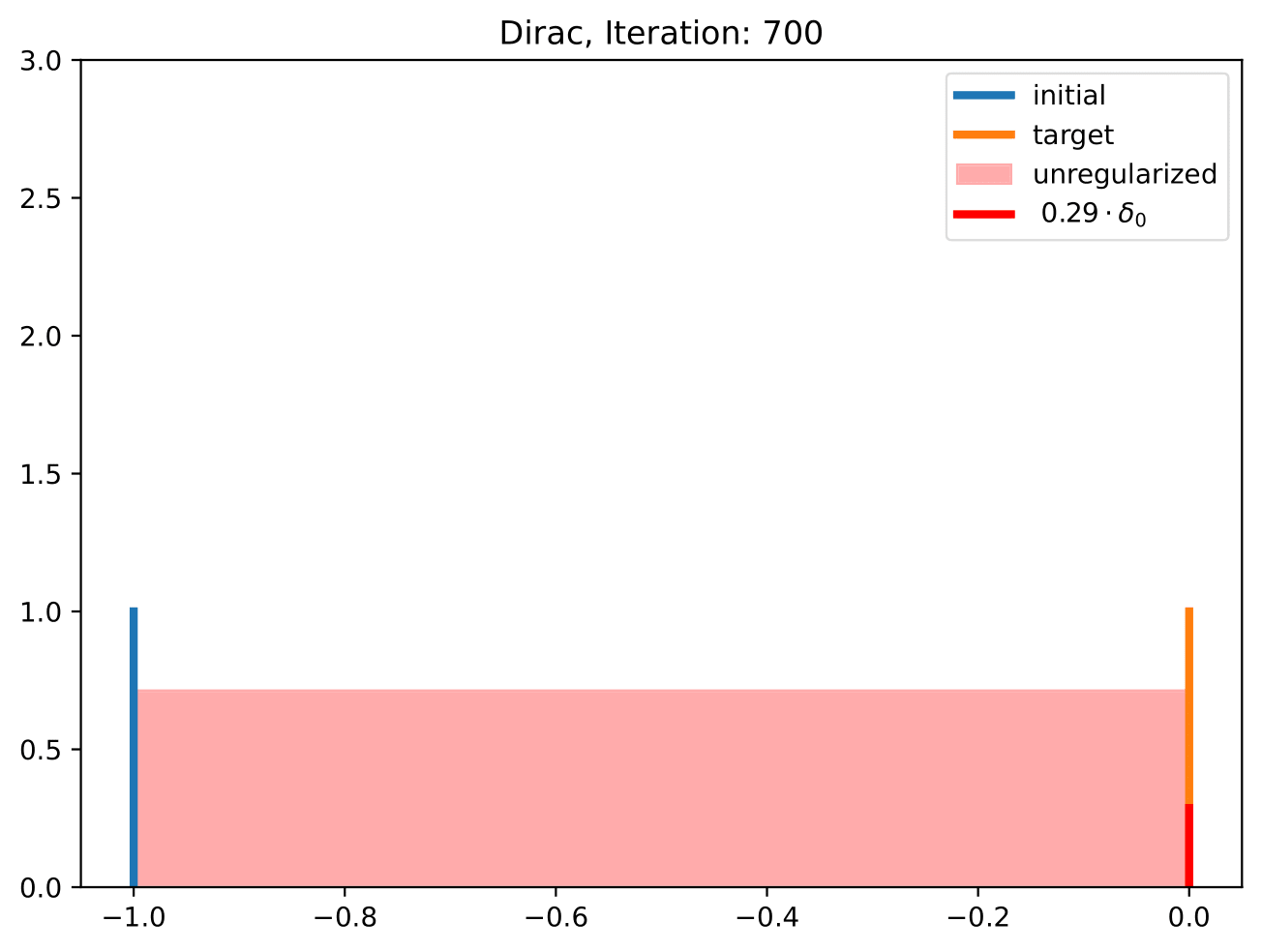}
    \\
    \includegraphics[width=.32\textwidth]{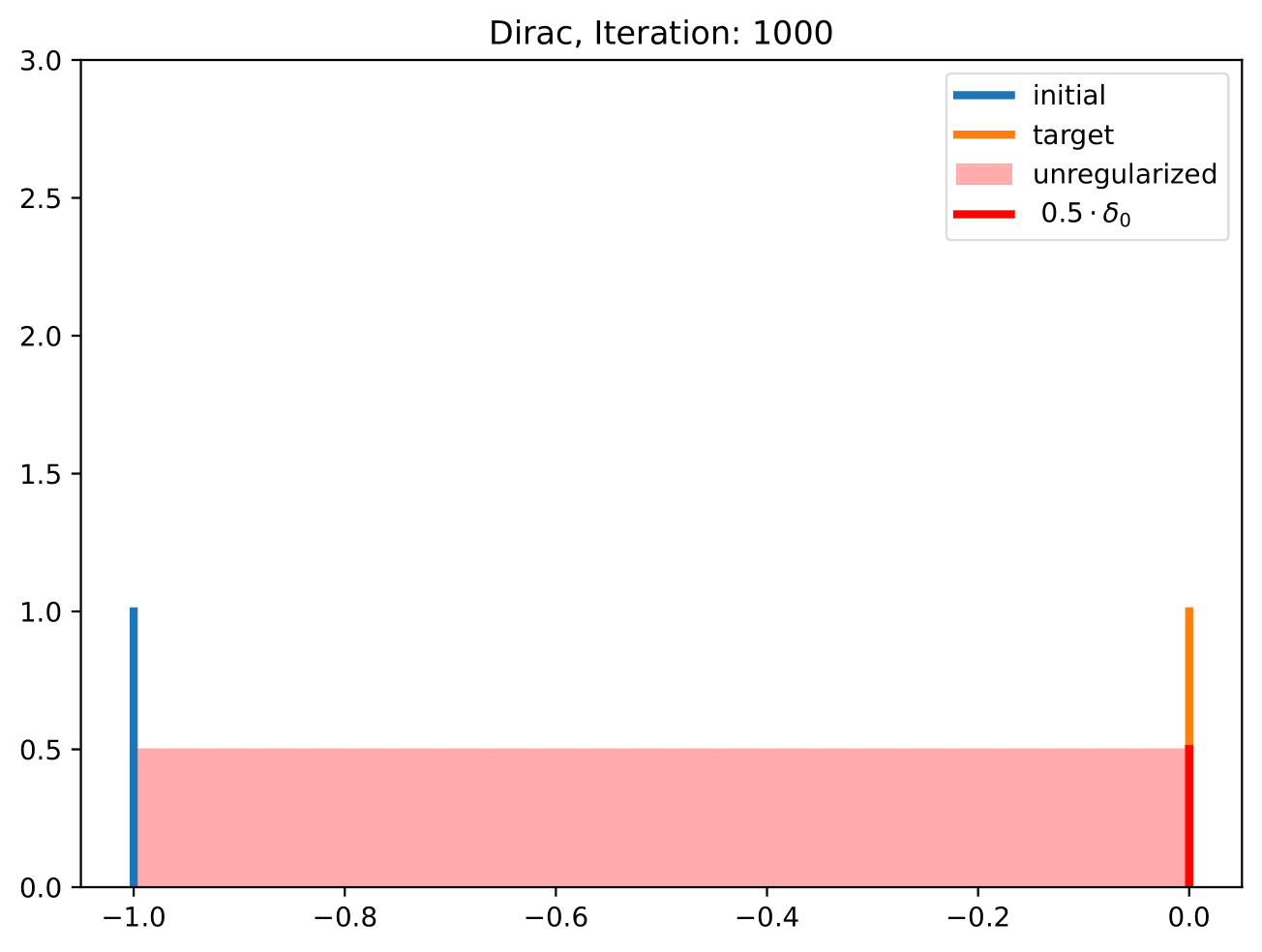}
    \includegraphics[width=.32\textwidth]{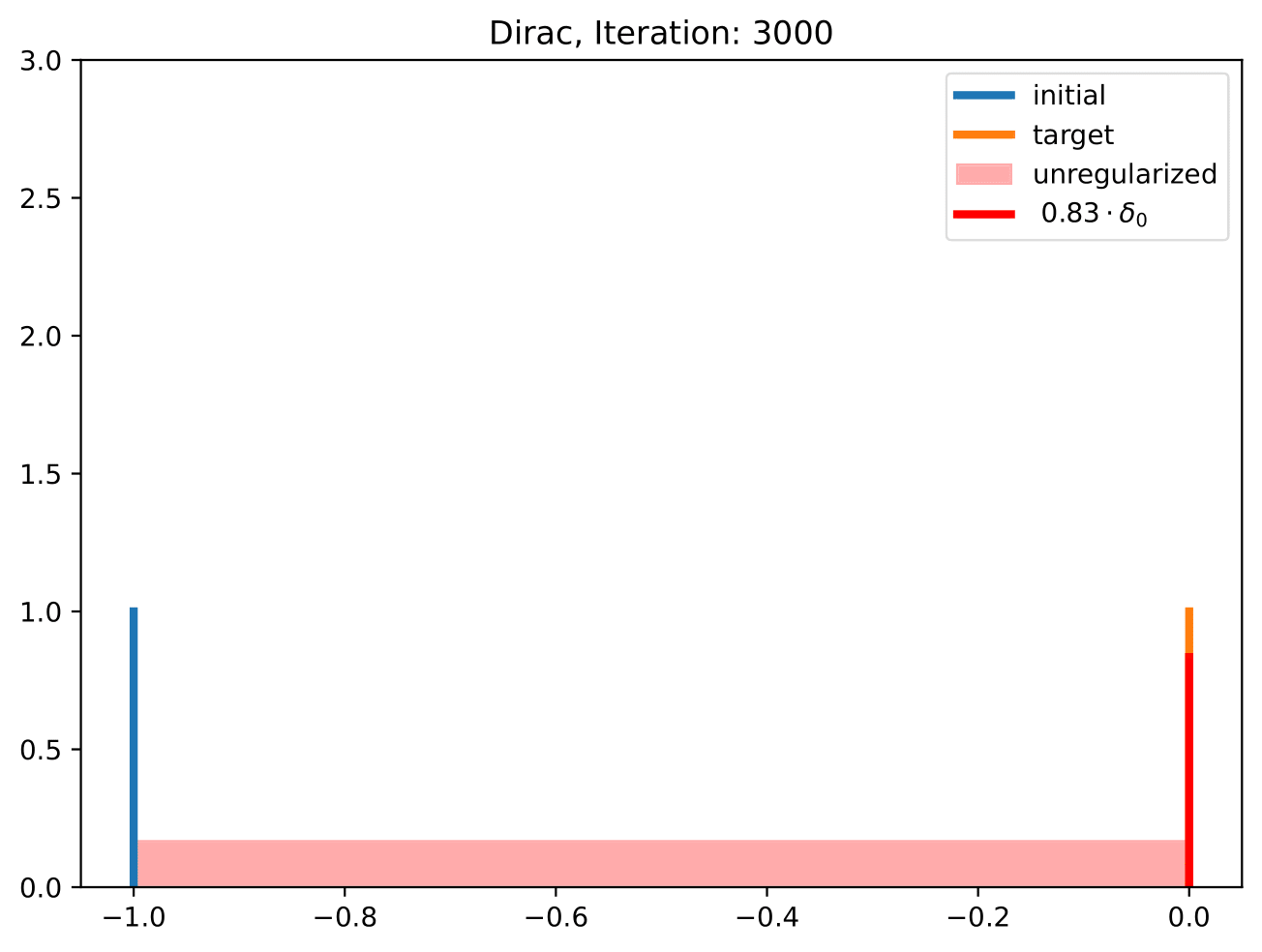}
    \includegraphics[width=.32\textwidth]{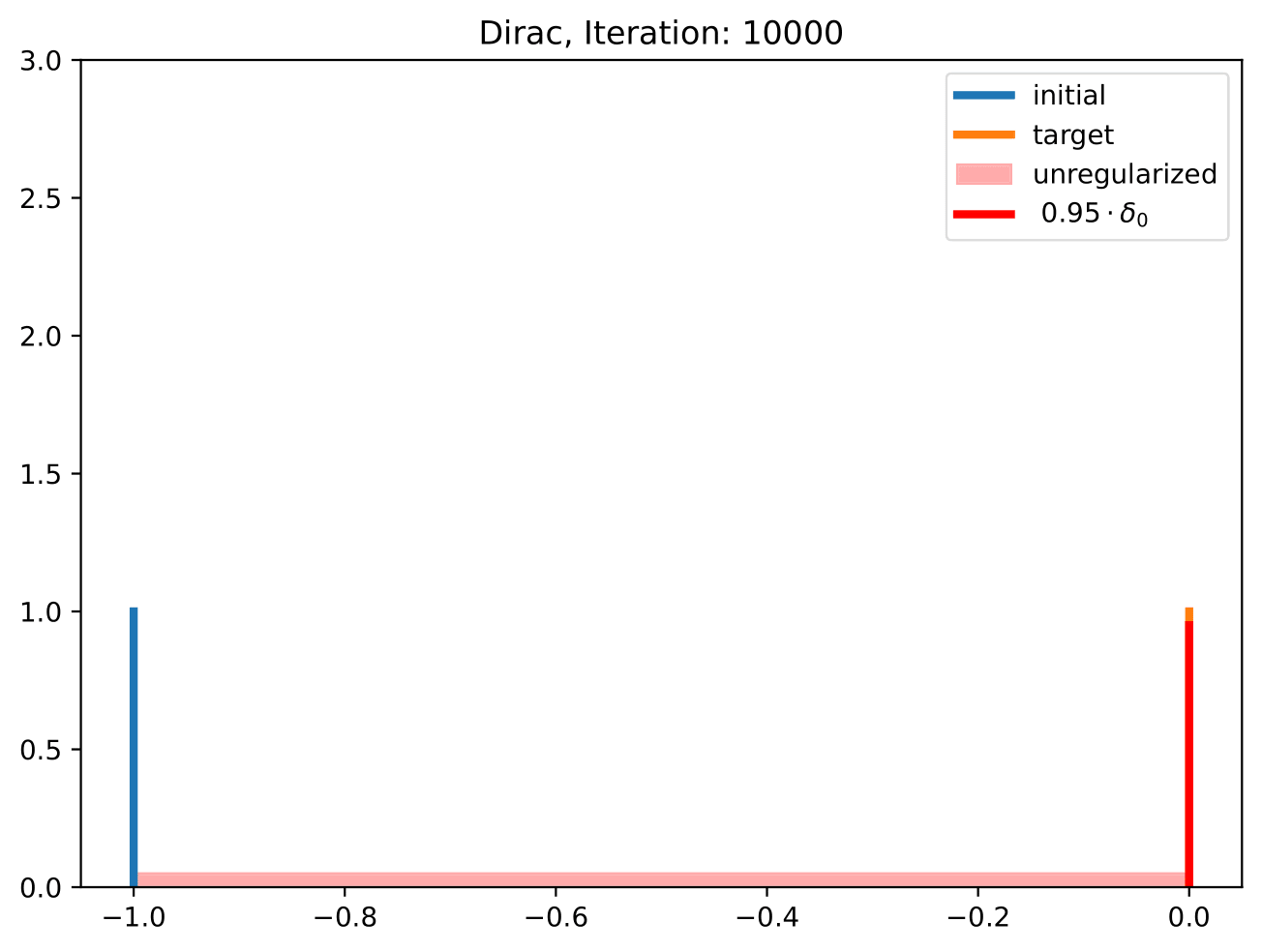}
    \caption{Unregularized ${\mathcal F}_\nu^-$-flow from $\delta_{-1}$ to $\delta_{0}$ 
		 for $\tau = 10^{-3}$.  The light color represents the absolute continuous part, while the dark color displays the singular part of the flow. The mass in $(-1,0)$ slowly dissipates, but never totally vanishes.
    }
    \label{fig:Dirac_to_Dirac_unreg}
\end{figure}

\begin{figure}[H]
    \centering
    \includegraphics[width=.32\textwidth]{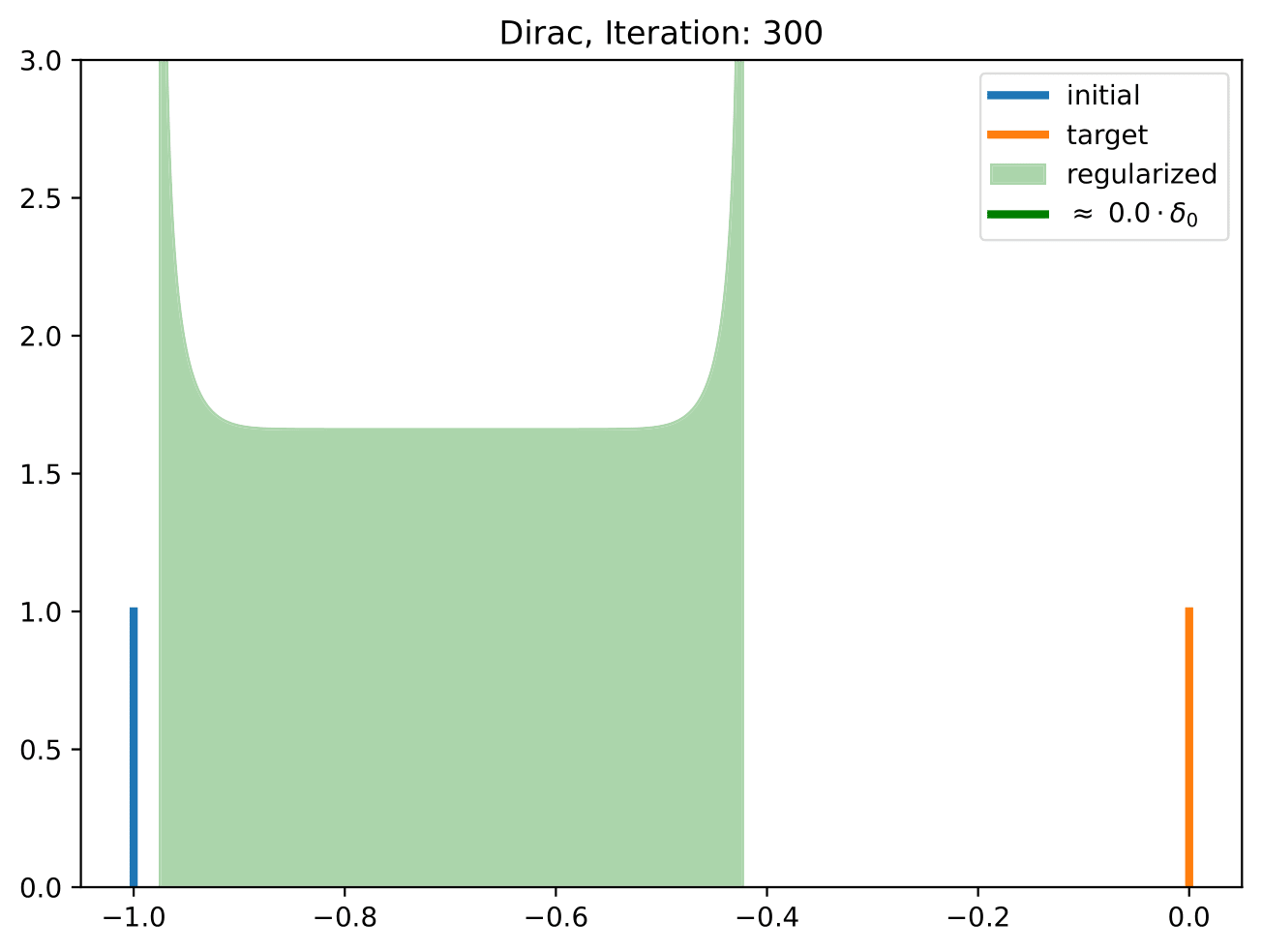}
    \includegraphics[width=.32\textwidth]{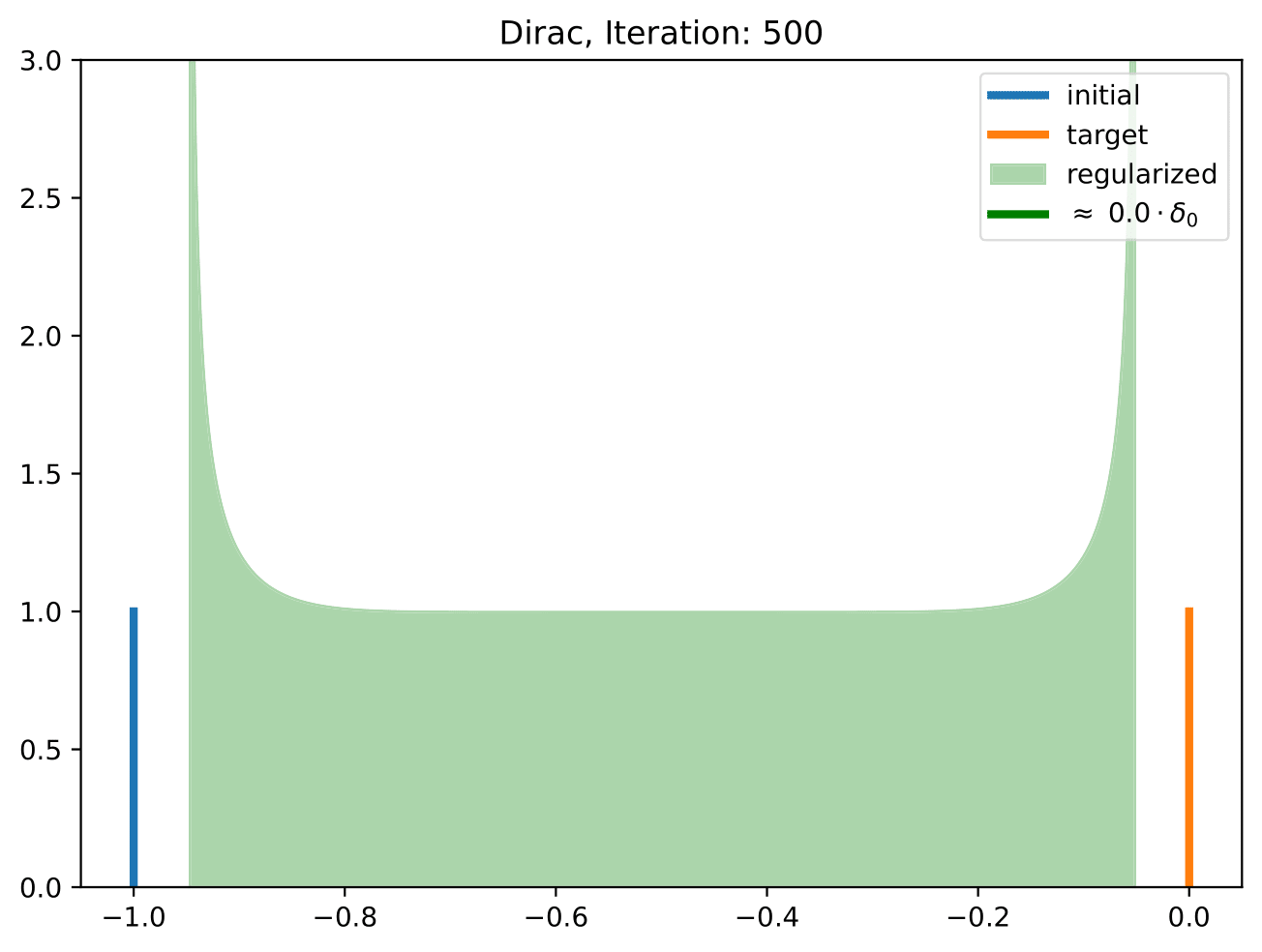}
    \includegraphics[width=.32\textwidth]{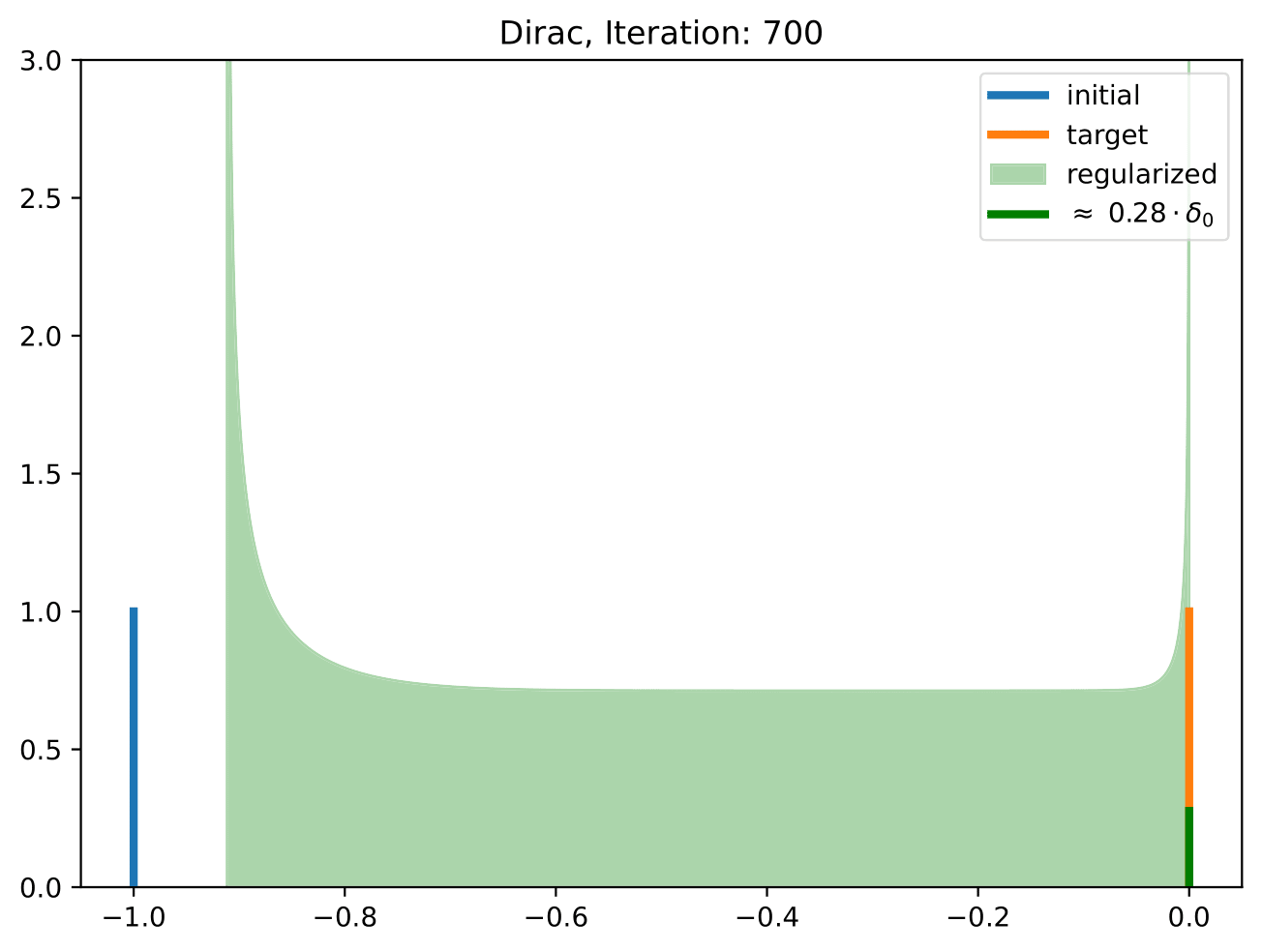}
    \\
    \includegraphics[width=.32\textwidth]{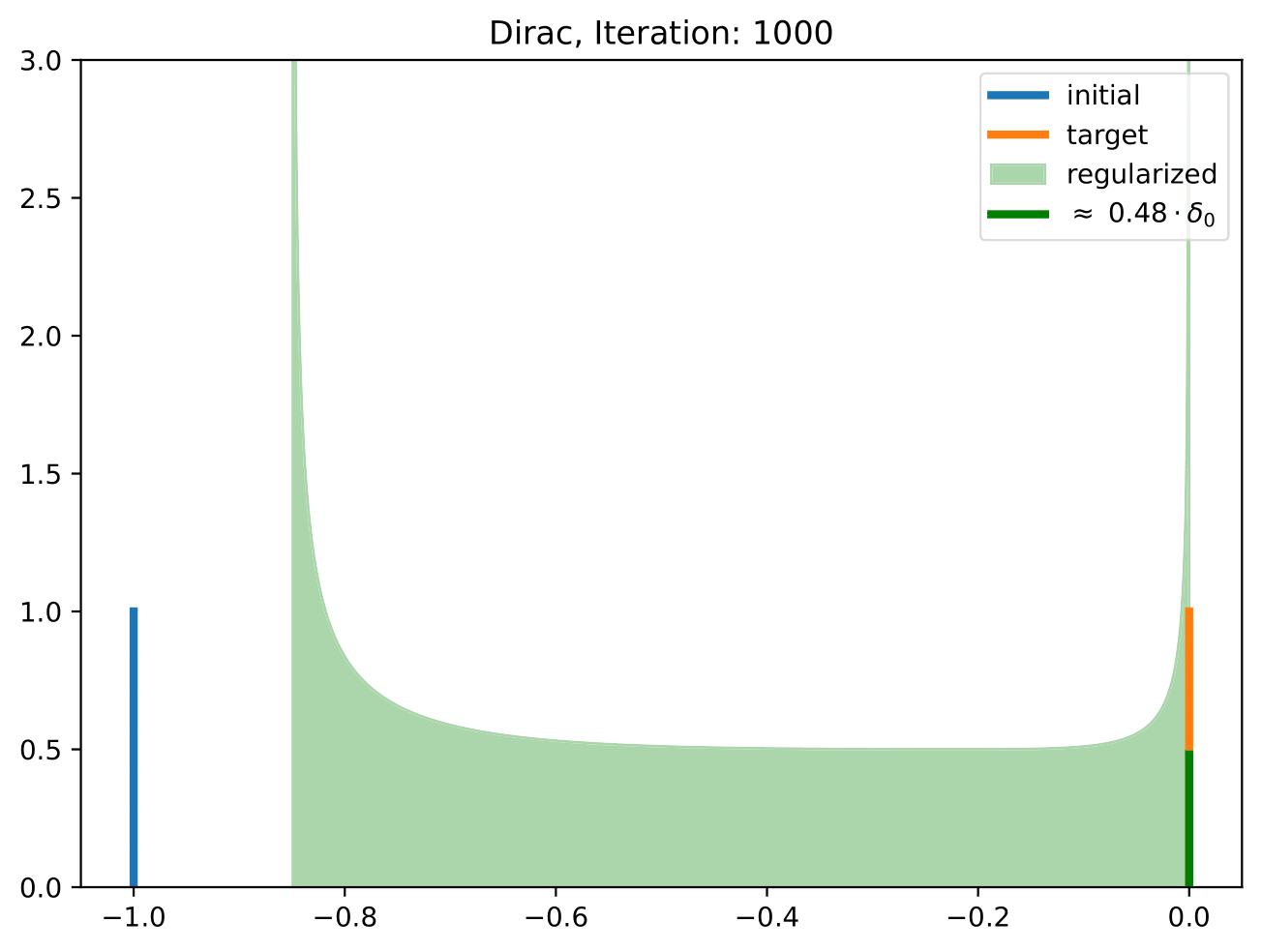}
    \includegraphics[width=.32\textwidth]{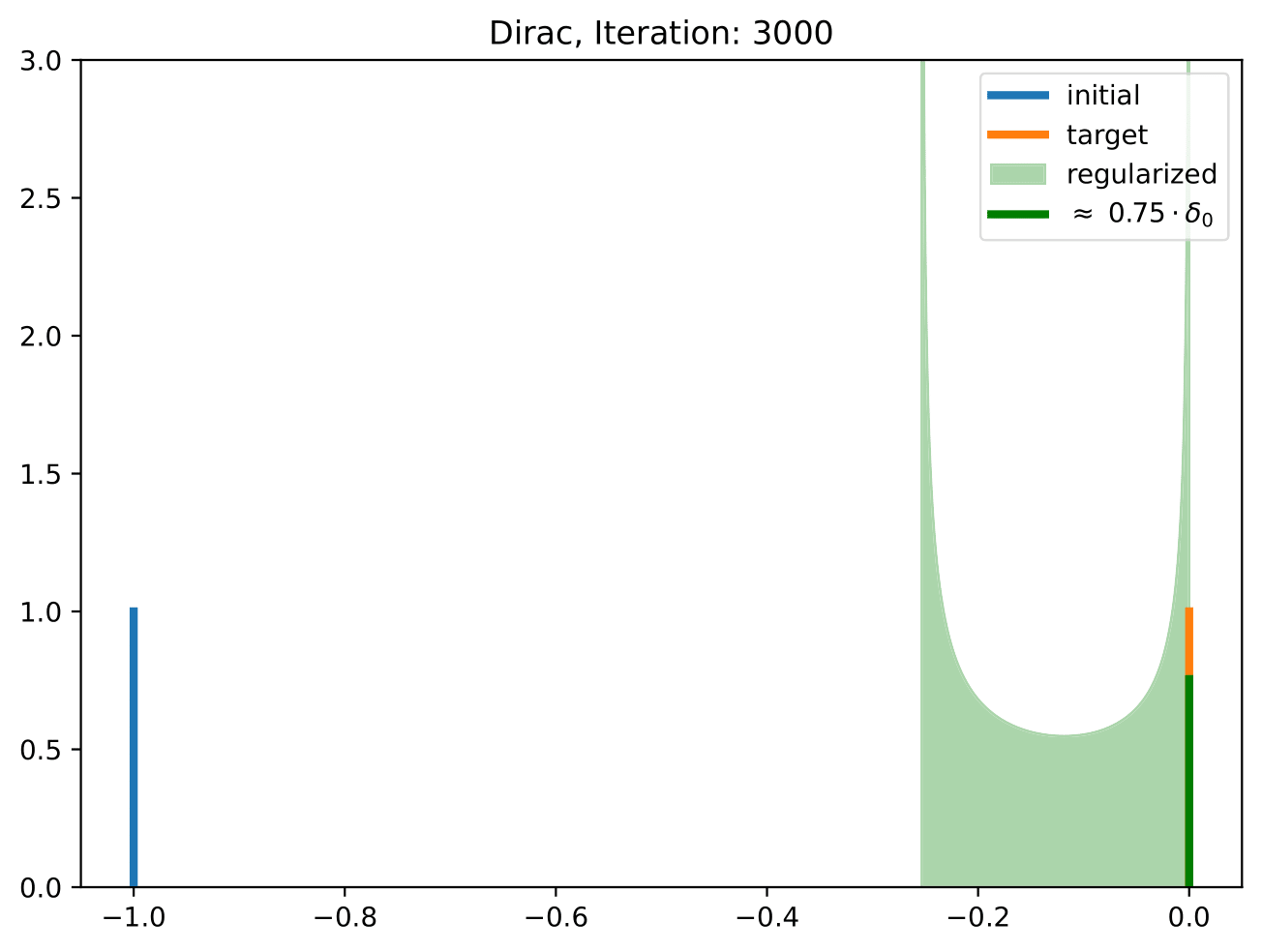}
    \includegraphics[width=.32\textwidth]{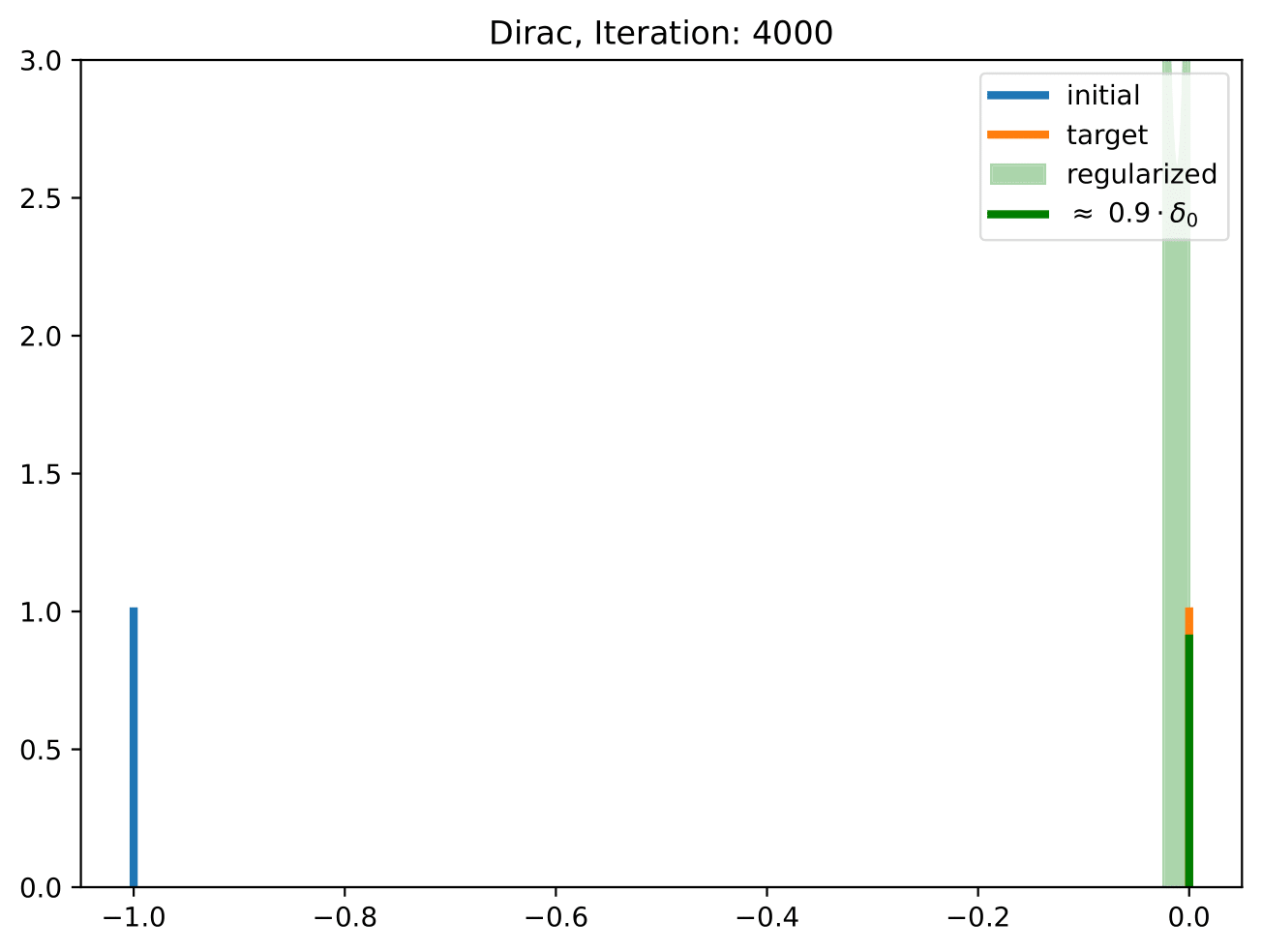}
    \caption{Regularized ${\mathcal F}_\nu^-$-flow from $\delta_{-1}$ to $\delta_{0}$ 
		for $\tau = 10^{-3}, \,\lambda = 10^{-2}$.  
		The "horns" at the boundary of the support are of height $Q_\nu'(0)^{-1} = Q_\nu'(1)^{-1} = \infty$. The support shifts towards the target.
    }
    \label{fig:Dirac_to_Dirac_reg}
\end{figure}
\newpage

In all the remaining examples, we observe a similar behaviour:
\emph{a dissipation of "sticky" mass} in the unregularized ${\F}_\nu^-$-flow; and a rectified flow of $\widetilde{\F}_\nu^-$ where the support actually \emph{moves} towards the target and the quantiles approximate the target in a smoother fashion. Hence, we are left to check whether the assumptions of Proposition \ref{prop:reduced-cauchy} are fulfilled.

\paragraph{Uniform-to-Uniform:} Consider as target the uniform distribution $\nu \sim \mathcal{U}[0,1]$, and as initial measure another Uniform distribution $\gamma_0 \sim \mathcal{U}[-3,-1]$. The corresponding quantiles are given by $Q_\nu(s) = s$ and $Q_0(s) = 2s-3$, and hence satisfy the assumptions of Proposition \ref{prop:reduced-cauchy} for any $\lambda > 0$. Note that here, the Neumann boundary values $Q_\nu'(0) = Q_\nu'(1) = 1$ instantly lead to "horn-shaped barricades" of height $1$. More precisely, by the inverse function rule, it holds 
\begin{equation*}
    Q_\nu'(0) = \frac{1}{f_\nu (Q_\nu(0))}, \quad Q_\nu'(1) = \frac{1}{f_\nu (Q_\nu(1))},
\end{equation*}
where $f_\nu$ denotes the density of $\nu$. Hence, the height of the "horns" is given by $f_\nu (Q_\nu(0)), \, f_\nu (Q_\nu(1))$, i.e., the values of the target density $f_\nu$ at the boundary of its support $\supp \nu$. \\
The flows of $\F_\nu^-$, $\widetilde{\F}_\nu^-$ and their quantile flows are depicted in Figues \ref{fig:Uniform_to_Uniform_unreg}, \ref{fig:Uniform_to_Uniform_reg} and \ref{fig:Uniform_to_Uniform_quantiles}, respectively.

\paragraph{Gaussian-to-Gaussian \emph{like}:} Consider as target a suitably modified, compactly supported version of a Gaussian distribution $\widetilde{\nu} \sim \mathcal{N}(-5,1)$, and as initial measure an analogous version of $\widetilde{\gamma_0} \sim \mathcal{N}(5,1)$.
The modifications can be chosen the following way:
Since the quantile $Q_{\widetilde{\nu}} \in C^\infty(0,1)$ is unbounded, we cut it off outside an interval $[a,b] \subset (0,1)$, which can be chosen arbitrarily close to $(0,1)$. Then, it holds that $Q_{\widetilde{\nu}}|_{[a,b]} \in C^\infty([a,b])$ is smooth and bounded. Hence, we can extend $Q_{\widetilde{\nu}}|_{[a,b]}$ to a function $Q_\nu \in C^3([0,1])$ with arbitrary boundary values $Q_\nu'(0), Q_\nu'(1) > 0$, and we can proceed the same way with $Q_0$.
The hereby modified quantiles $Q_\nu$ and $Q_0$ hence satisfy the assumptions of Proposition \ref{prop:reduced-cauchy} for sufficiently small $\lambda > 0$. The approximate flows of $\F_\nu^-$, $\widetilde{\F}_\nu^-$ and their quantile flows are depicted in Figues \ref{fig:Gaussian_to_Gaussian_unreg}, \ref{fig:Gaussian_to_Gaussian_reg} and \ref{fig:Gaussian_to_Gaussian_quantiles}, respectively.

\paragraph{Technical implementation in the Gaussian-like case:} 
{
For simplicity, we chose to only work with the cut-off function $Q_{\widetilde{\nu}}|_{[a,b]}$ on the smaller interval $[a,b] = [1/(n+1), n/(n+1)]\subset (0,1)$, where we chose $n=1000$. 
Therefore, we only solve \eqref{eq:dirichlet-problem} on $[1/(n+1), n/(n+1)]$, where the boundary values are given by $Q_{\tilde{\nu}}'(1/(n+1))$ and $Q_{\tilde{\nu}}'(n/(n+1))$. As a result, the "horns" of the theoretical extension $Q_\nu \in C^3([0,1])$ are cut-off in Figure \ref{fig:Gaussian_to_Gaussian_reg}. But since the horns can be chosen to have small height $Q_\nu'(0)^{-1}, \, Q_\nu'(1)^{-1} \approx 0$ anyway, this simplification is reasonable.
}

\newpage
\begin{figure}[H]
    \centering
      \includegraphics[width=.32\textwidth]{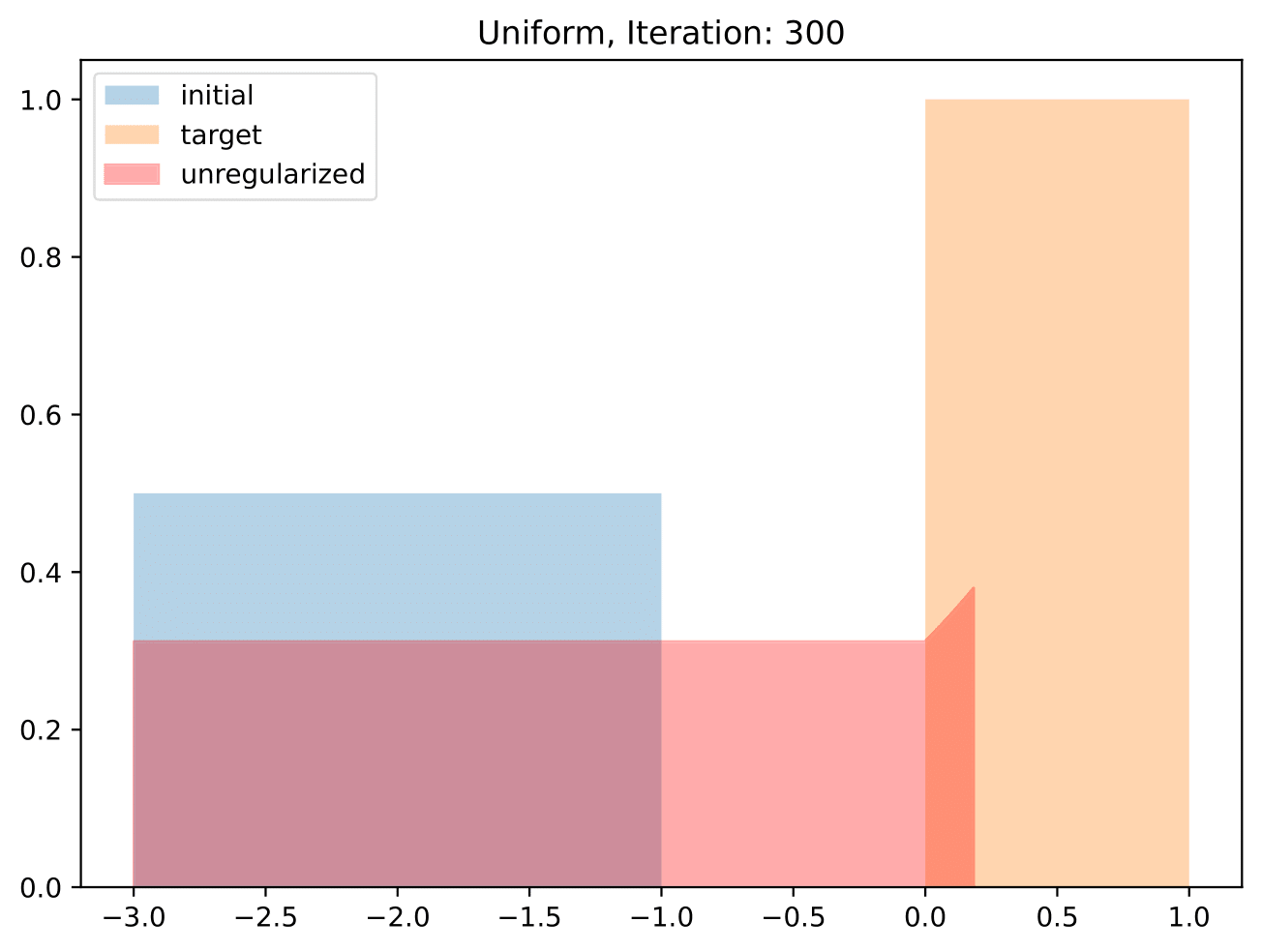}
    \includegraphics[width=.32\textwidth]{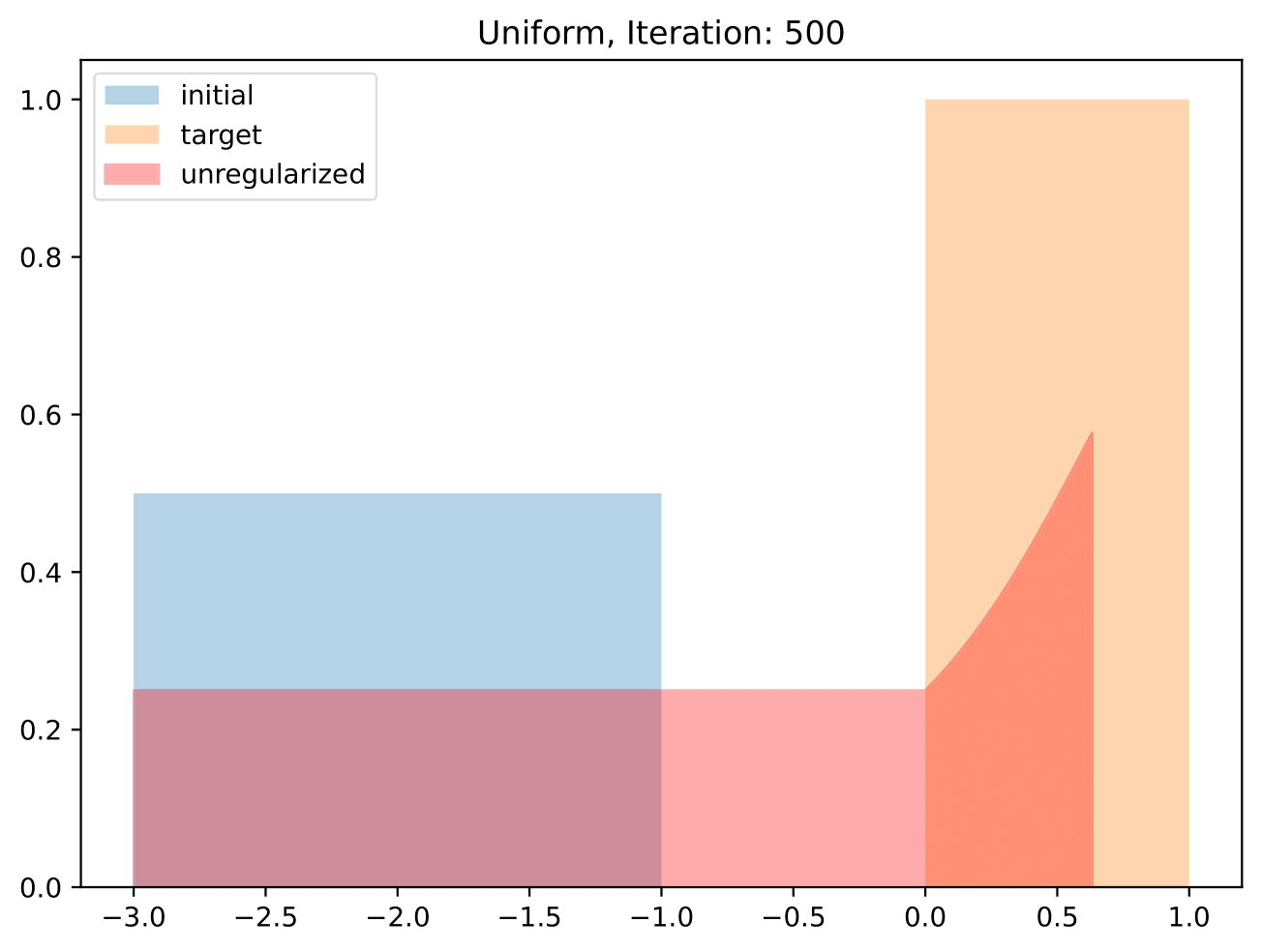}
    \includegraphics[width=.32\textwidth]{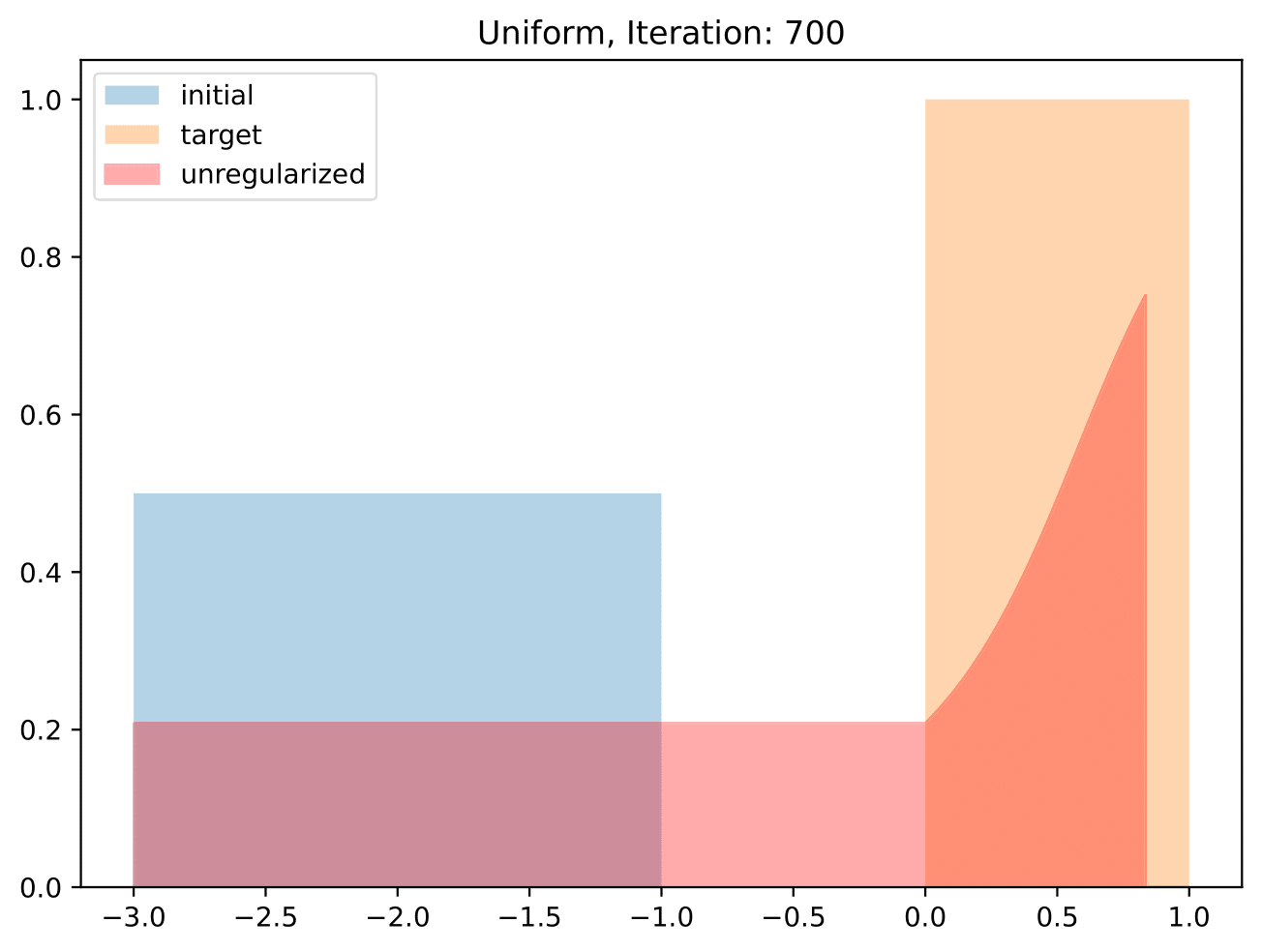}
    \\
    \includegraphics[width=.32\textwidth]{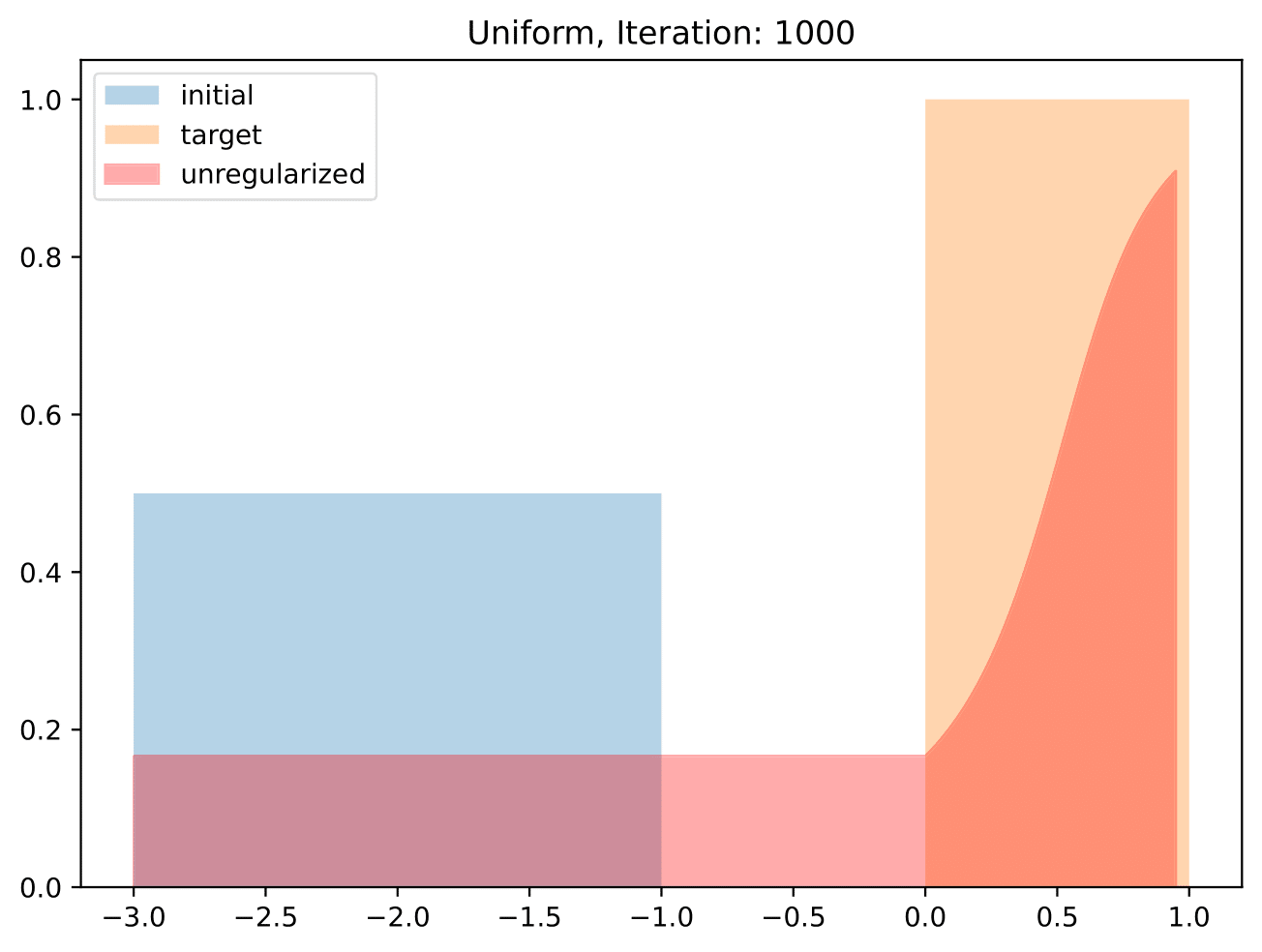}
    \includegraphics[width=.32\textwidth]{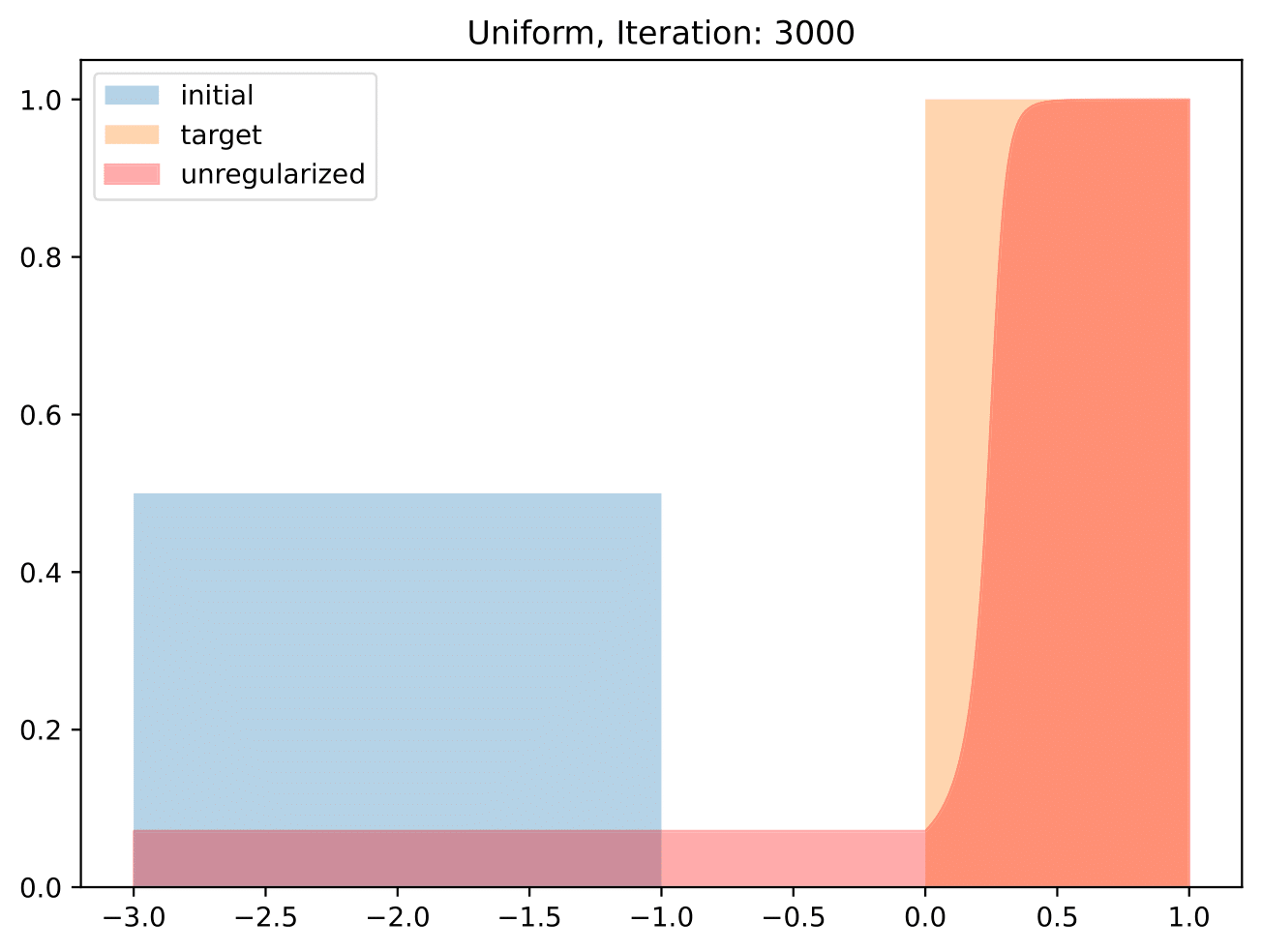}
    \includegraphics[width=.32\textwidth]{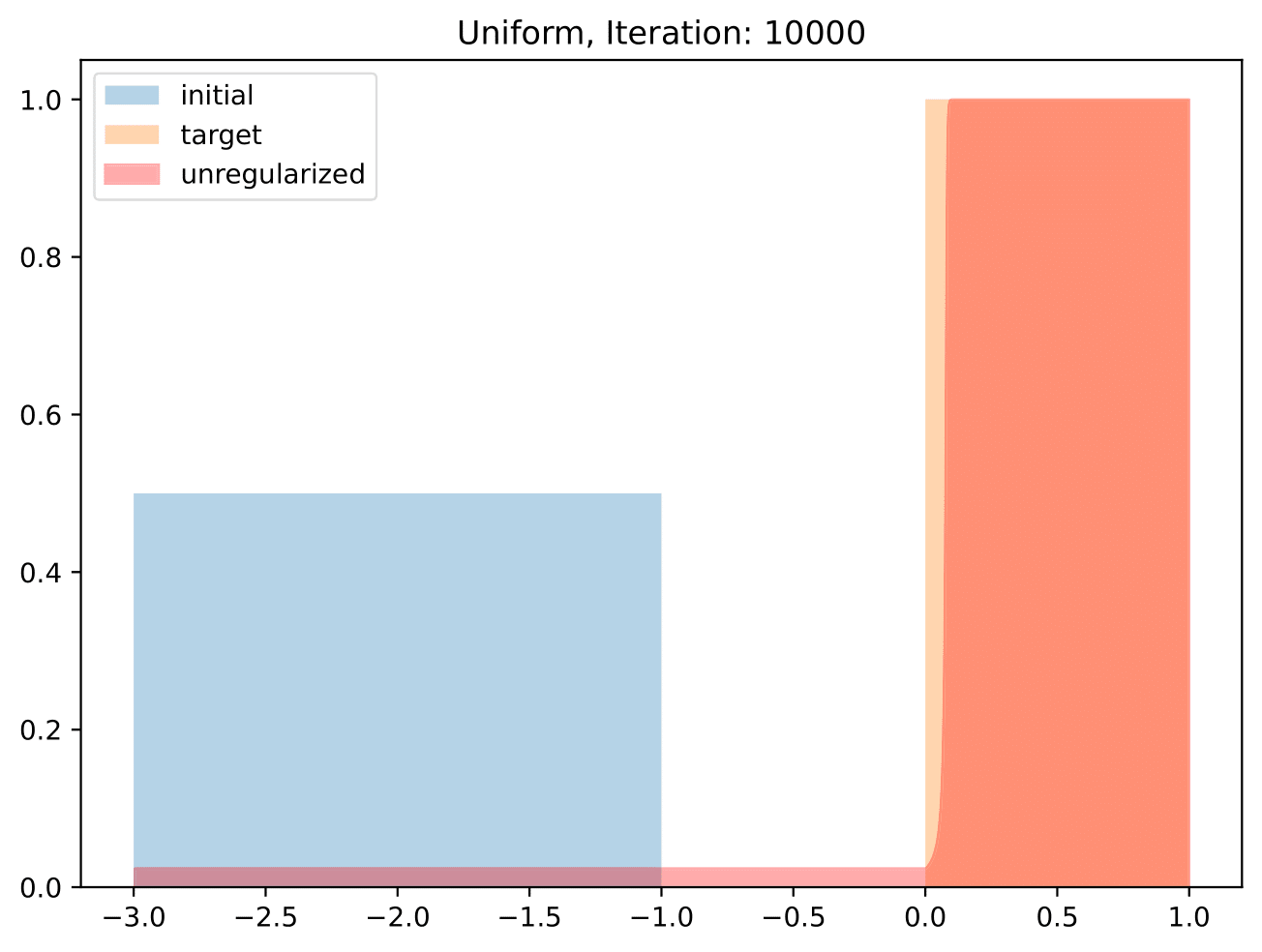}
    \caption{Unregularized ${\mathcal F}_\nu^-$-flow  between the uniform measures $\gamma_0 \sim \mathcal{U}[-3,-1]$ 
		and $\nu \sim \mathcal{U}[0,1]$ for
		$\tau = 2\cdot 10^{-3}$. The mass outside $\supp \nu$ dissipates, but never totally vanishes.
    }
    \label{fig:Uniform_to_Uniform_unreg}
\end{figure}

\begin{figure}[H]
    \centering
    \includegraphics[width=.32\textwidth]{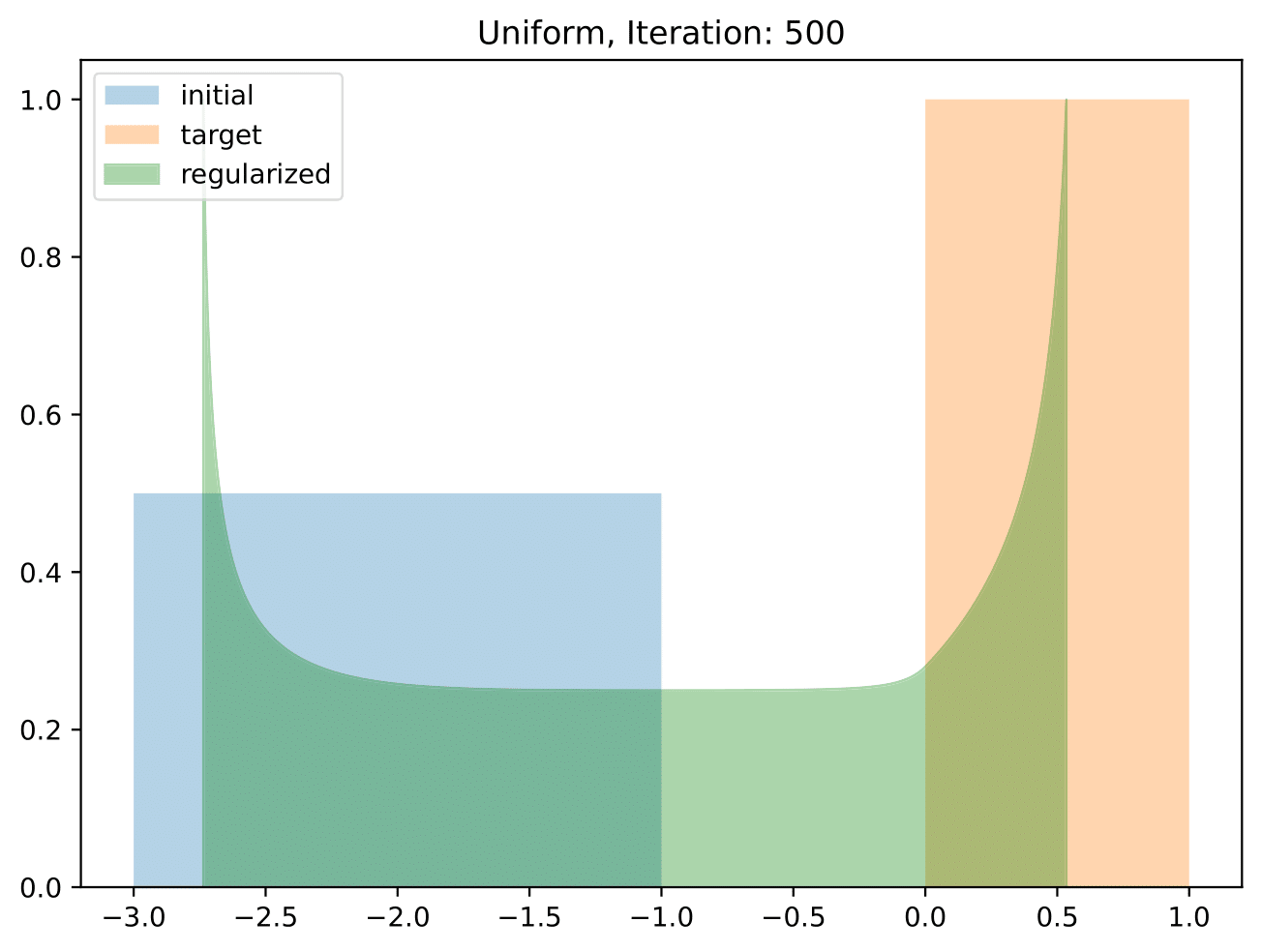}
    \includegraphics[width=.32\textwidth]{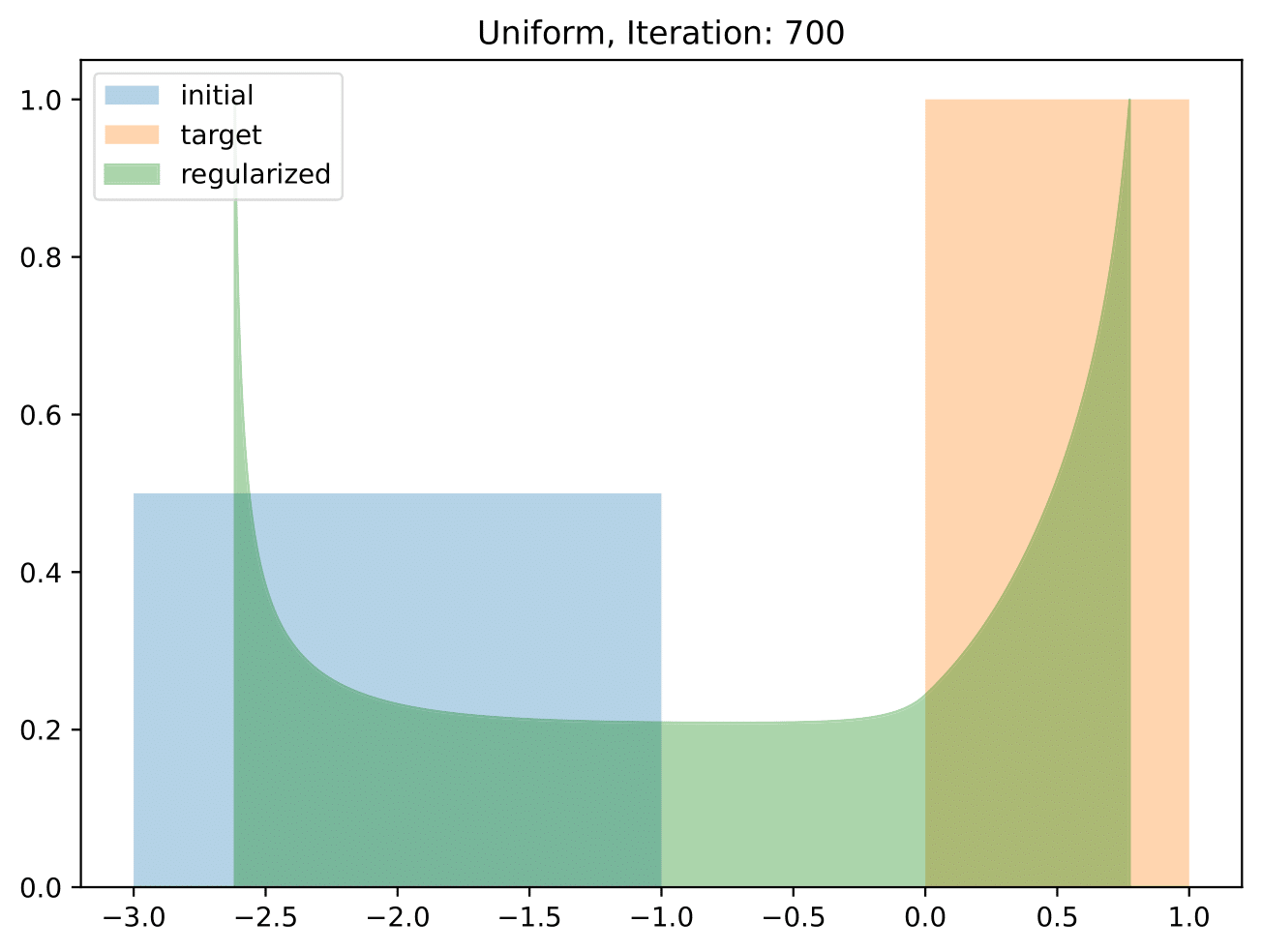}
    \includegraphics[width=.32\textwidth]{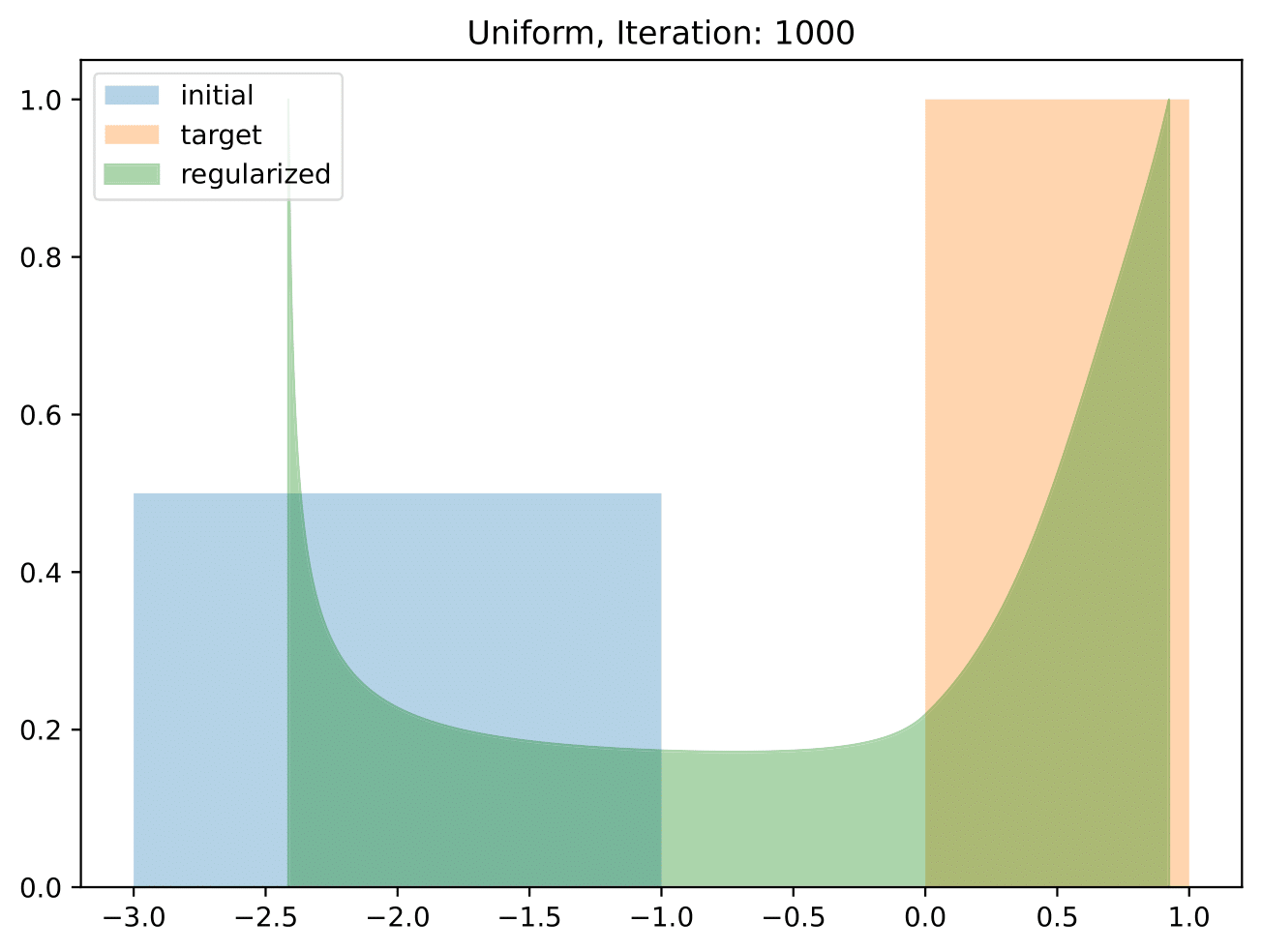}
    \\
    \includegraphics[width=.32\textwidth]{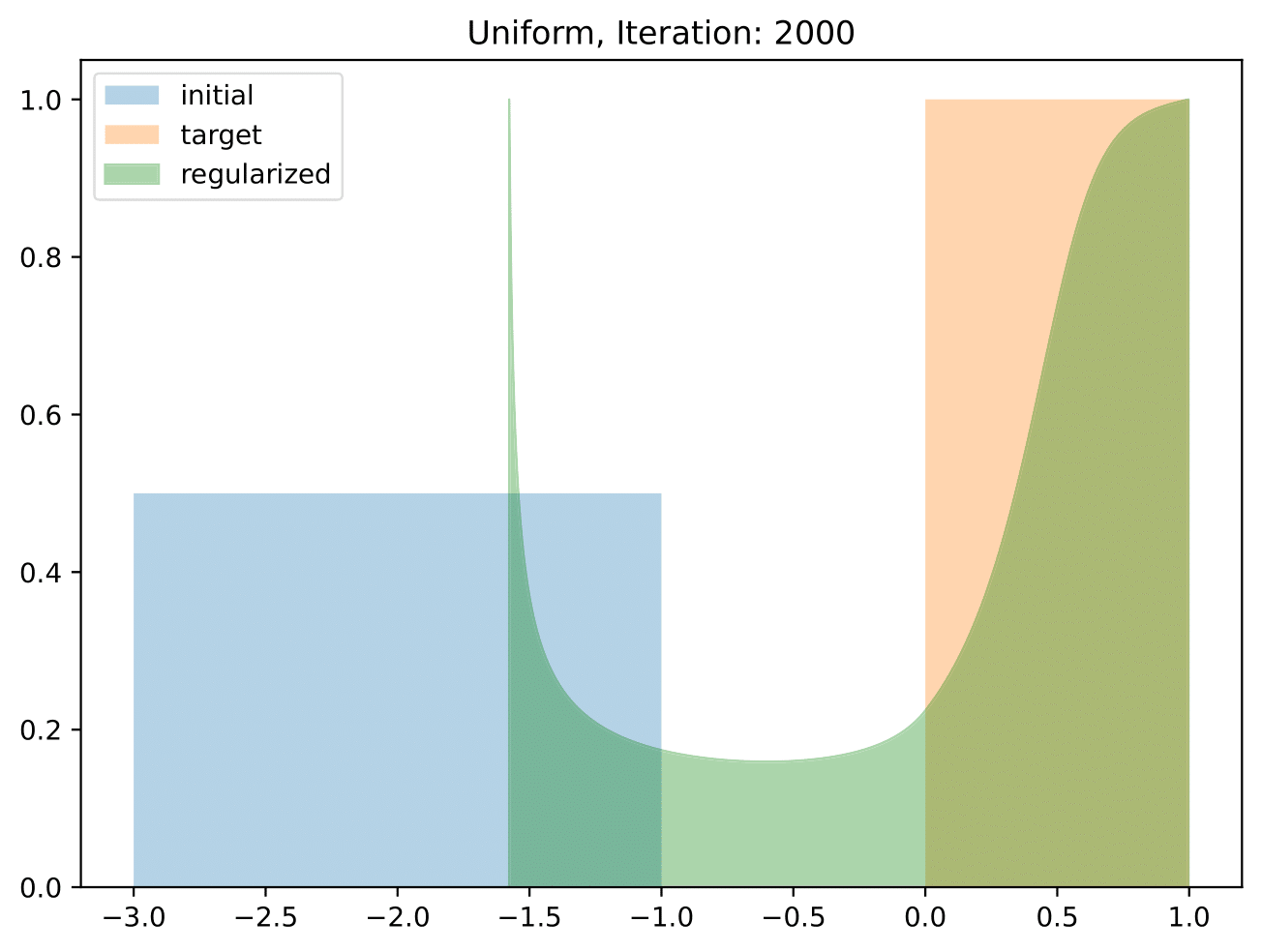}
    \includegraphics[width=.32\textwidth]{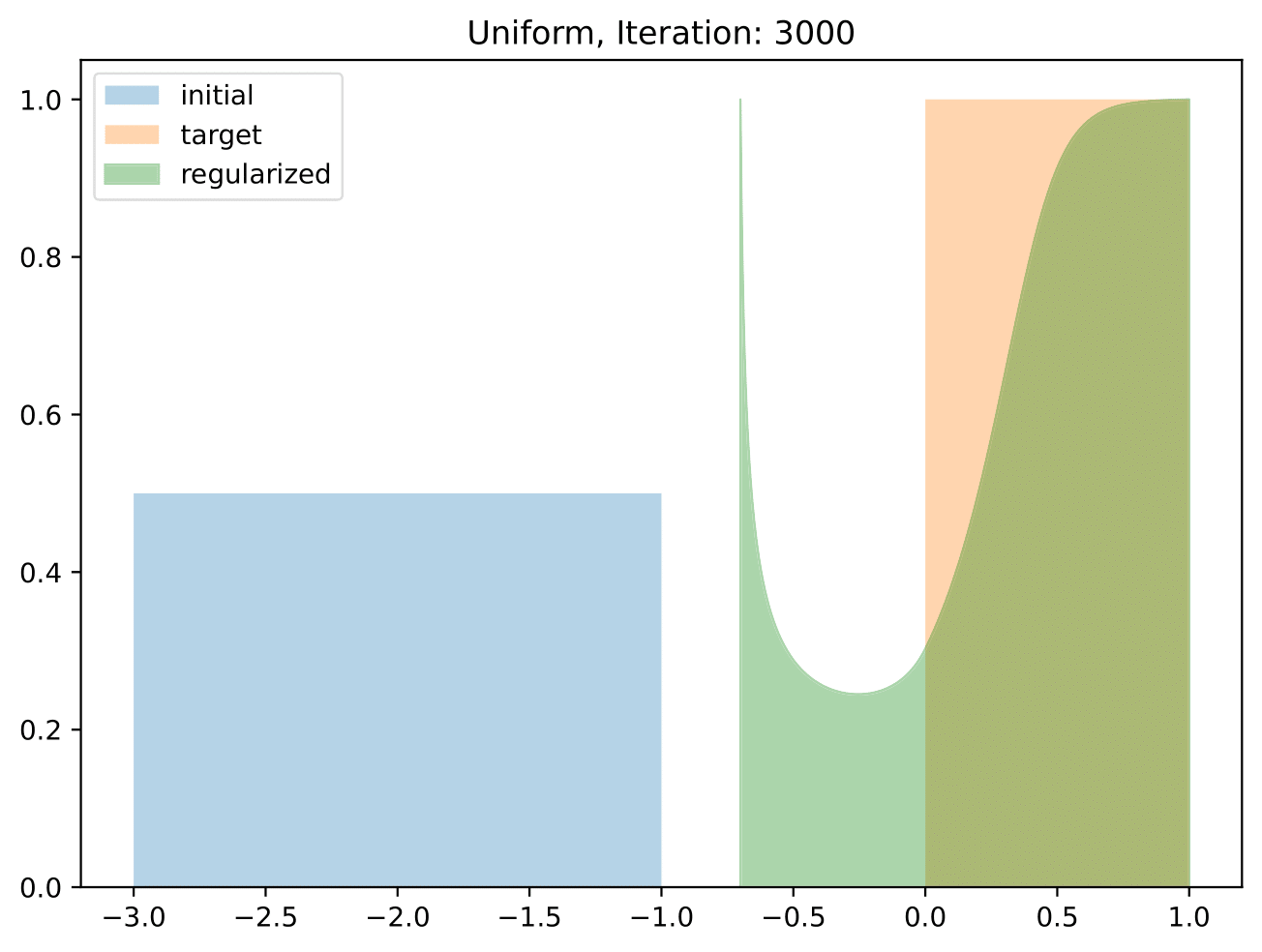}
    \includegraphics[width=.32\textwidth]{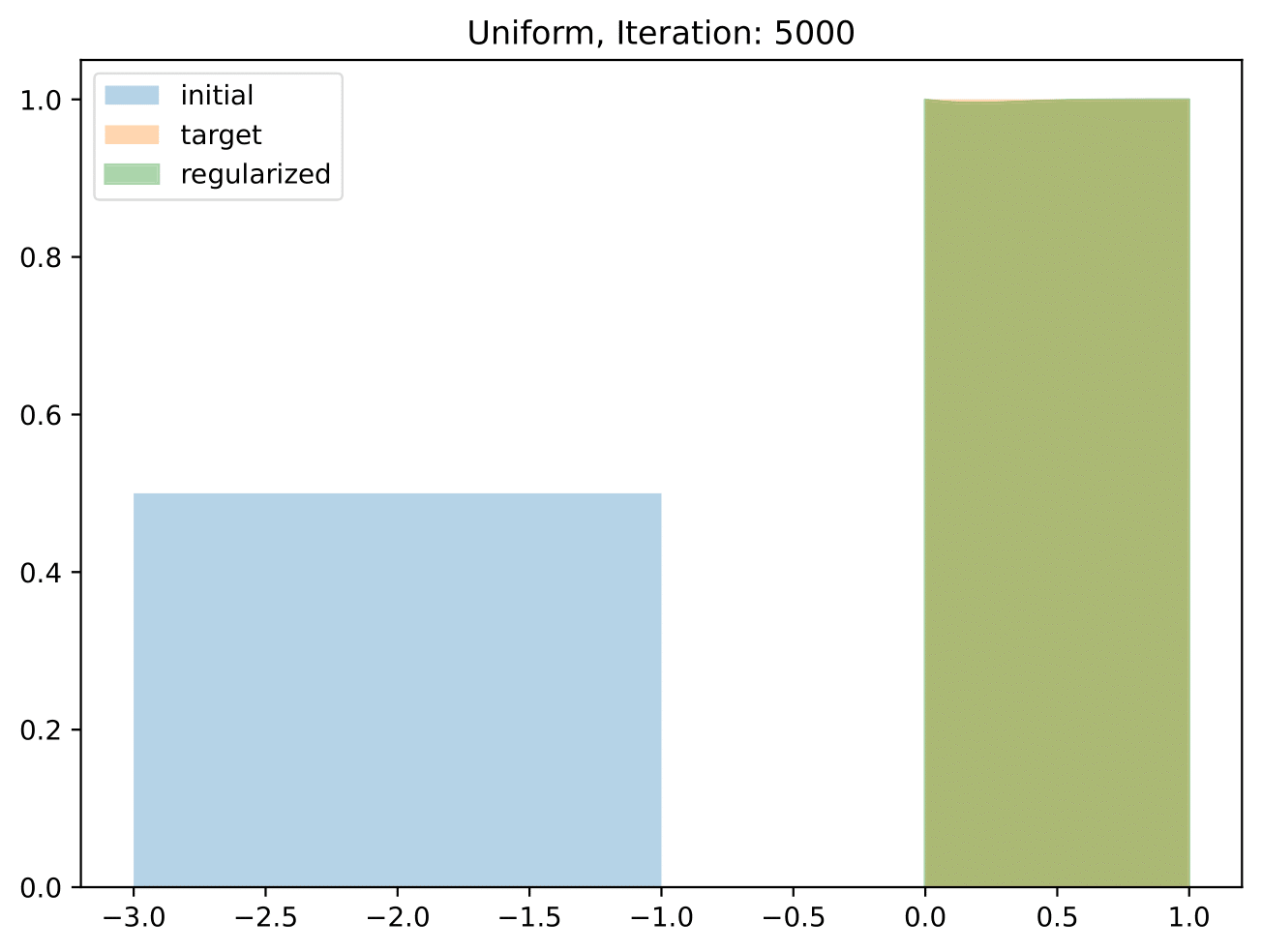}
    \caption{Regularized ${\mathcal F}_\nu^-$-flow  between the uniform measures $\gamma_0 \sim \mathcal{U}[-3,-1]$ 
		and $\nu \sim \mathcal{U}[0,1]$ for
		$\tau = 2\cdot 10^{-3}$, $\lambda = 10^{-2}$. The "horns" at the boundary of the support are of height $Q_\nu'(0)^{-1} = Q_\nu'(1)^{-1} = 1$, mirroring the target density values $f_\nu$ at the boundary of $\supp \nu$. The support shifts towards the target.
    }
    \label{fig:Uniform_to_Uniform_reg}
\end{figure}

\begin{figure}[H]
    \centering    
    \includegraphics[width=.318\textwidth]{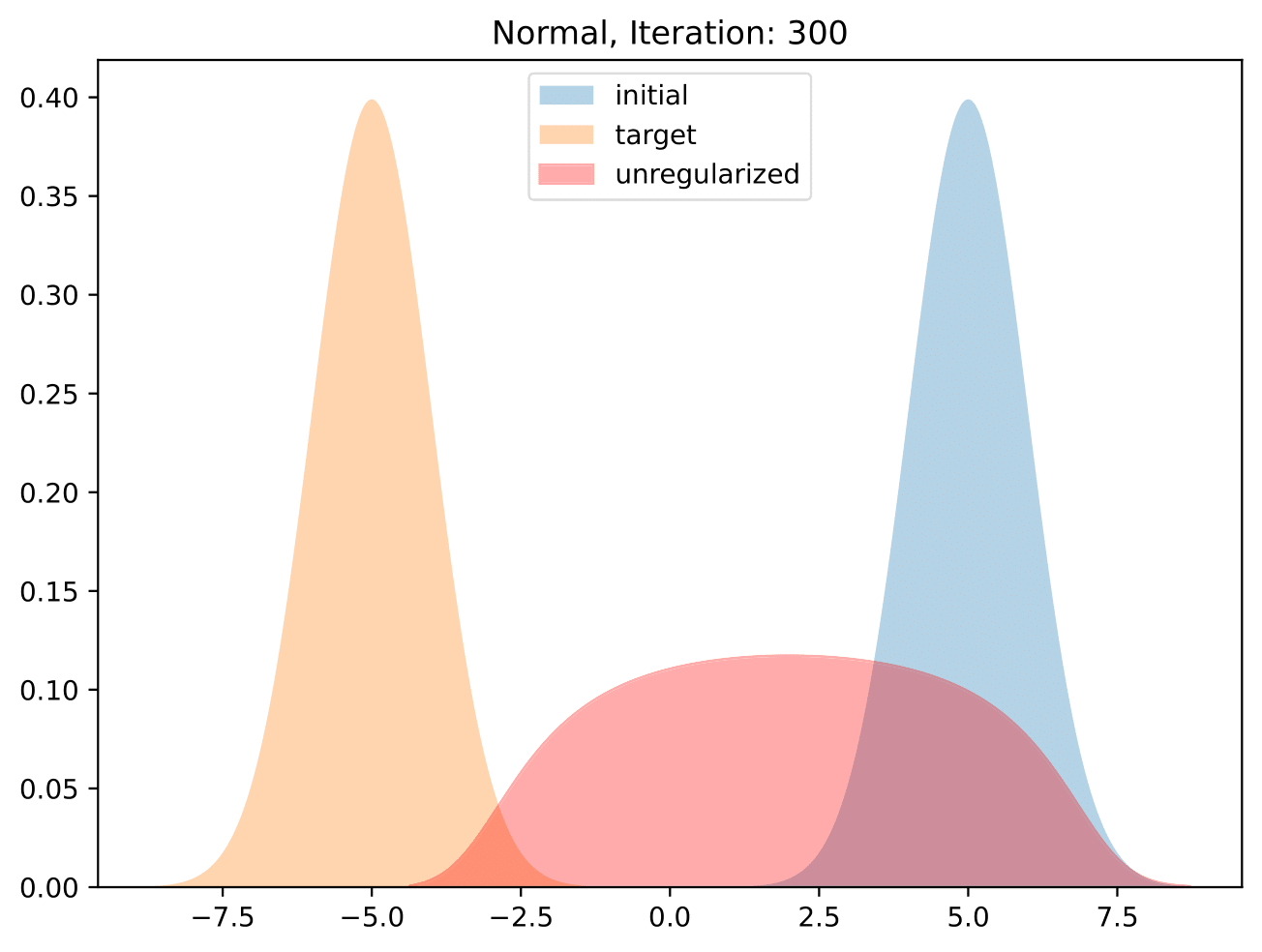}
    \includegraphics[width=.318\textwidth]{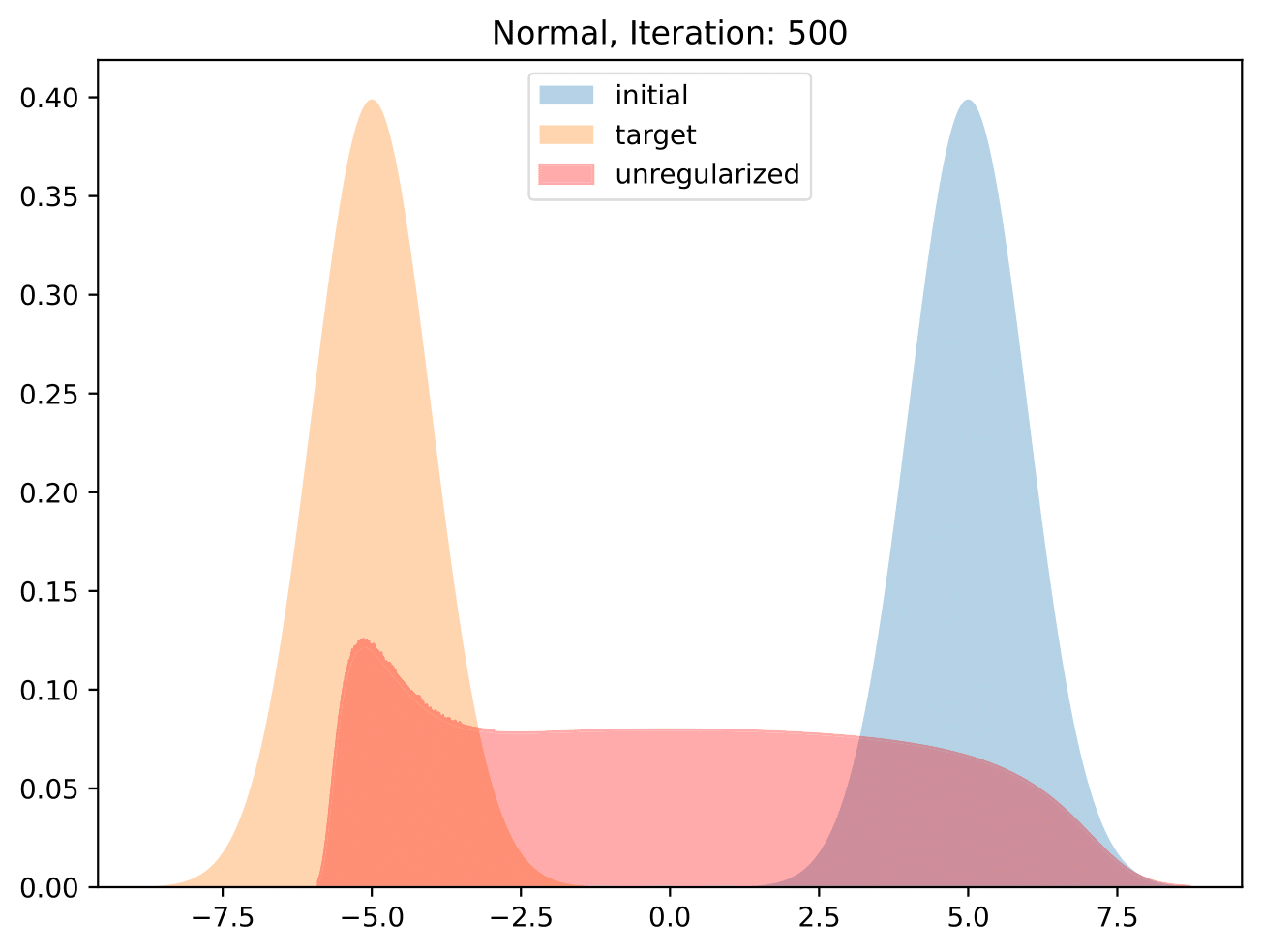}
    \includegraphics[width=.318\textwidth]{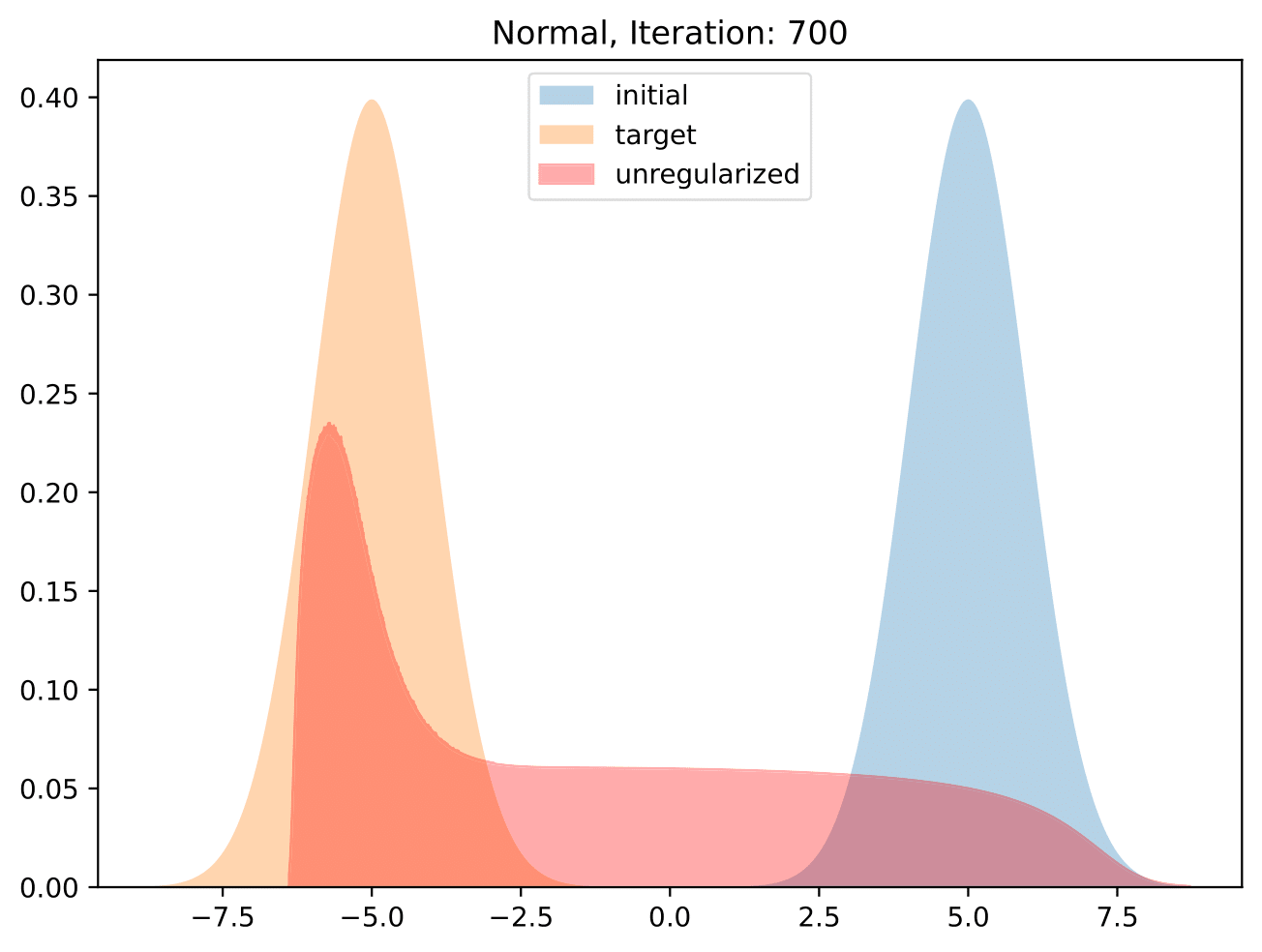}
    \\
    \includegraphics[width=.318\textwidth]{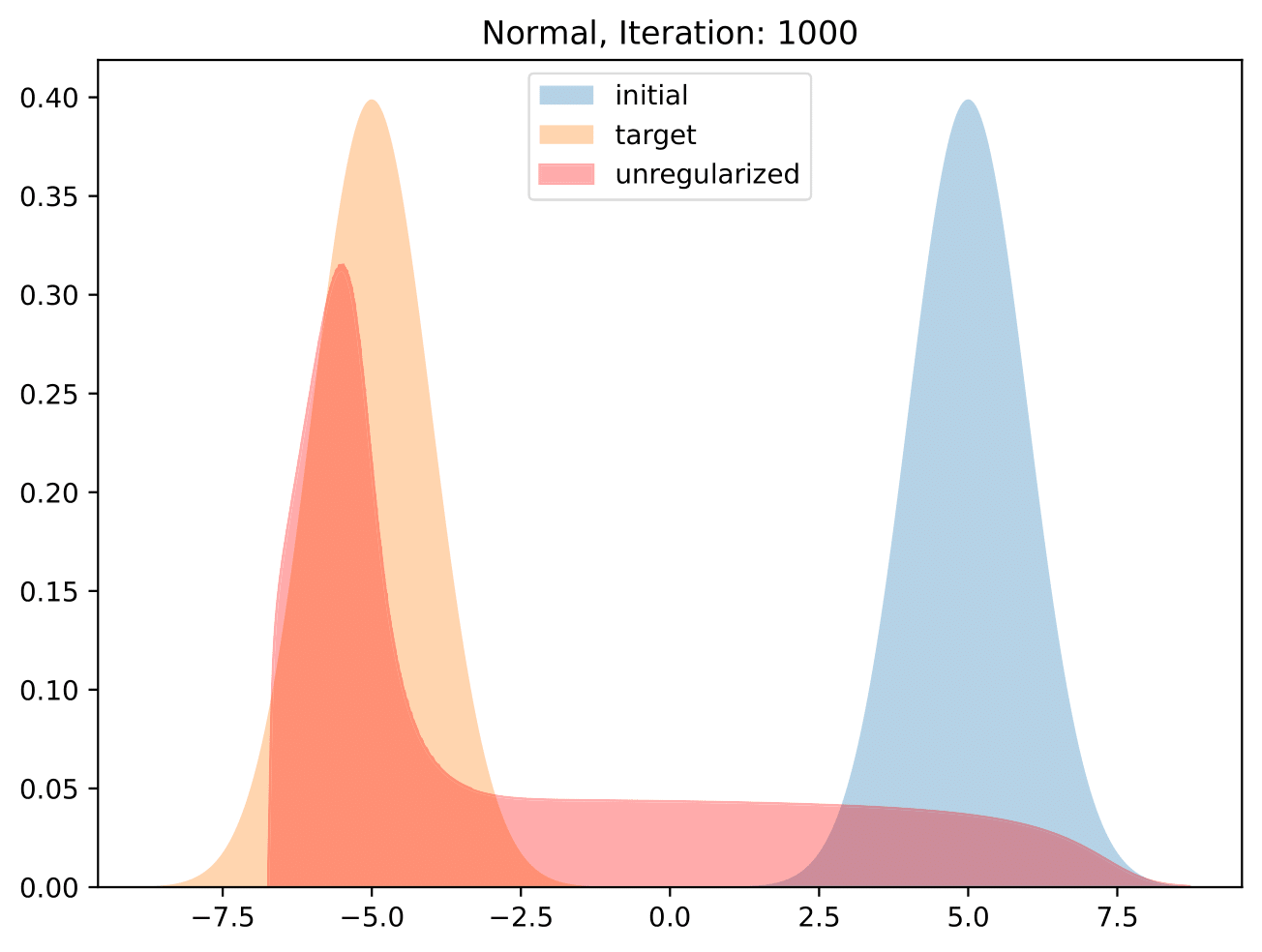}
    \includegraphics[width=.318\textwidth]{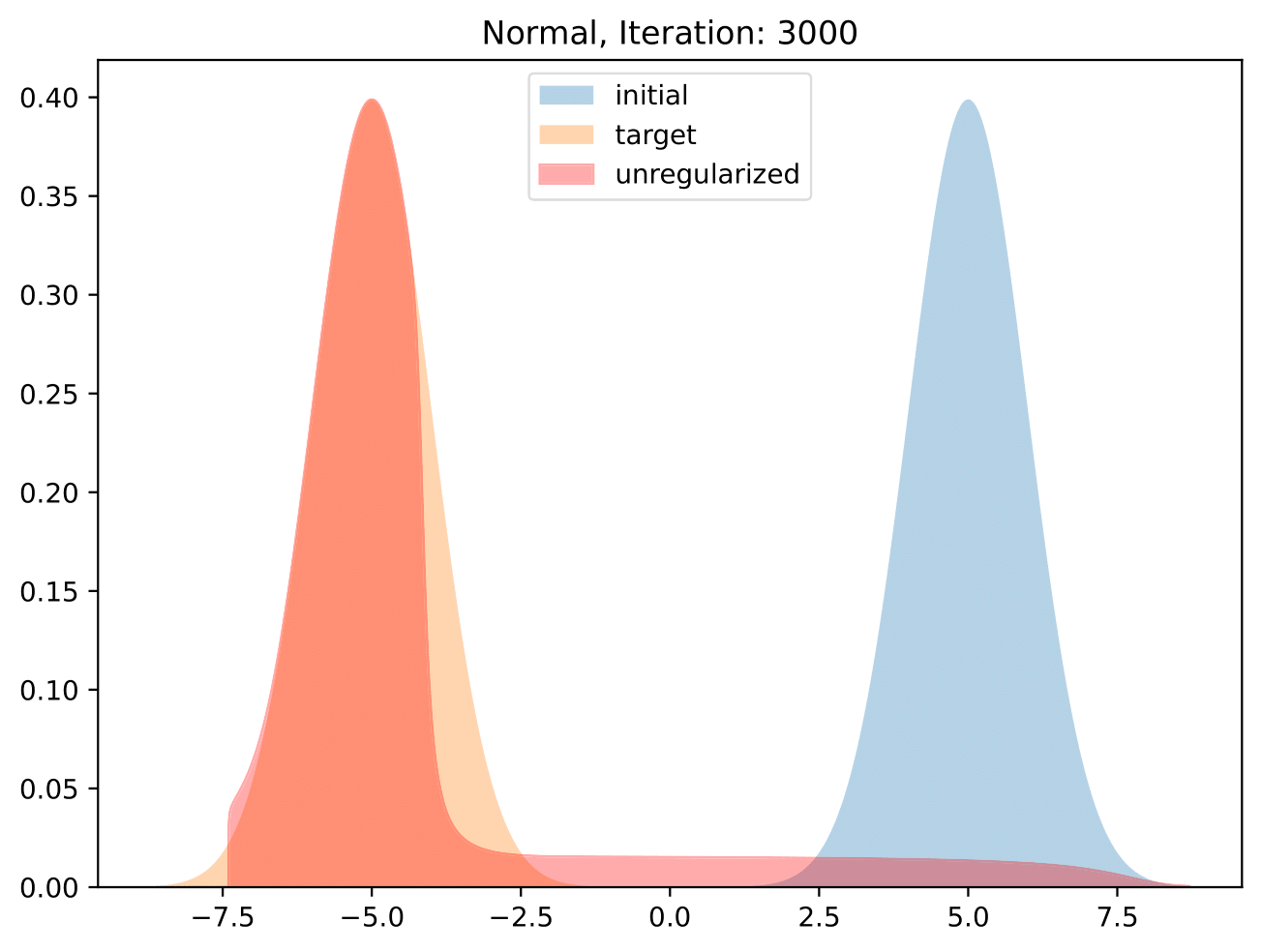}
    \includegraphics[width=.318\textwidth]{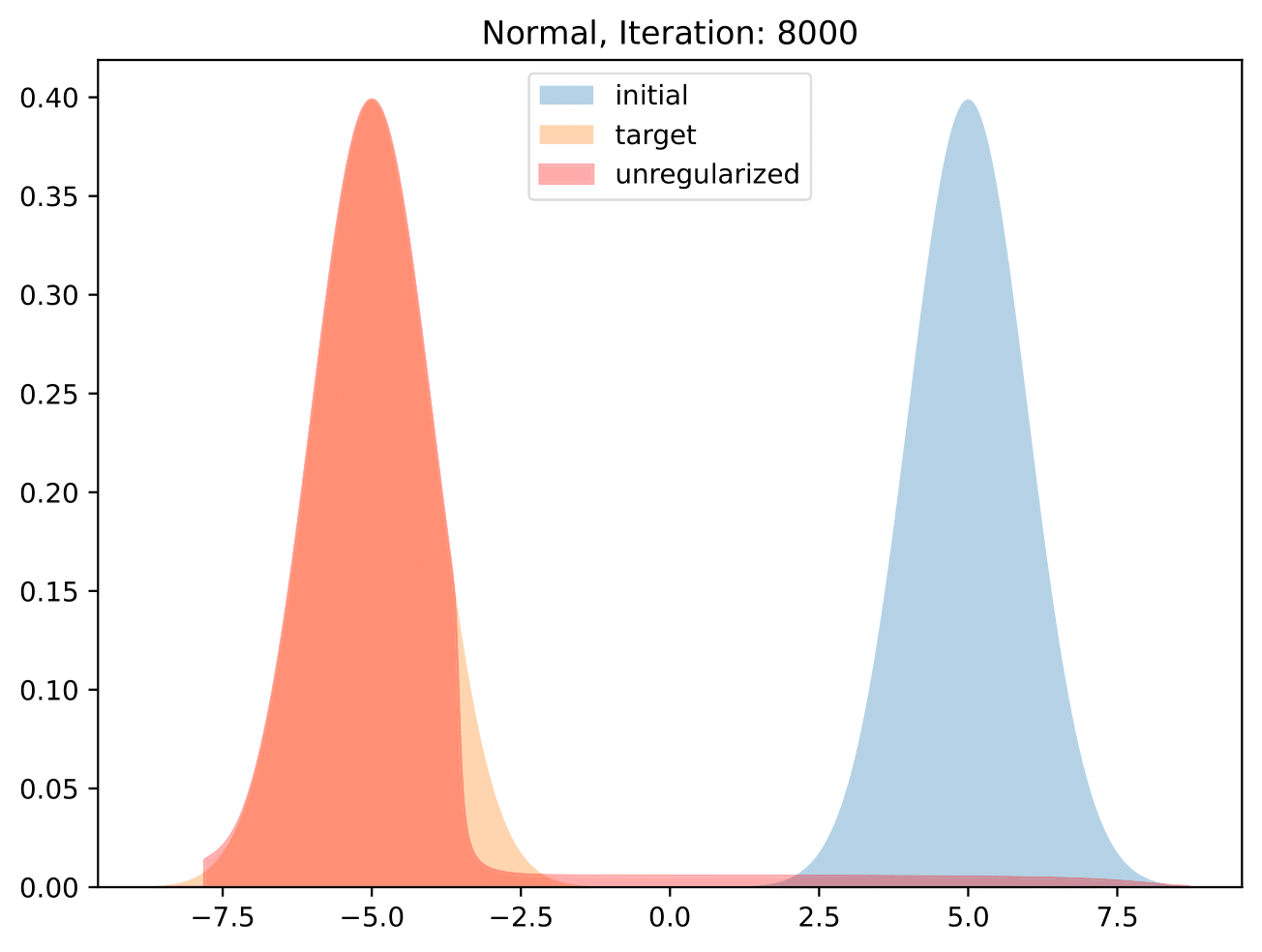}
    \caption{Unregularized ${\mathcal F}_\nu^-$-flow  between the cut-off Gaussians $\mathcal{N}(5,1)$ and $\mathcal{N}(-5,1)$ for
		$\tau = 10^{-2}$. The mass "sticks" to its once attained position.
    }
    \label{fig:Gaussian_to_Gaussian_unreg}
\end{figure}

\begin{figure}[H]
    \centering
    \includegraphics[width=.318\textwidth]{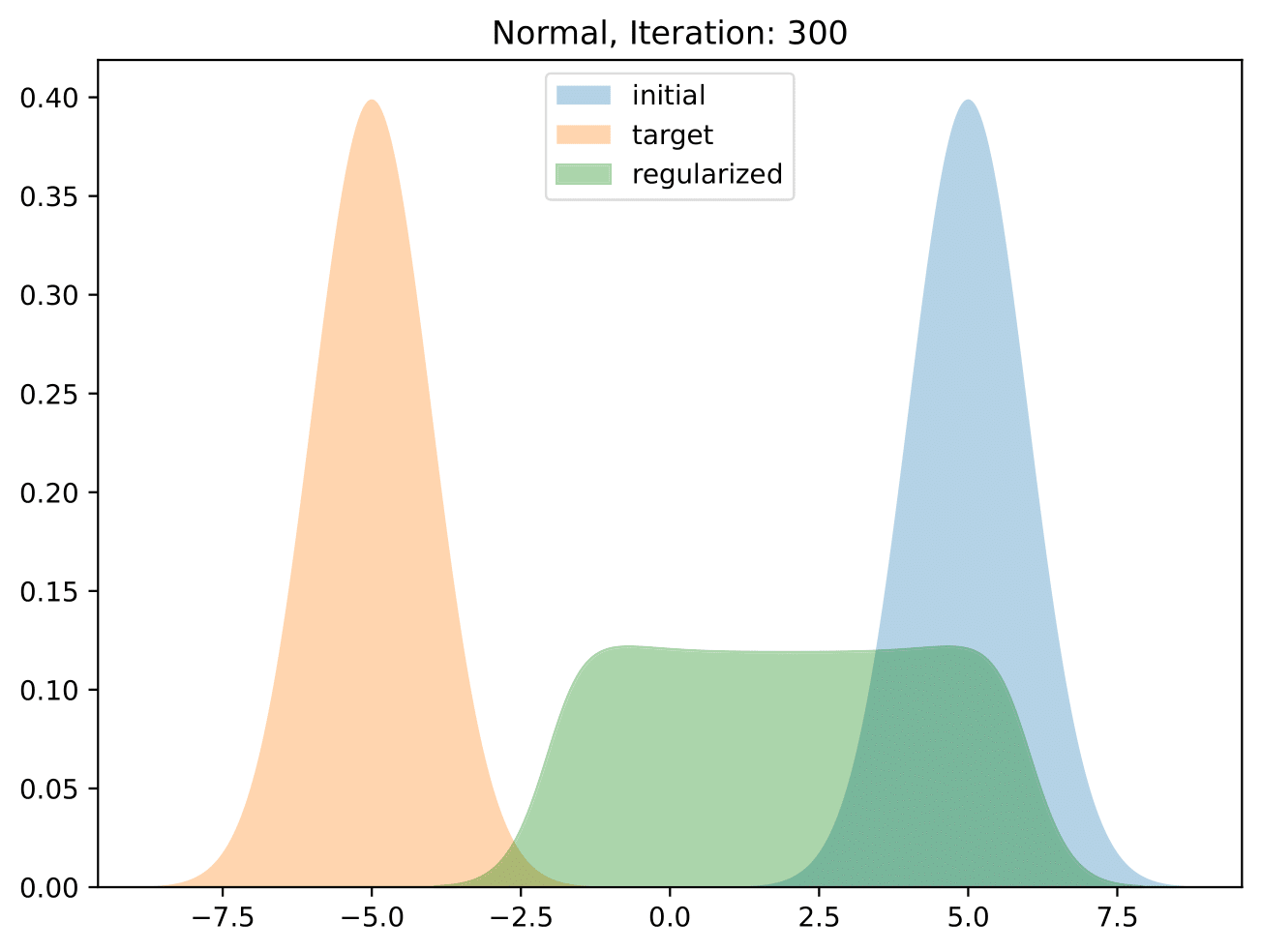}
    \includegraphics[width=.318\textwidth]{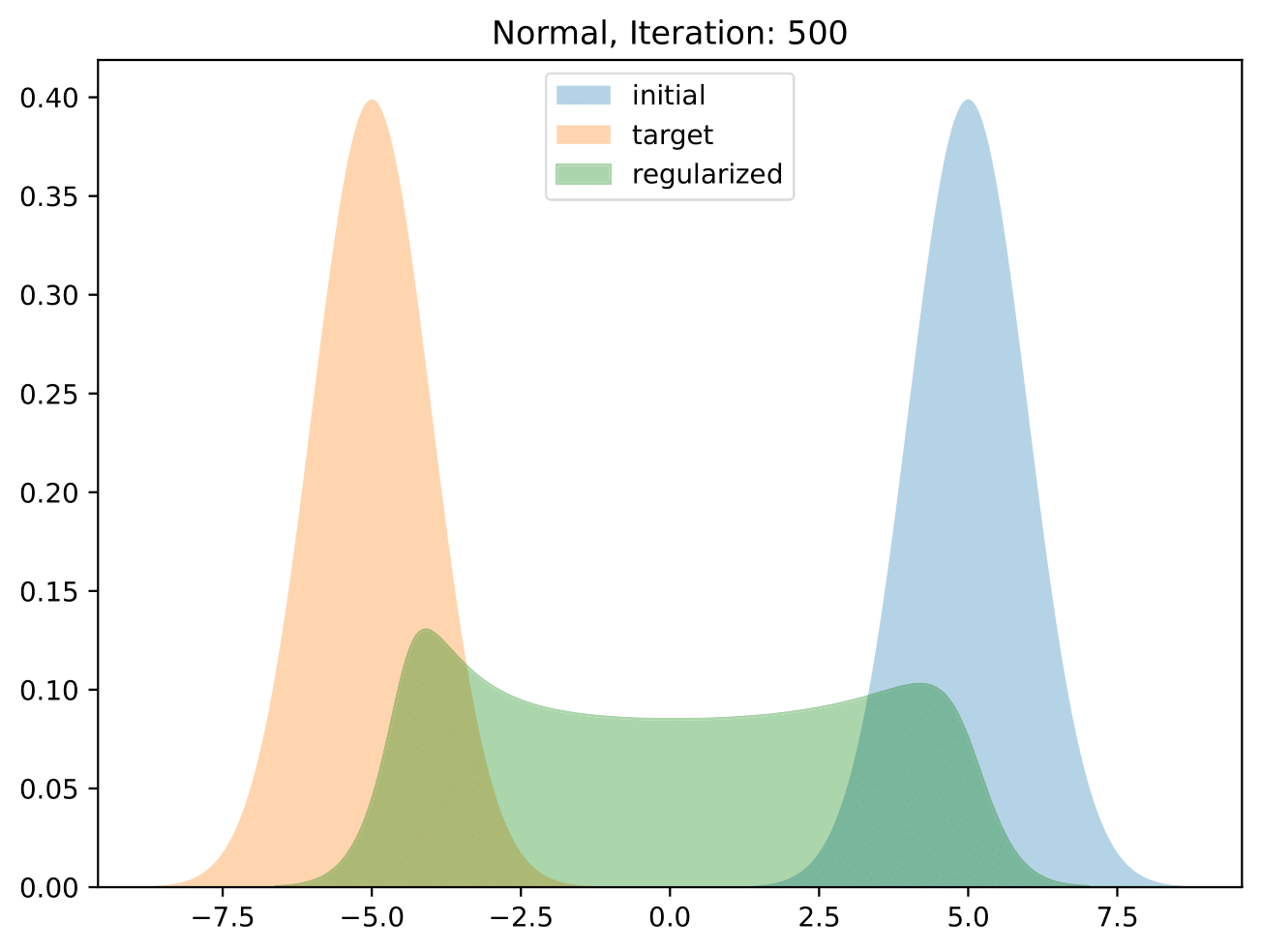}
    \includegraphics[width=.318\textwidth]{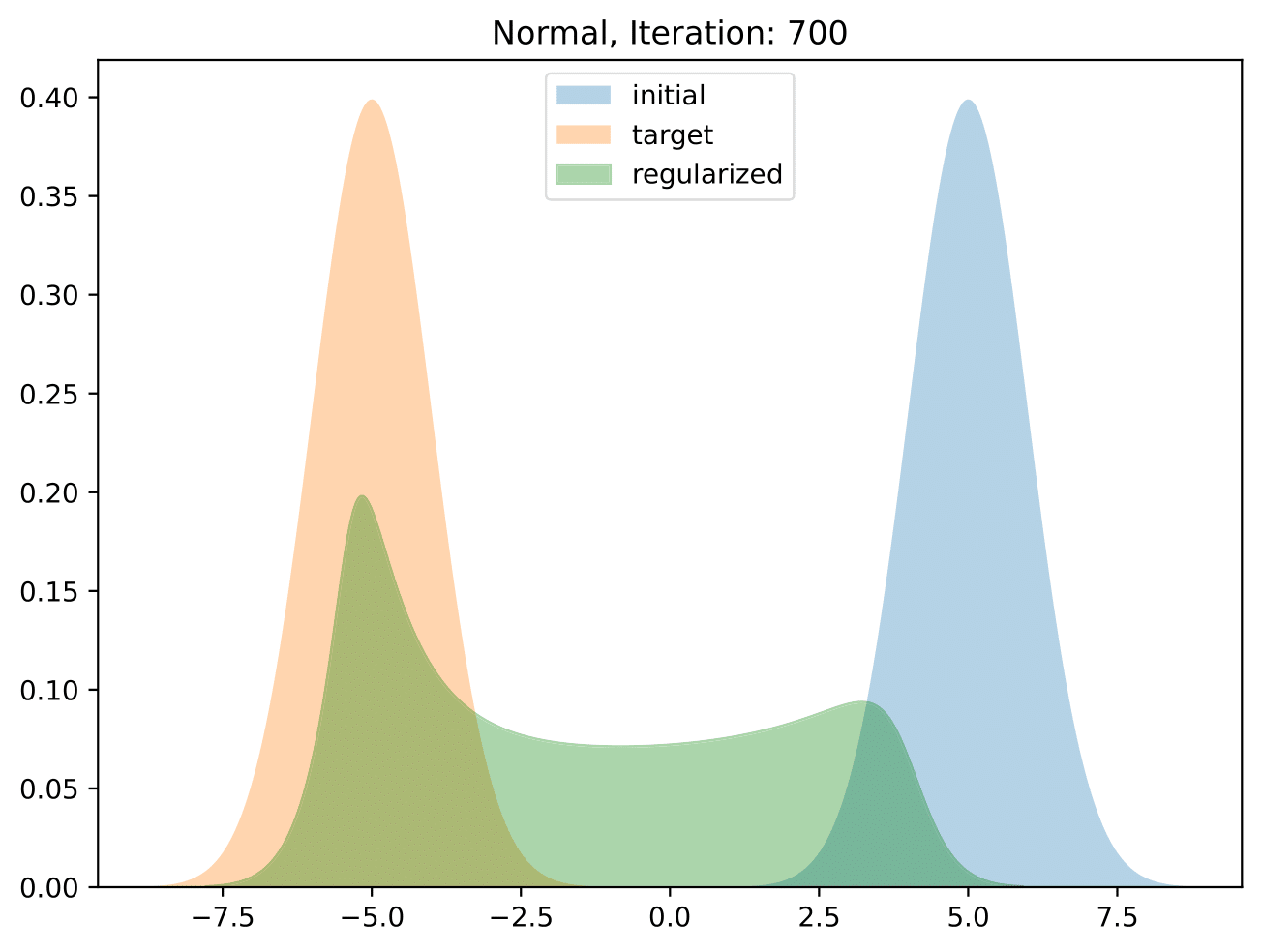}
    \\
    \includegraphics[width=.318\textwidth]{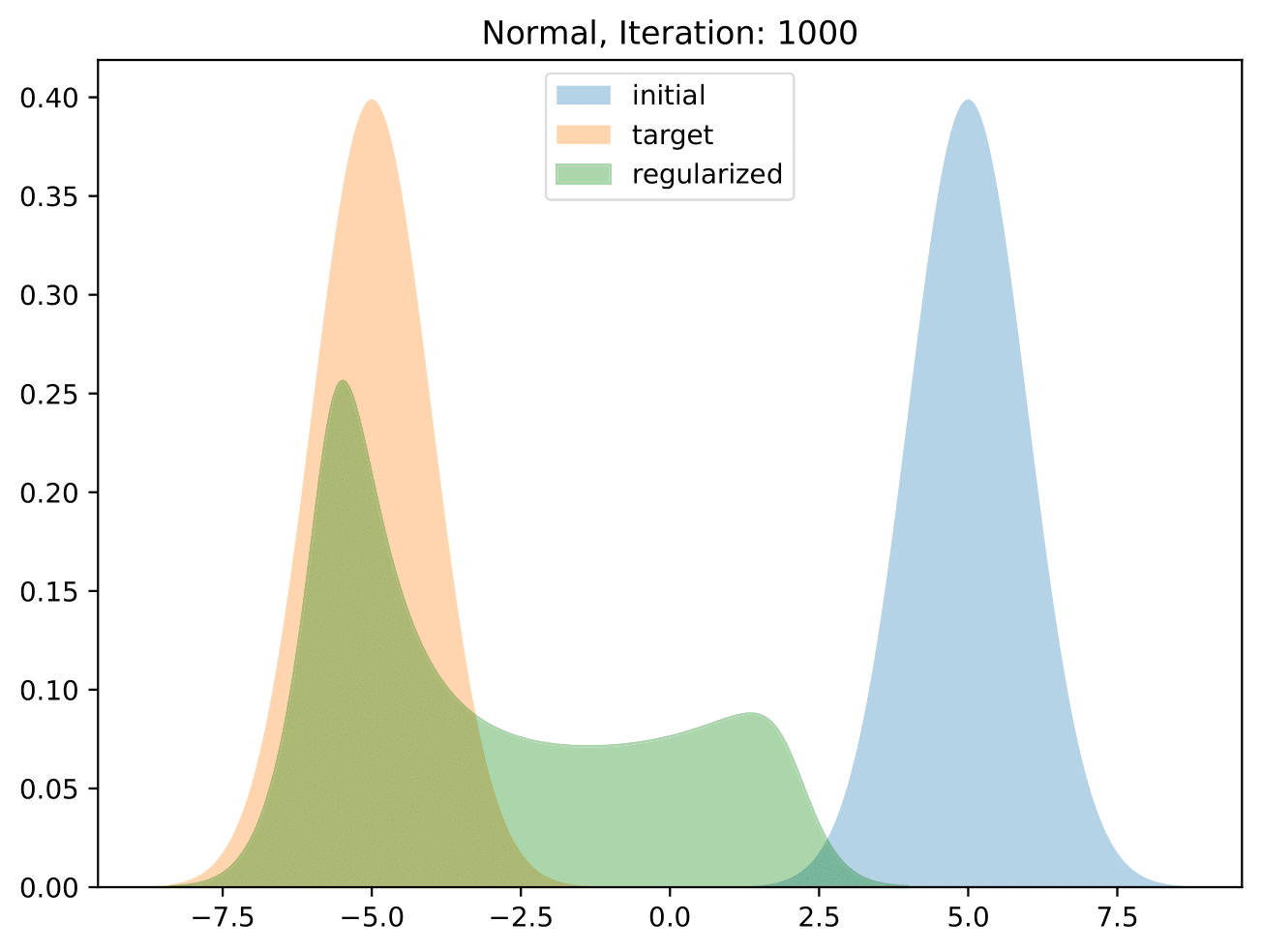}
    \includegraphics[width=.318\textwidth]{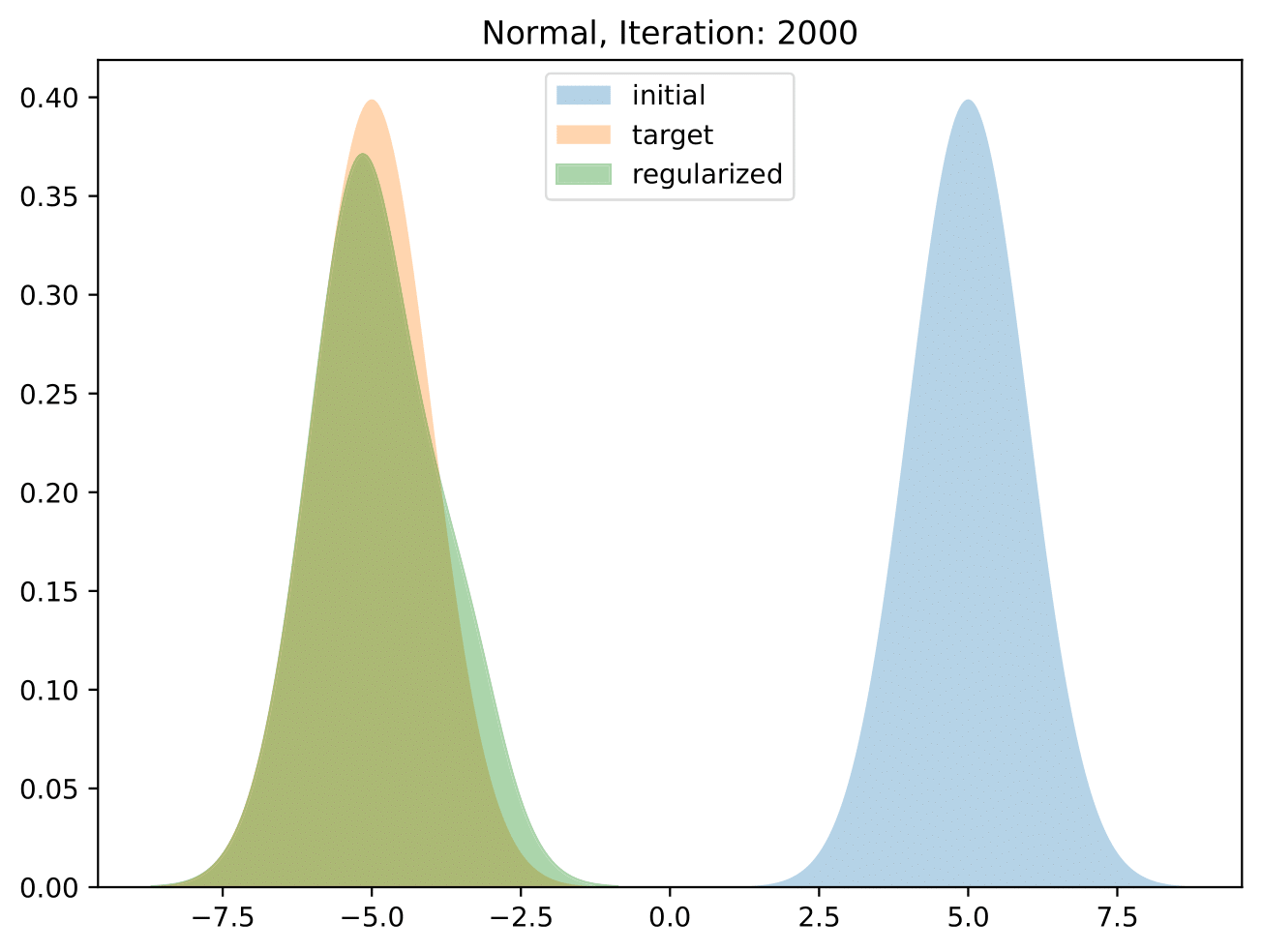}
    \includegraphics[width=.318\textwidth]{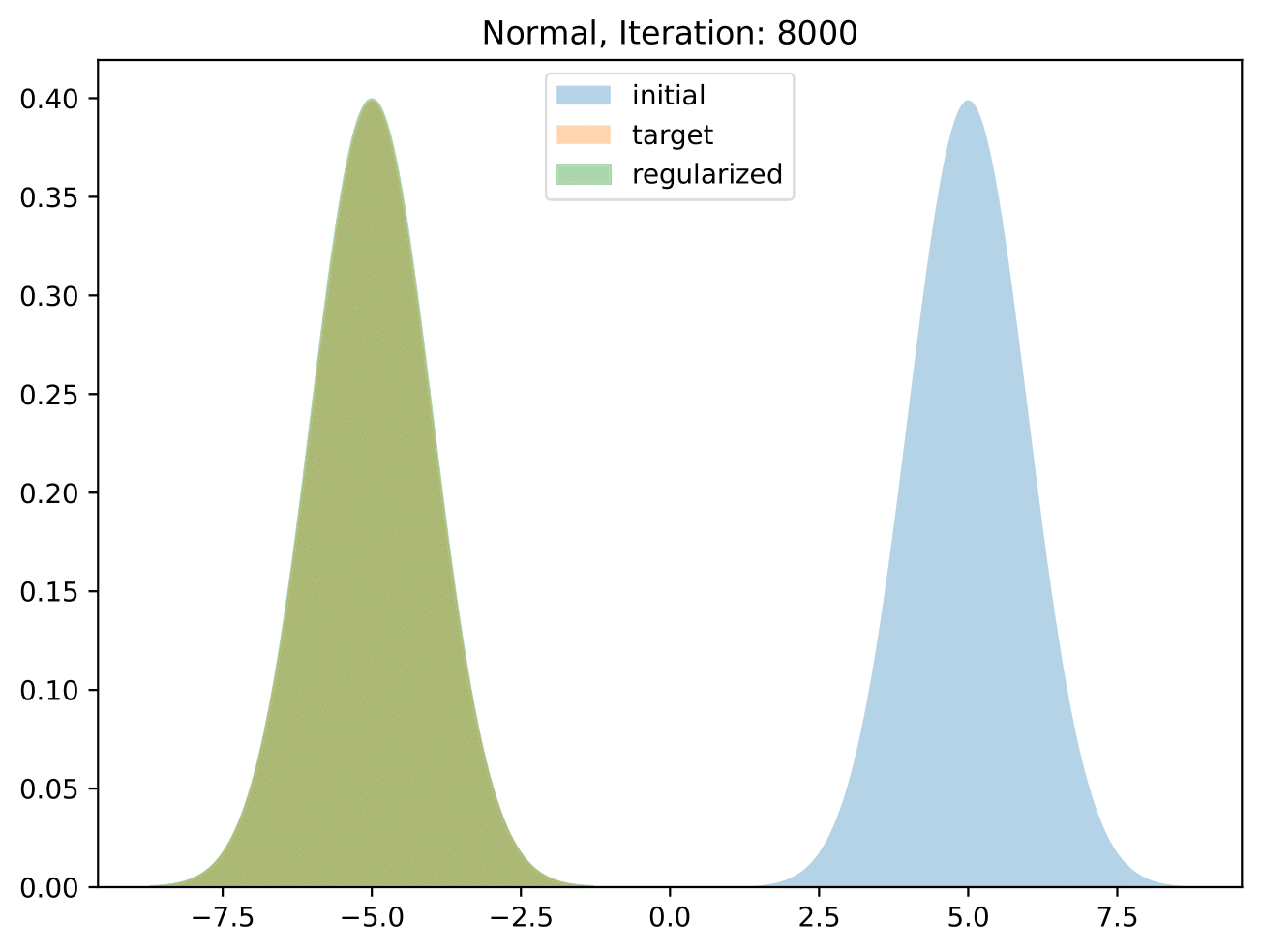}
    \caption{{Regularized ${\mathcal F}_\nu^-$-flow  between the cut-off Gaussians $\mathcal{N}(5,1)$ and $\mathcal{N}(-5,1)$ for
		$\tau =  \lambda = 10^{-2}$. Since for simplicity of implementation, we chose to only work with quantiles defined on a slightly smaller interval $[a,b] \subset (0,1)$}, the "horns" at the boundary are cut off. Still, the support visibly moves in direction of the target.
    }
    \label{fig:Gaussian_to_Gaussian_reg}
\end{figure}

\begin{figure}[H]
    \centering
    \includegraphics[width=.32\textwidth]{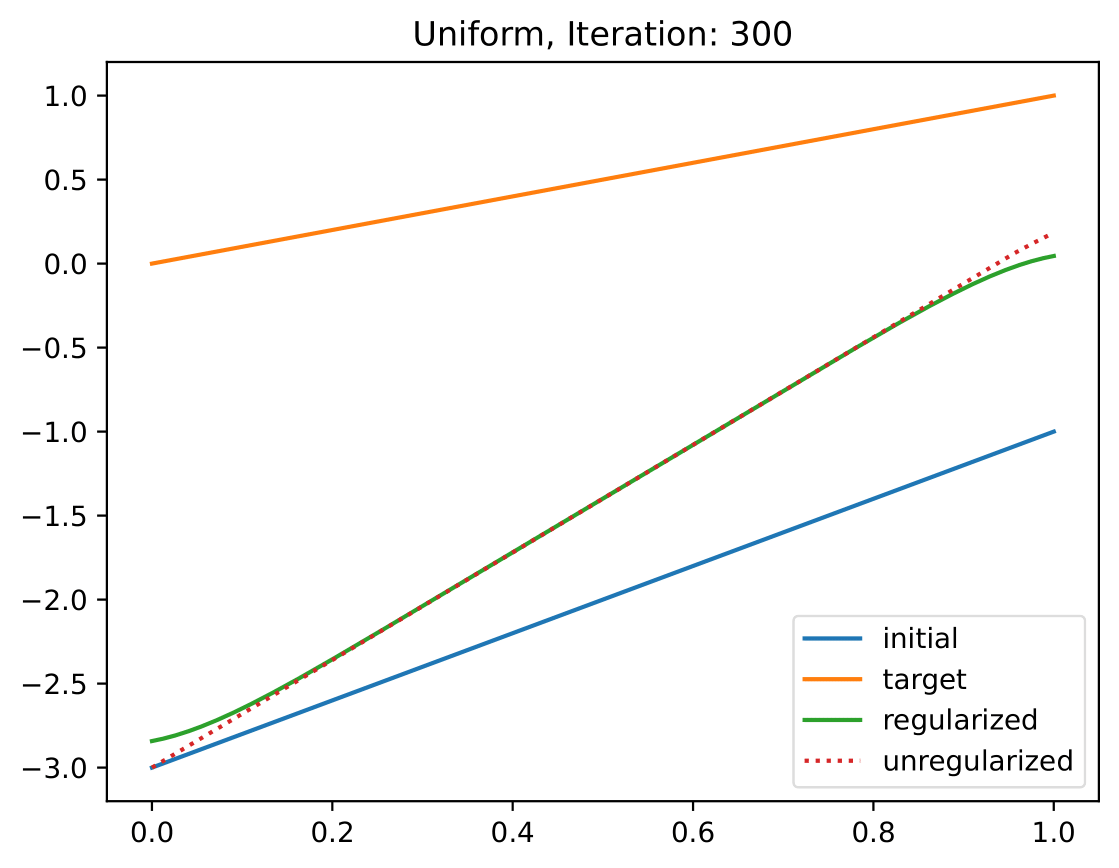}
    \includegraphics[width=.32\textwidth]{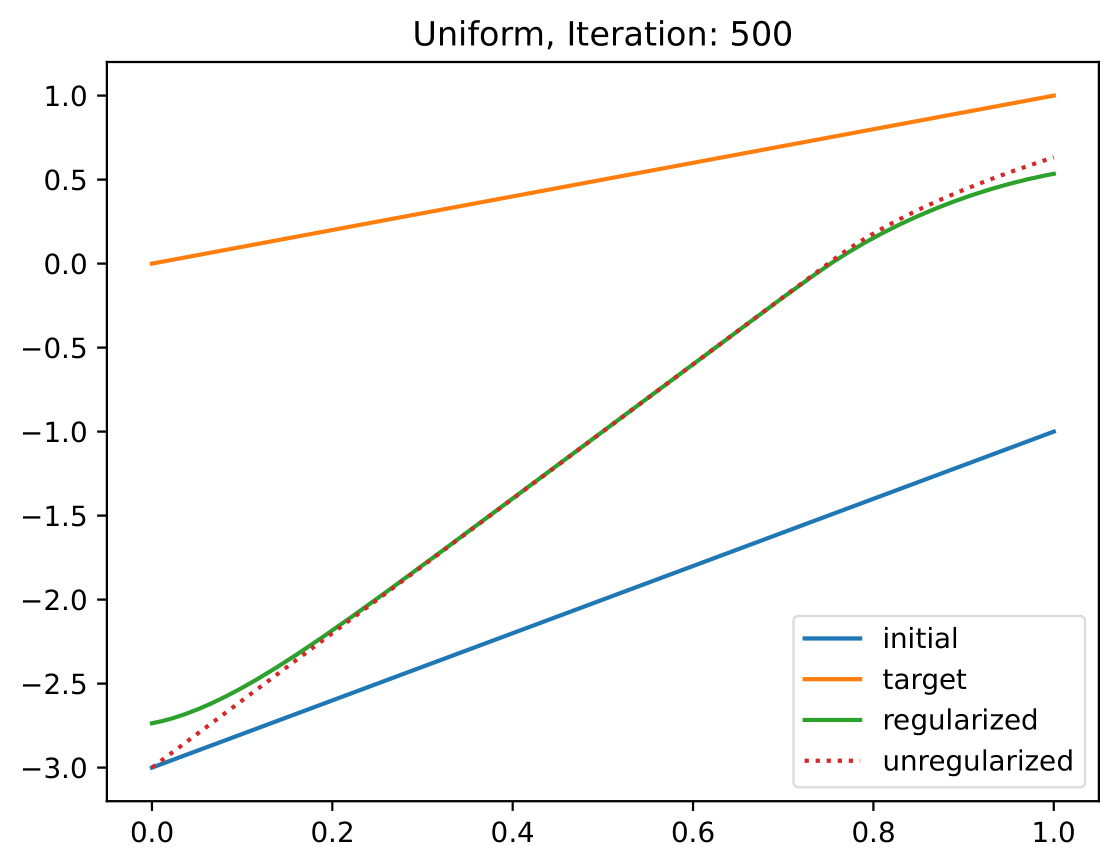}
    \includegraphics[width=.32\textwidth]{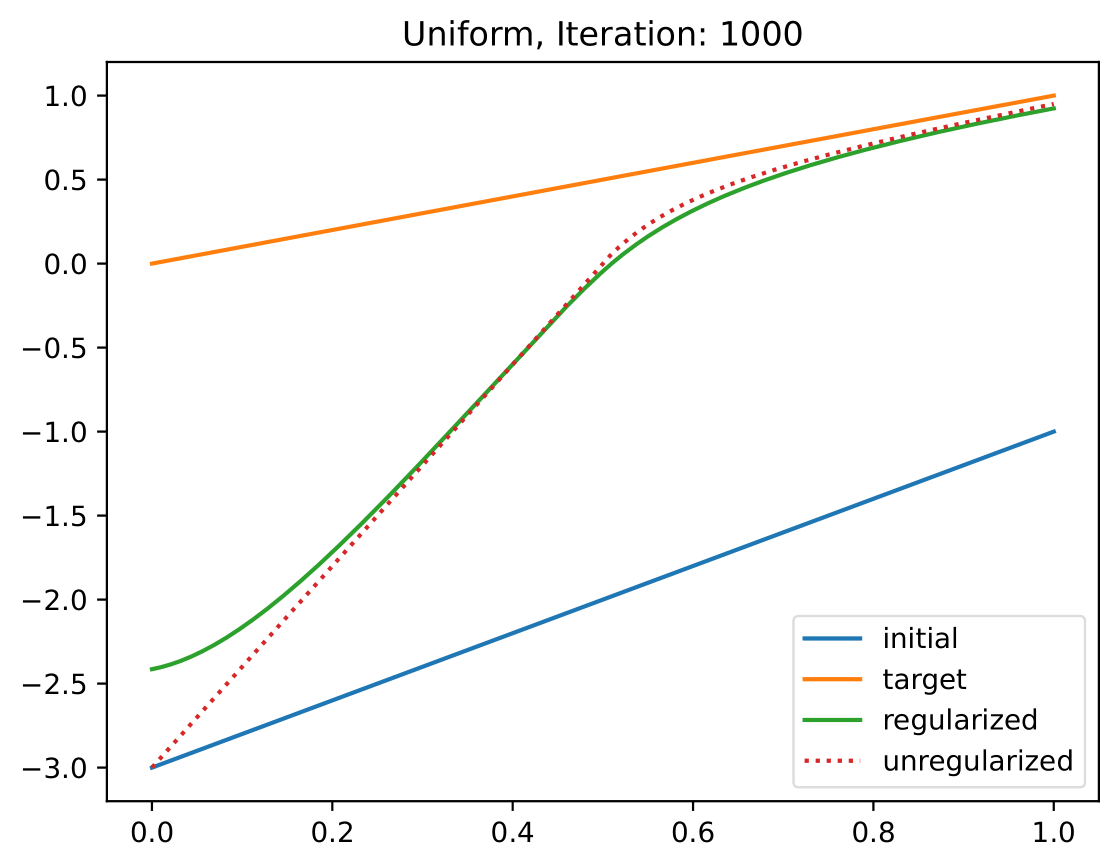}
    \\
    \includegraphics[width=.32\textwidth]{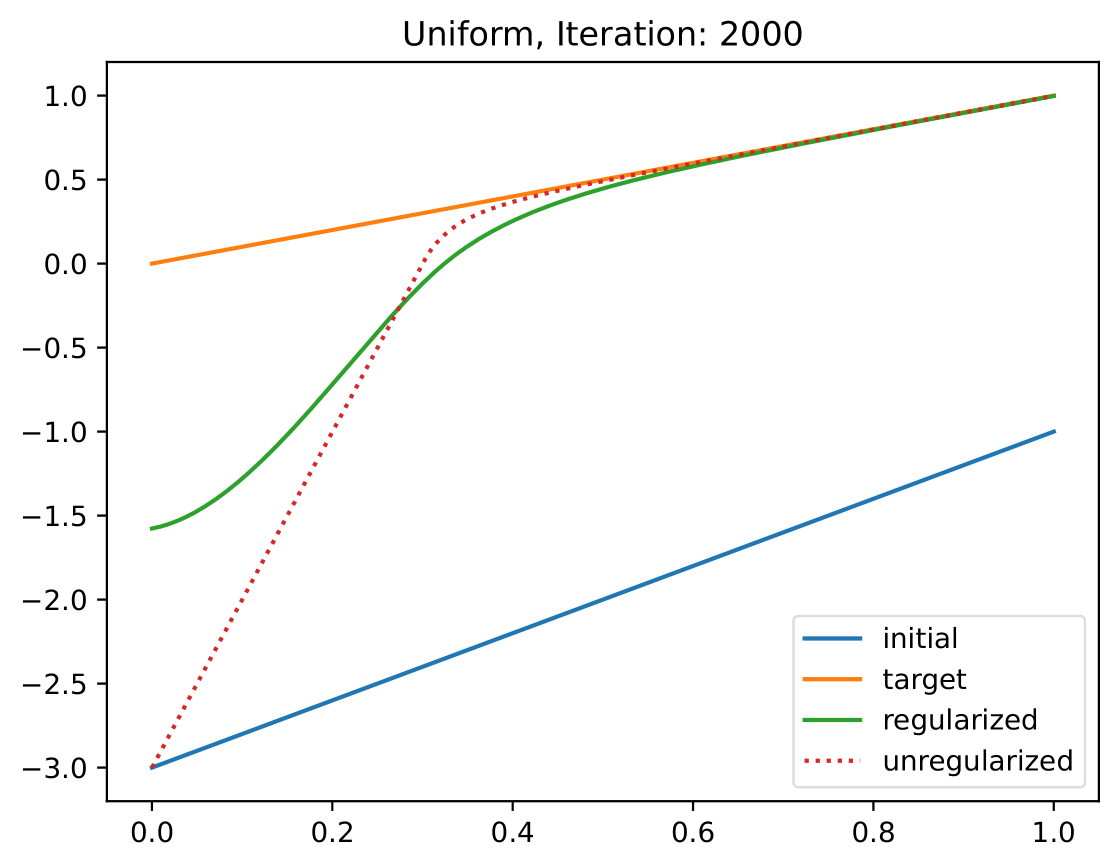}
    \includegraphics[width=.32\textwidth]{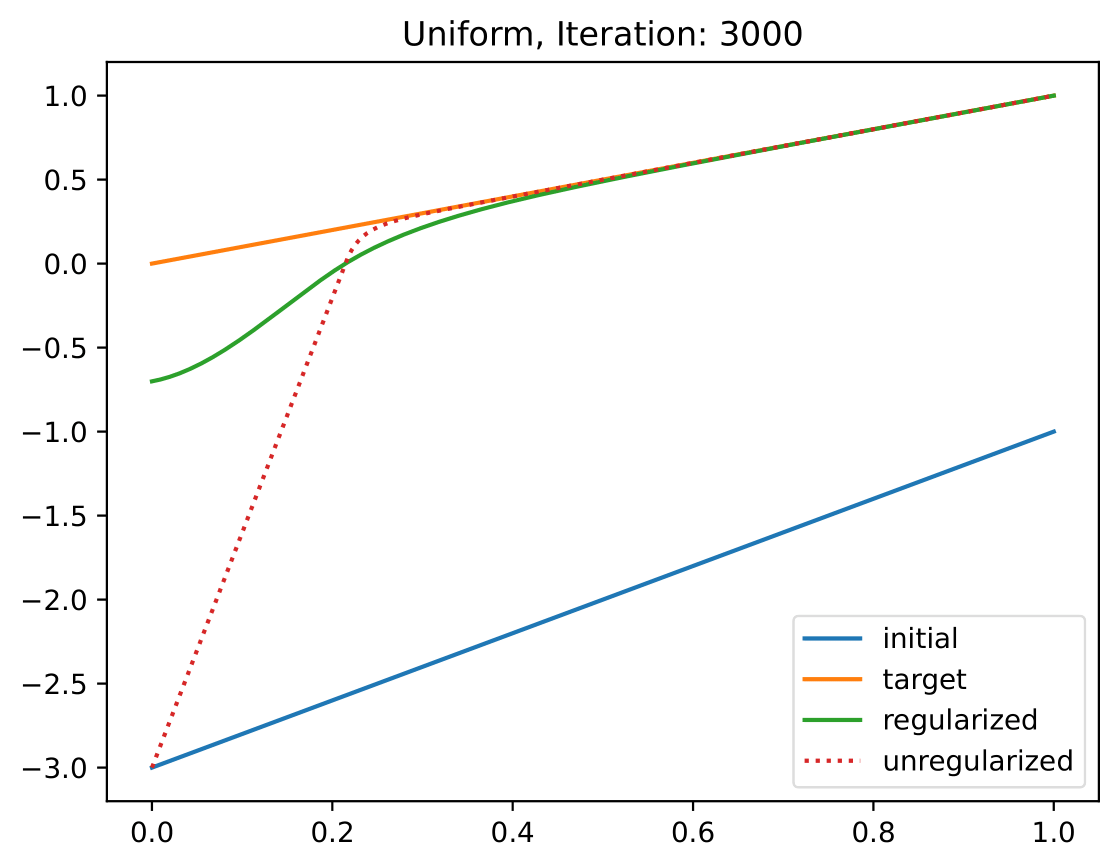}
    \includegraphics[width=.32\textwidth]{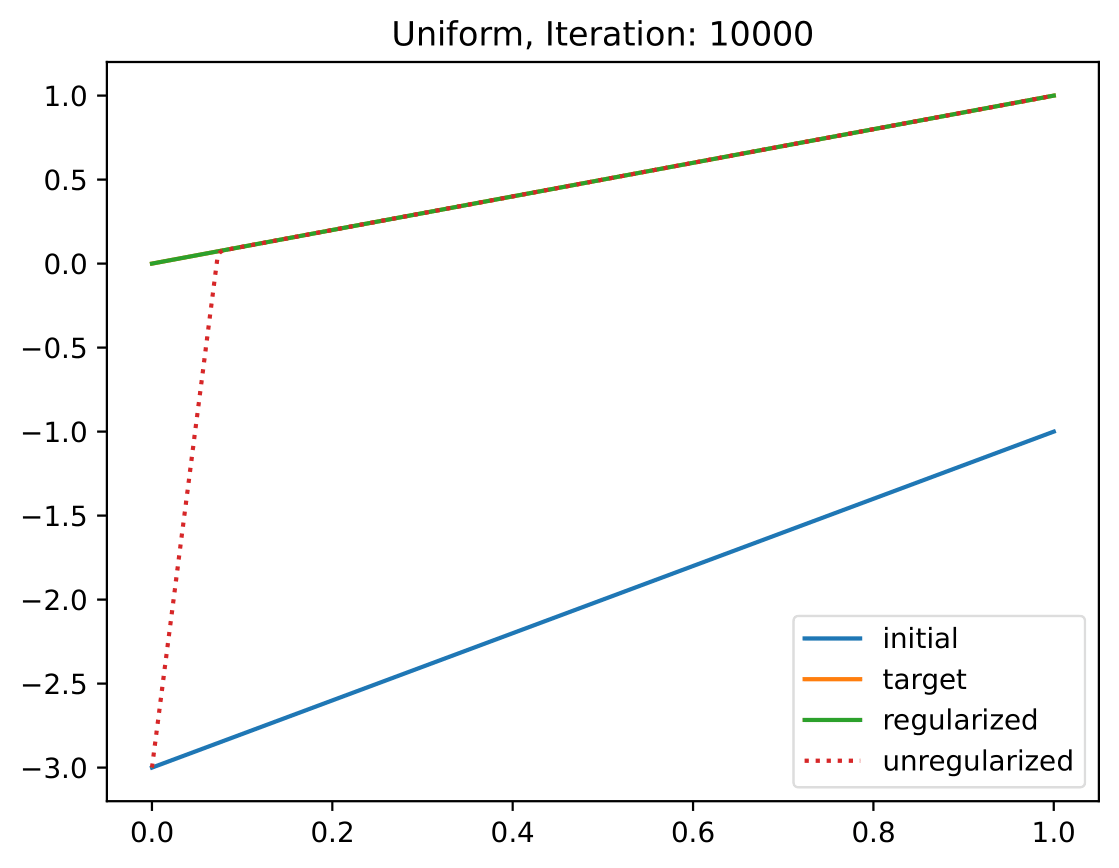}
    \caption{Quantiles of the regularized 
    and unregularized ${F}_\nu^-$-flow between uniform measures from $Q_0(s) = 2s-3$ to $Q_\nu(s) = s$
    with $\tau = 2\cdot 10^{-3}, \,\lambda = 10^{-2}$.
    }
    \label{fig:Uniform_to_Uniform_quantiles}
\end{figure}

\begin{figure}[H]
    \centering
    \includegraphics[width=.32\textwidth]{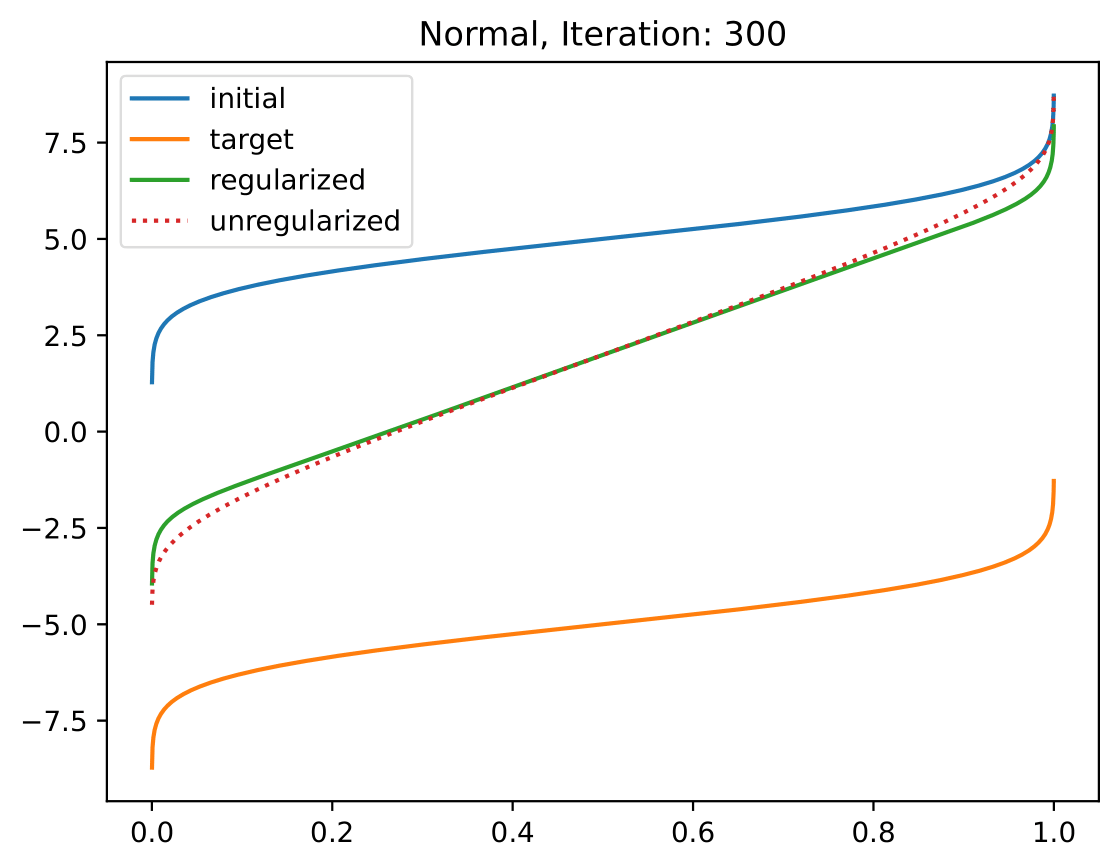}
    \includegraphics[width=.32\textwidth]{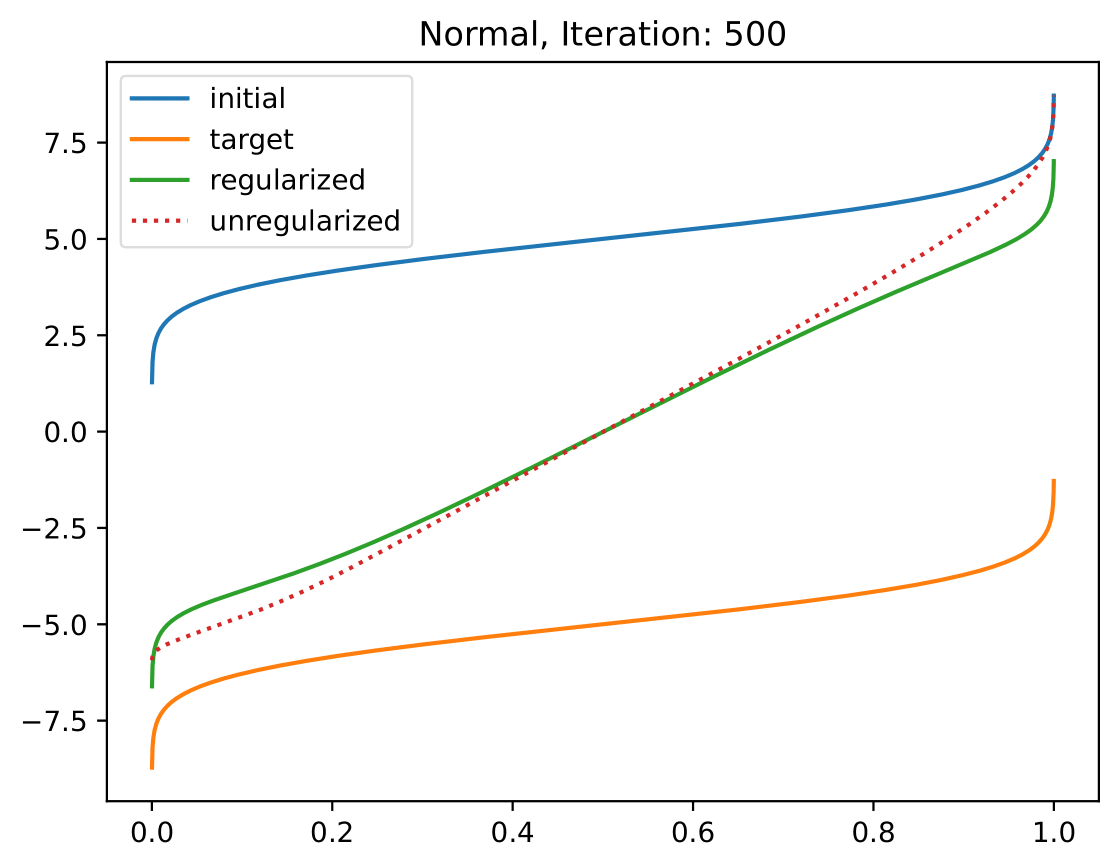}
    \includegraphics[width=.32\textwidth]{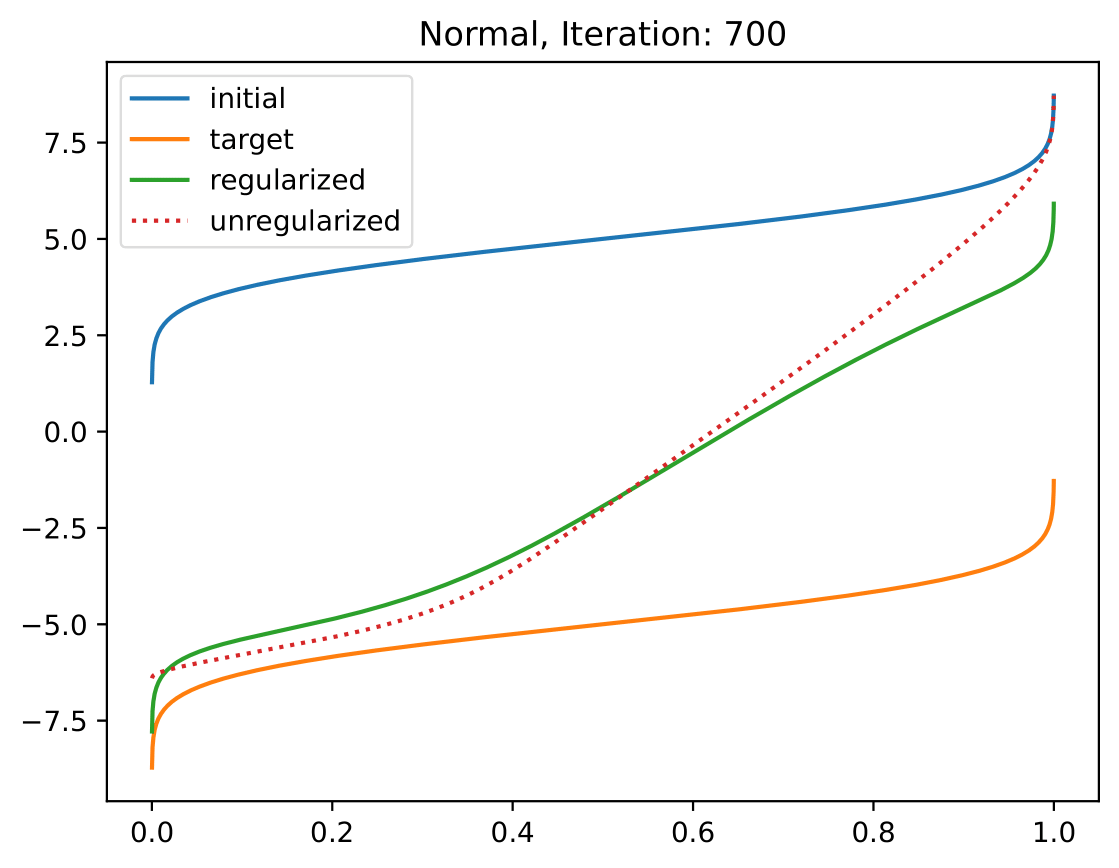}
    \\
    \includegraphics[width=.32\textwidth]{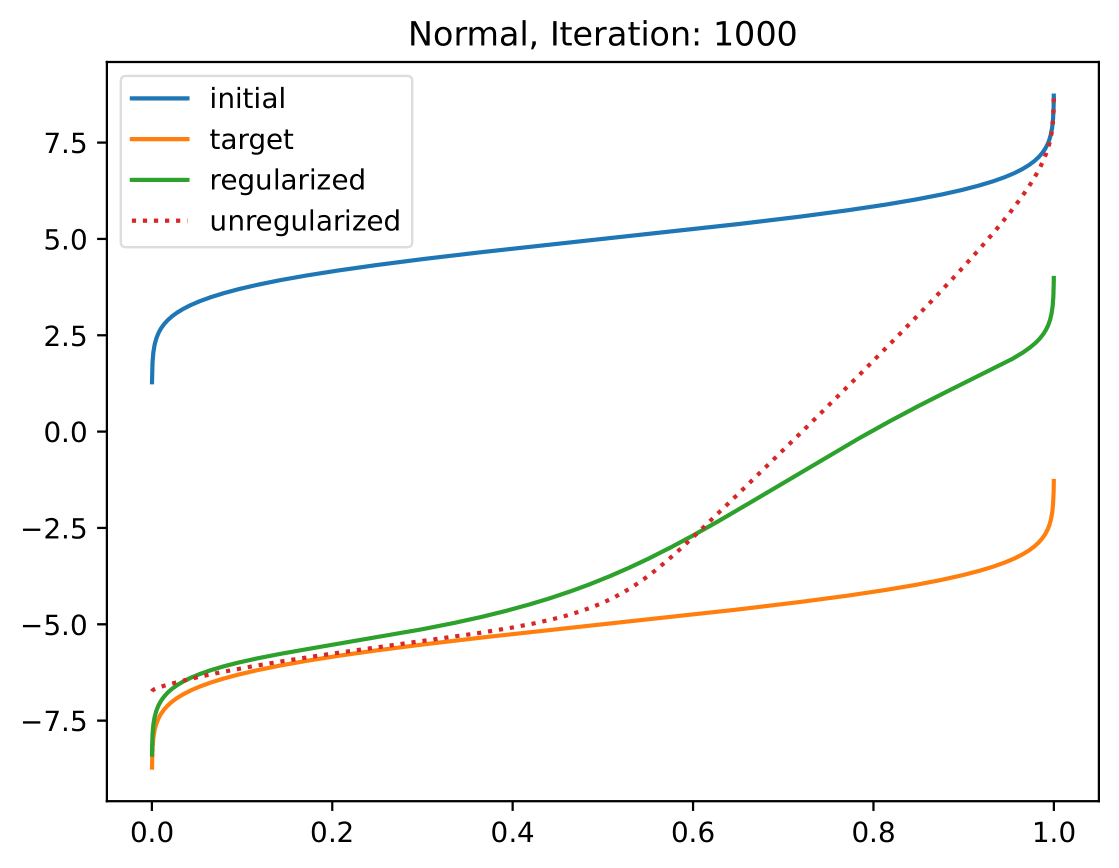}
    \includegraphics[width=.32\textwidth]{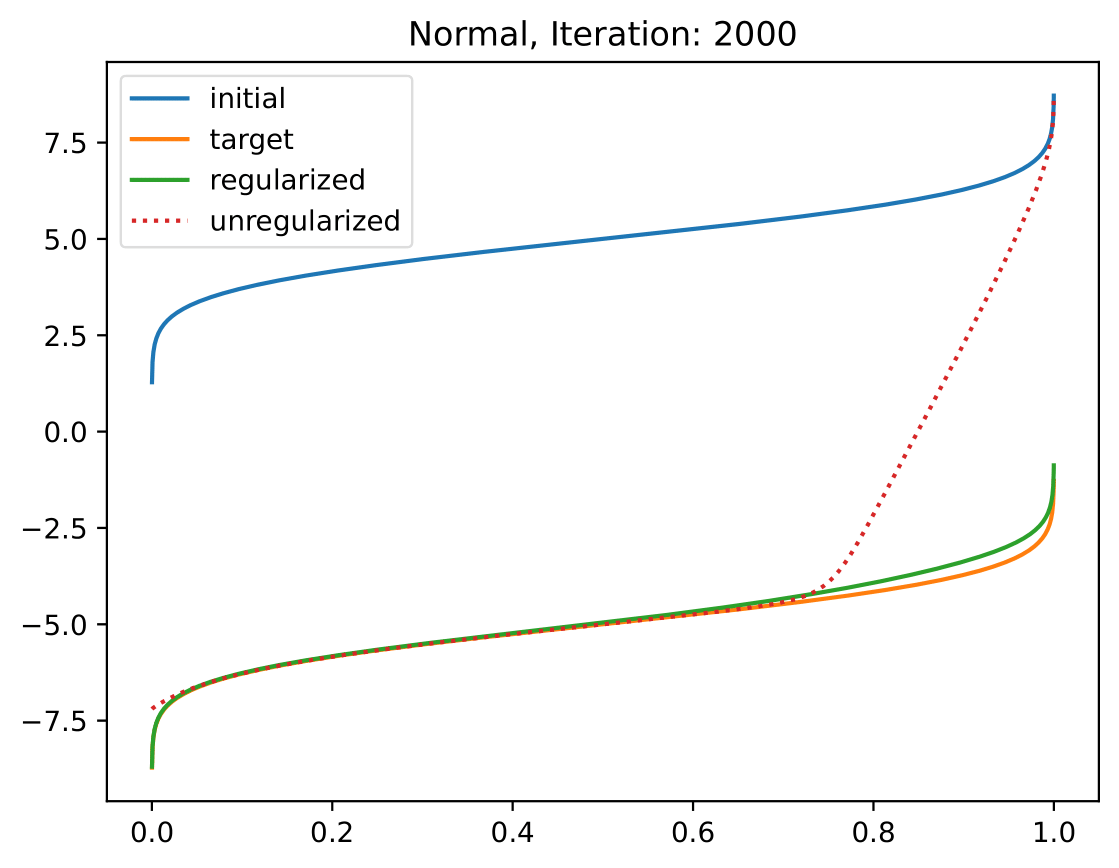}
    \includegraphics[width=.32\textwidth]{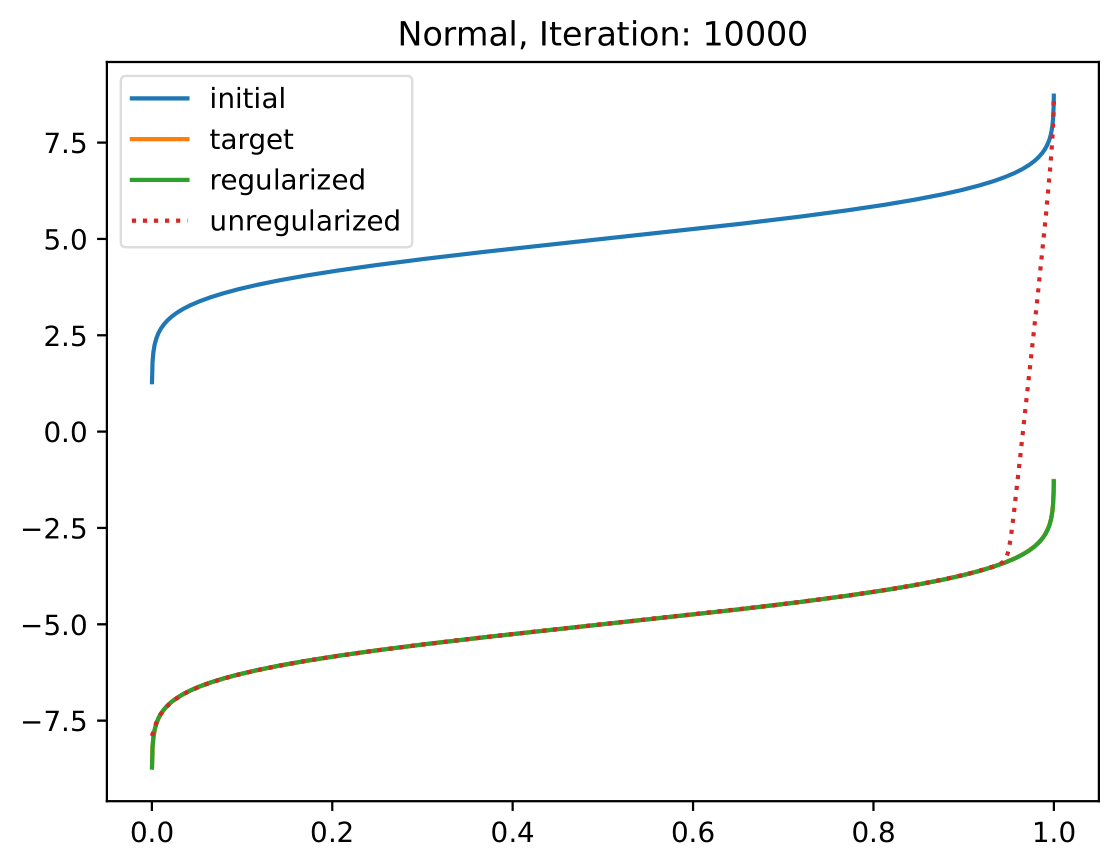}
    
    \caption{Quantiles of the regularized 
    and unregularized ${F}_\nu^-$-flow between the cut-out Gaussians $\mathcal{N}(5,1)$ and $\mathcal{N}(-5,1)$
		for $\tau = \lambda = 10^{-2}$.
    }
    \label{fig:Gaussian_to_Gaussian_quantiles}
\end{figure}

\paragraph*{Acknowledgments.}
RD acknowledges funding by the 
 German Research Foundation (DFG) within the project STE 571/16-1.
Furthermore, NR  gratefully acknowledges funding by the German Federal Ministry of Education and
Research BMBF 01|S20053B project SA$\ell$E.

\bibliographystyle{abbrv}
\bibliography{Bibliography}

\appendix
\section{Appendix}\label{sec:appendix}
\subsection{Proof of Theorem \texorpdfstring{\ref{thm:L2_representation}}{2.4}}
   The proof is taken from \cite[Theorem 3.5]{DSBHS2024} with slight adjustments to our setting.
    Let $\gamma_t \coloneqq (g(t))_{\#} \lebesgue_{(0,1)}$. Then, it is easy to check that $\gamma_t$ are probability measures in $\P_2(\R)$ with quantile functions $Q_{\gamma_t} = g(t)$.
    Next, we show that $\gamma$ is locally absolutely continuous (more precisely, in ${AC}^2_{\rm loc}((0,\infty); \P_2(\R))$).
    It holds by the isometry of Theorem \ref{prop:Q} that
    \begin{equation*}
        W_2(\gamma_t, \gamma_s)
        = \| g(t) - g(s) \|_{L_2(0, 1)}
        \le \int_{s}^{t} \|g'(\tau) \|_{L_2(0, 1)} \d{\tau} \qquad \text{for all } 0 < s \le t < \infty,
    \end{equation*}
    where with $m(\tau) \coloneqq \|g'(\tau) \|_{L_2(0, 1)}$, it holds 
    $m \in L_{2, {\rm loc}}(0, \infty)$ by assumption.
    Hence, the curve $\gamma \colon [0, \infty) \to (\P_2(\R), W_2)$ is locally absolutely continuous.
    
    Next, we show that the optimal velocity field $v_t$ of $\gamma_t$ given by \eqref{eq:CE} fulfills $v_t\in -\partial \F(\gamma_t)$.
    By \cite[Prop~8.4.6]{BookAmGiSa05},
    for a.e.~$t \in (0,\infty)$,
    the velocity field satisfies
    \begin{align}
        0
        &=\lim_{\tau\downarrow0}\frac{W_2(\gamma_{t+\tau},(\mathrm{Id}+\tau \, v_t)_\#\gamma_t)}{\tau} \\
        & =\lim_{\tau\downarrow0}\frac{W_2(g(t + \tau)_\#\lebesgue_{(0,1)},
        \bigl(g(t) +\tau \left(v_t\circ g(t) \right) \bigr)_\#\lebesgue_{(0,1)})}{\tau}.
    \end{align}
    Since $v_t \in \mathrm T_{\gamma_t} \mathcal P_2 (\R^d)$, 
    there exists $\tilde\tau > 0$
    such that
    $(\mathrm{Id}, \mathrm{Id} + \tilde\tau v_t)_\# \gamma_t$
    becomes an optimal transport plan.
    Moreover,
    \cite[Lem~7.2.1]{BookAmGiSa05} implies
    that the transport plans remain optimal 
    for all $0 < \tau \le \tilde \tau$.
    For small $\tau > 0$,
    the mappings $\mathrm{Id} + \tau v_t$ are thus optimal
    and, especially, monotonically increasing.
    Consequently,
    the functions $g(t) + \tau (v_t \circ g(t))$ are also monotonically increasing,
    and their left-continuous representatives are quantile functions.
    Employing the isometry to $L_2(0,1)$,
    we hence obtain
    \begin{equation}
    0
         = \lim_{\tau \downarrow0} \Big\| \frac{g(t+\tau)-g(t)}{\tau}- v_t \circ g (t) \Big\|_{L_2(0,1)}
        = \|\partial_t g(t) - v_t \circ g(t) \|_{L_2(0,1)}.
    \end{equation}
    Thus, since $g$ solves \eqref{eq:cauchy}, we see that
    $v_t \circ g(t) \in - \partial F \left( g(t) \right)$ for a.e.~$t$.
    In particular, 
    for any $\mu\in\P_2(\R)$,
    we obtain for a.e. $0 < t < \infty$ that
    \begin{align}
       0 \,\le \, & F(Q_\mu) - F\left(g(t) \right)
        + \int_0^1 \left( v_t\circ g(t)\right)(s) \, \big(Q_\mu(s)- \left(g(t)\right)(s) \big) \d s + o(\|Q_\mu - g(t)\|_{L_2(0,1)}) \\
        =\, &\mathcal F(\mu)-\mathcal F(\gamma_t) +\int_{\R\times\R} v_t(x) \, (y-x) \d \pi(x,y) + o(W_2(\mu, \gamma_t)),
    \end{align}
    where $\pi\coloneqq (g(t),Q_\mu)_\#\lebesgue_{(0,1)}$.
    By Theorem~\ref{prop:Q}, the plan
    $\pi$ is optimal between $\gamma_t$ and $\mu$, see also \cite[Thm.~2.18]{Vil03}.
    By \cite[Thm.~16.1(i),(ii)]{M2023}, we also know that $\pi$ is unique, so  that $v_t\in -\partial \F(\gamma_t)$. 
    
    Finally, the initial condition $\gamma(0+) = (g_0)_{\#} \lebesgue_{(0,1)}$ is satisfied, since $g$ is continuous at $t=0$ with $g(0) = g_0$. 
\hfill $\Box$

\subsection{Material from \texorpdfstring{\cite{RS2006}}{[Rossi, Savaré]}}
\begin{lemma}[{\cite[Lemma 1.2]{RS2006}}]
\label{lem: compacnes lemma}
Let us assume that
\begin{align}
\phi \colon L_2(0,1) \to (-\infty, +\infty]&\text{ is proper, lower semicontinuous, and} \\
\exists \, \tau_* > 0 : v \mapsto \varphi(v) + &\frac{1}{2 \tau_*} |v|^2 \text{ has compact sublevels,} \label{eq:comp} 
\end{align}
and the data satisfy
\begin{align}
u_0 \in \dom(\phi). \label{eq:data} 
\end{align}
Then, $\textup{GMM}(\phi, u_0)$ is not empty and every $u \in \textup{GMM}(\phi, u_0)$ belongs to $H^1(0, T ;L_2(0,1))$.
\end{lemma}

\begin{theorem}[{\cite[Theorem 3]{RS2006}}]
 Let us suppose that $
\phi \colon L_2(0,1) \to (-\infty, +\infty]$ and $ u_0 $
satisfy the assumptions \eqref{eq:comp} and \eqref{eq:data} of the compactness Lemma \ref{lem: compacnes lemma}, so that \( \text{GMM}(\phi, u_0) \) is not empty.
If \( \phi \) satisfies the following  chain rule condition: \begin{align} \text{if } v \in H^1(0, T; L_2(0,1)),\, \xi \in \, &L_2(0, T; L_2(0,1)), \, \xi \in \partial \phi(v)  \text{ a.e. in }  (0, T),\\
\text{ and }  \phi \circ v  \text{ is bounded, then }&\phi \circ v  \text{ is absolutely continuous on }(0, T) \text{ and } \label{eq:chain2} 
\\
\frac{d}{dt} \phi(v(t)) = &\langle \xi(t), {v}'(t) \rangle
\text{ for a.e. }  t \in (0, T),
\end{align}
then, every \( u \in \textup{GMM}(\phi, u_0) \) is a solution in \( H^1(0, T; L_2(0,1)) \) to 
\begin{align}
\begin{cases}
u'(t) + \partial_l \phi(u(t)) \ni 0, & \text{a.e. in } (0, T), \\
u(0) = u_0. & 
\end{cases}
\end{align}
\end{theorem}

\begin{remark}[{\cite[Remark 1.9]{RS2006}}] Assume \(\psi_1\) and \(\psi_2\) are convex and 
\begin{align}
\phi = \psi_1 - \psi_2 \quad \text{in } \dom(\phi)&, \text{ where}\\
\psi_1, \psi_2 \colon \dom(\phi) \to \mathbb{R} \text{ are lsc and convex}&, ~\dom(\partial \psi_1) \subset \dom(\partial \psi_2). 
\end{align}
Then, $\phi$ satisfies the chain rule \eqref{eq:chain2}.
\end{remark}

\newpage
\subsection{Influence of different values of the regularization parameter \texorpdfstring{$\lambda$}{ }}

\begin{figure}[H]
    \centering
    \includegraphics[width=.314\textwidth]{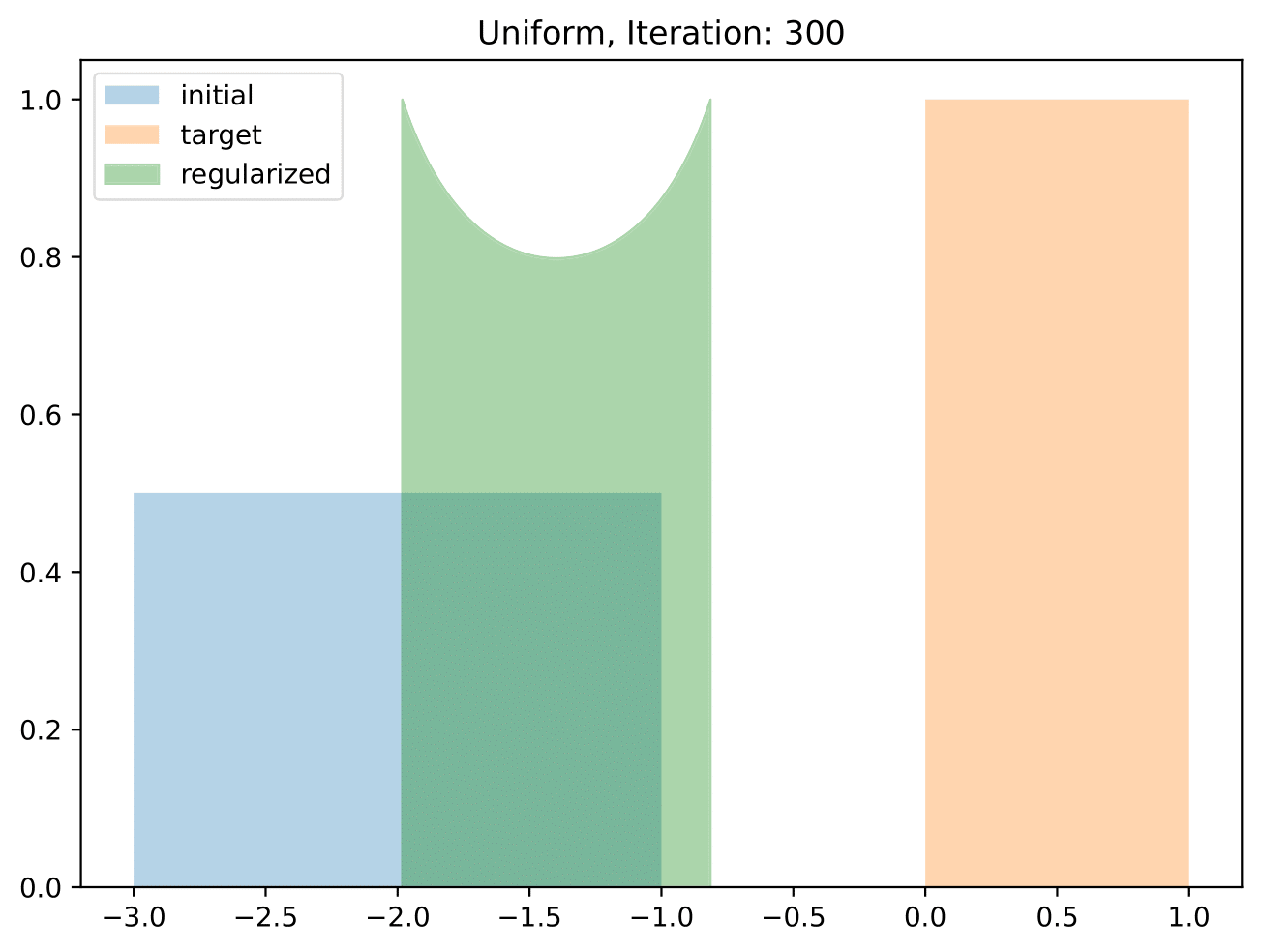}
    \includegraphics[width=.314\textwidth]{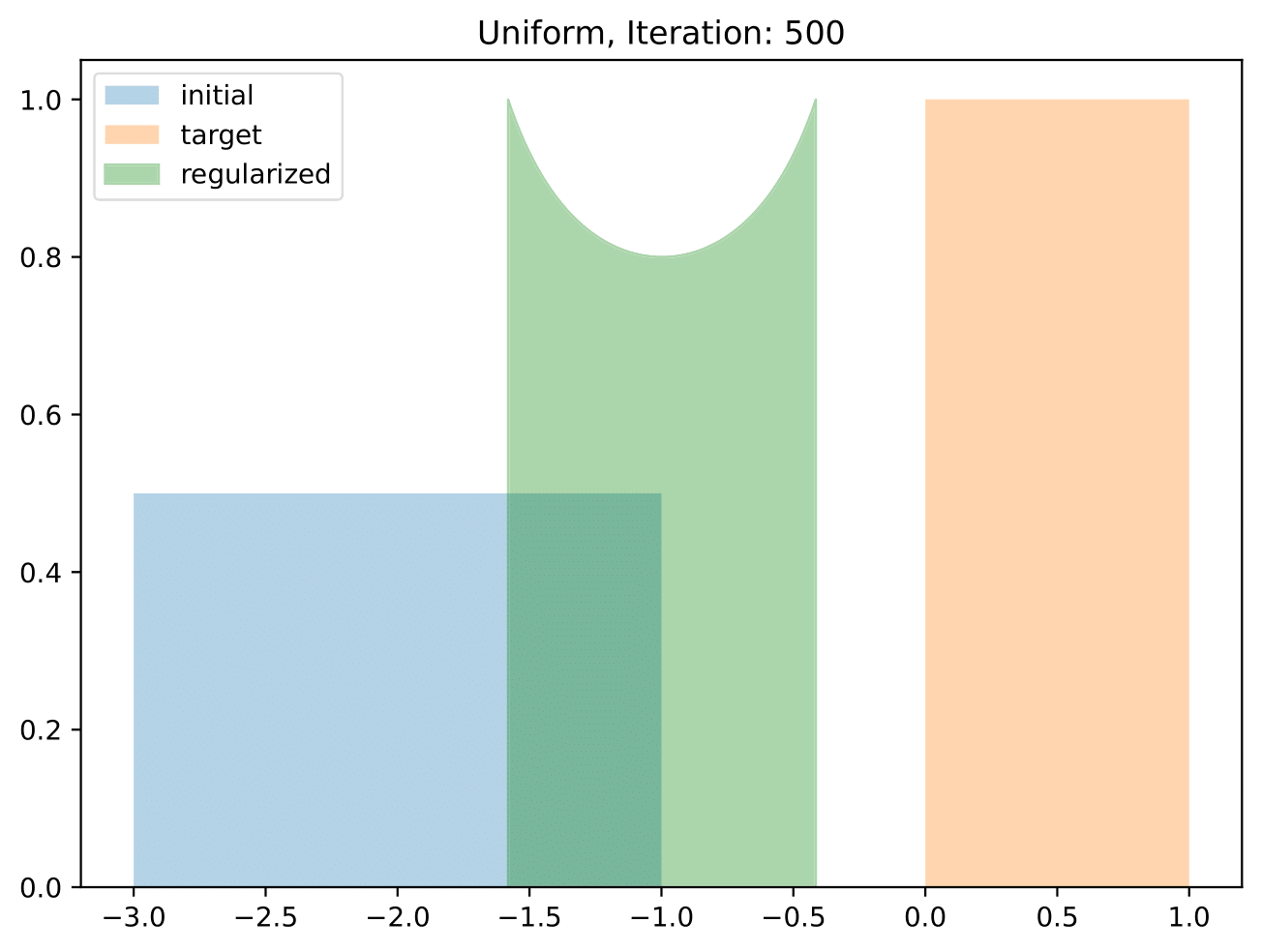}
    \includegraphics[width=.314\textwidth]{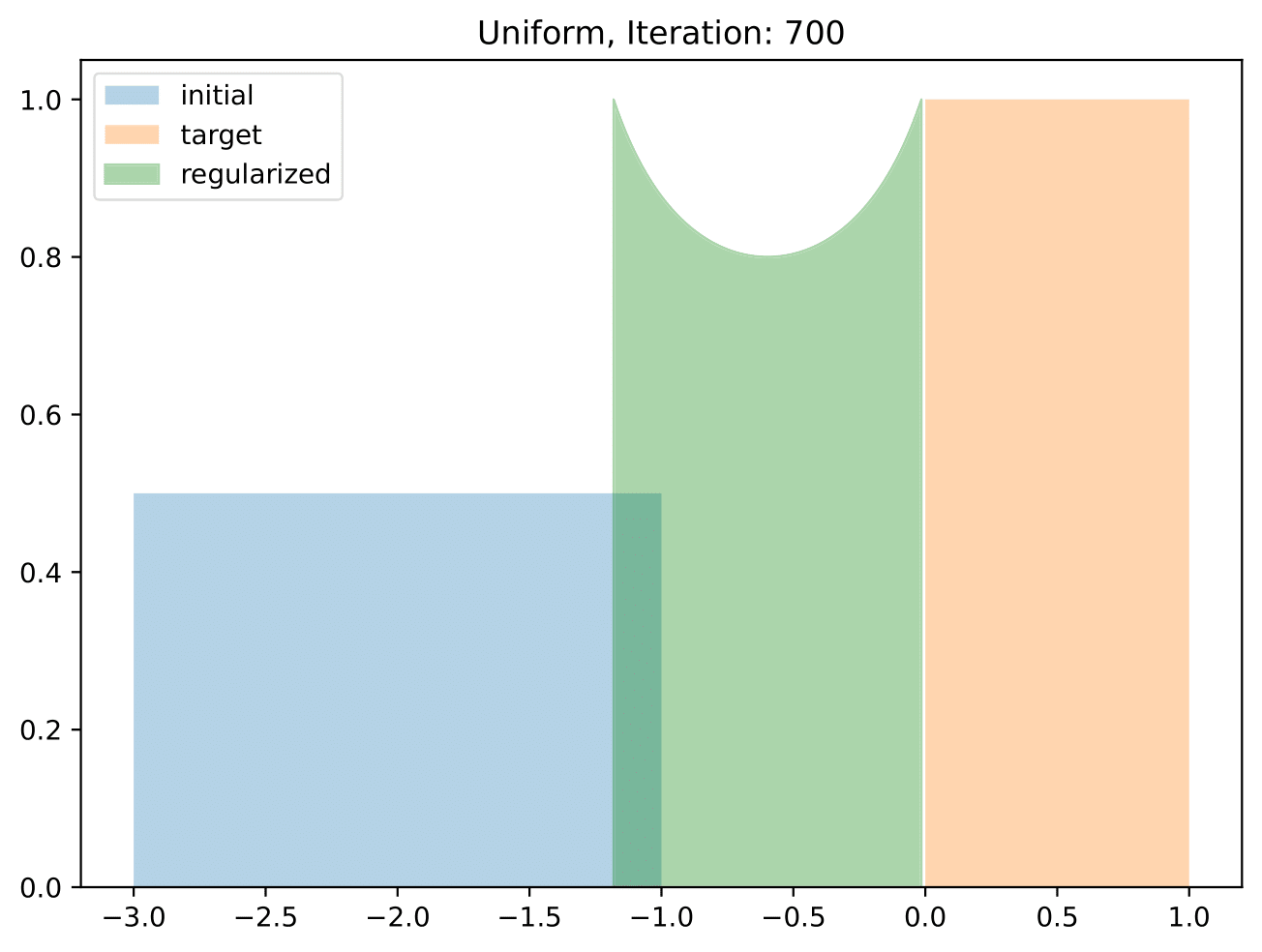}
    \\
    \includegraphics[width=.314\textwidth]{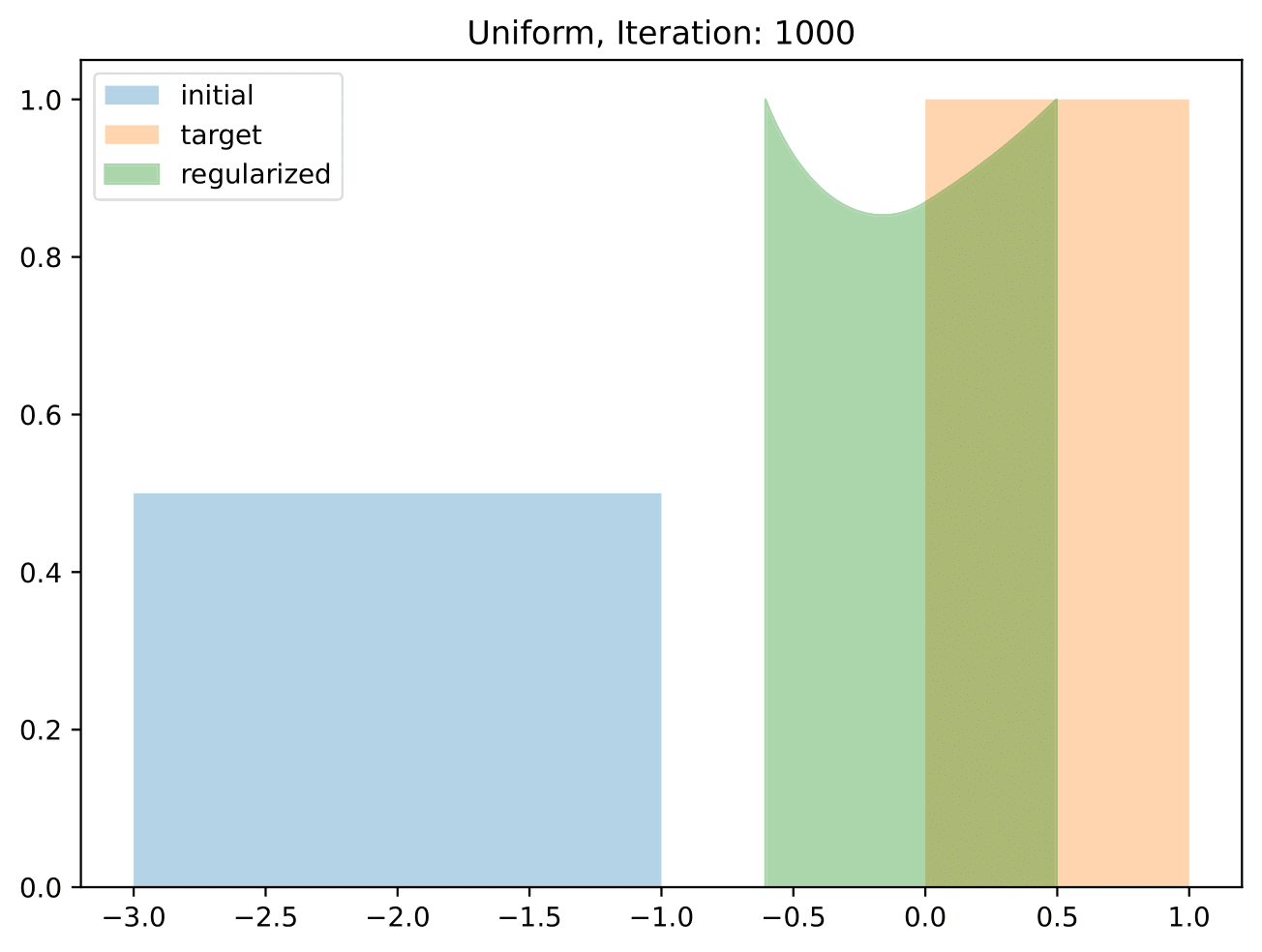}
    \includegraphics[width=.314\textwidth]{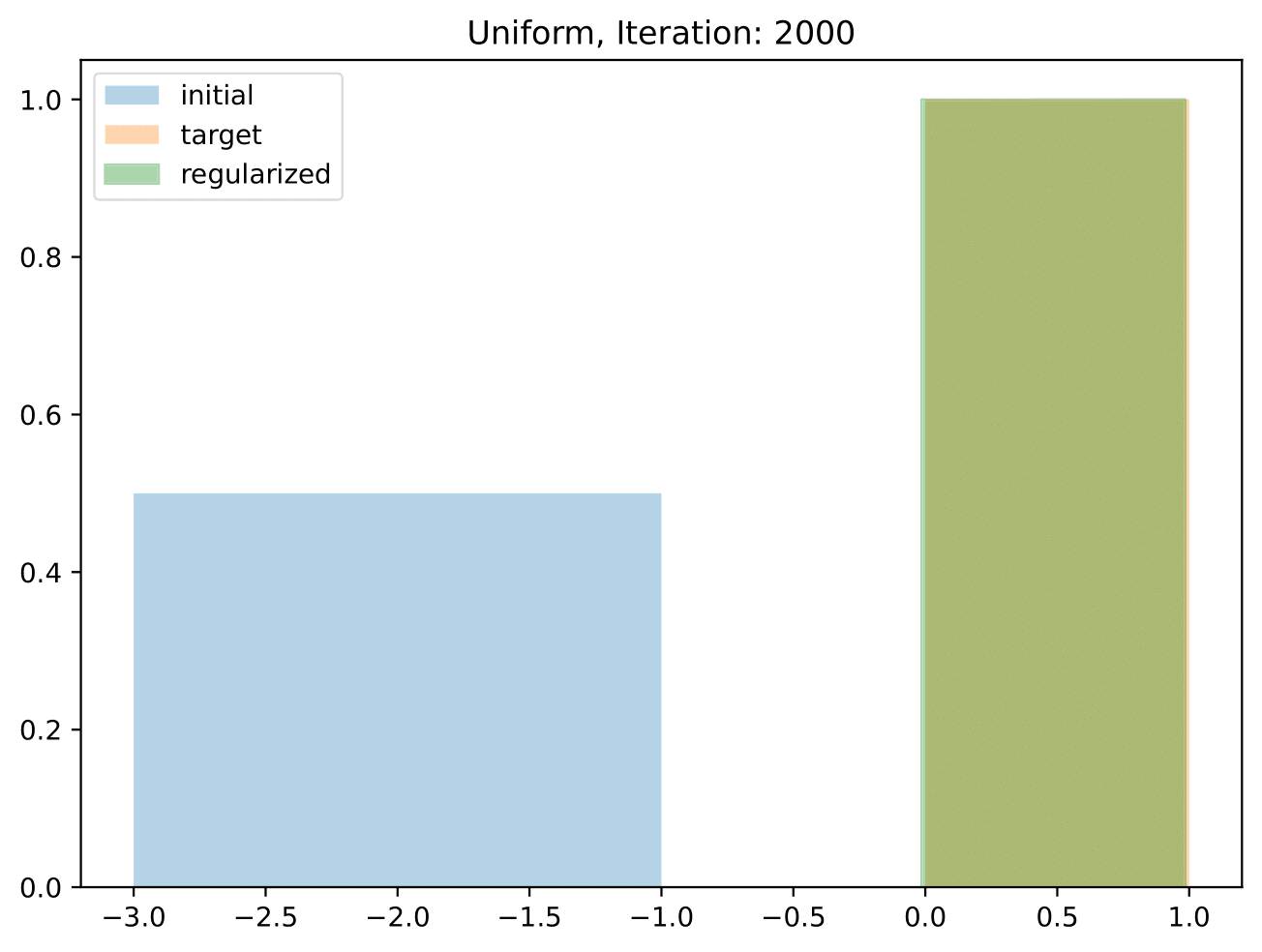}
    \includegraphics[width=.314\textwidth]{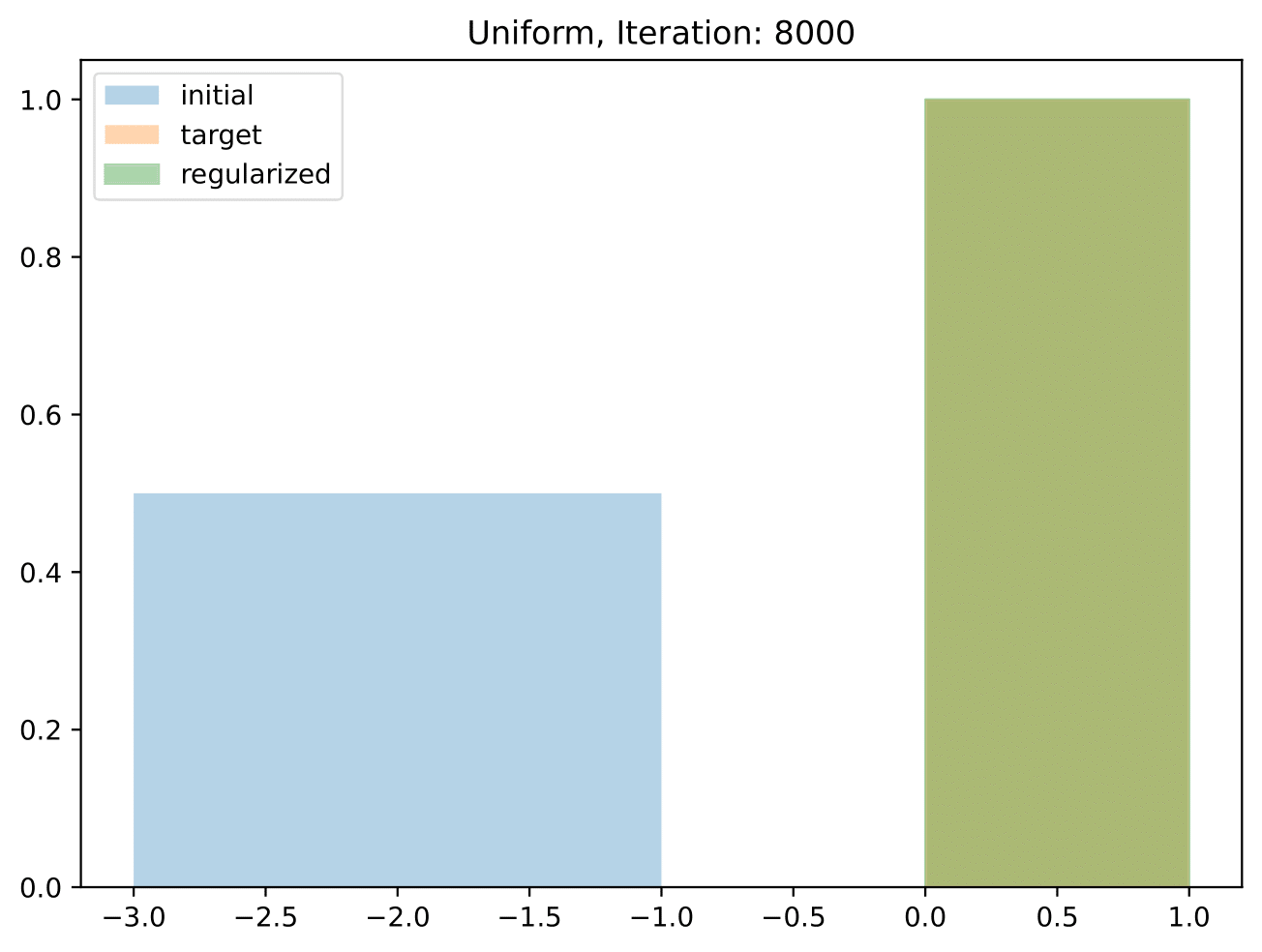}
    \caption{Regularized ${\mathcal F}_\nu^-$-flow  between the uniform measures $\gamma_0 \sim \mathcal{U}[-3,-1]$ 
		and $\nu \sim \mathcal{U}[0,1]$ for
		$\tau = 2\cdot 10^{-3}$ and "large" $\lambda = 1$. The support shifts quickly towards the target, as indicated in Remark \ref{rem:long-time}.
    }
    \label{fig:Uniform_to_Uniform_large}
\end{figure}
    
\begin{figure}[H]
    \centering
    \includegraphics[width=.316\textwidth]{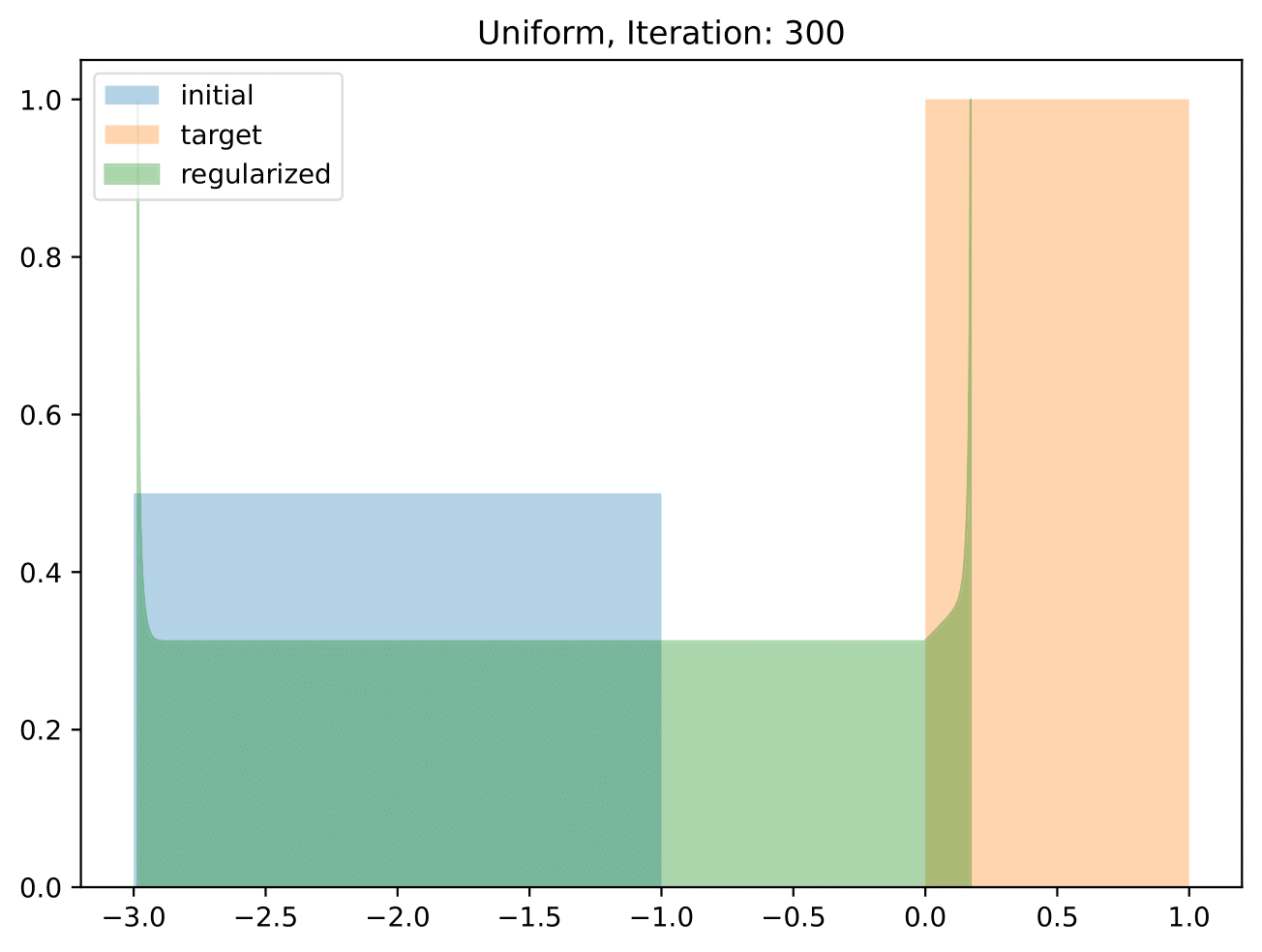}
    \includegraphics[width=.316\textwidth]{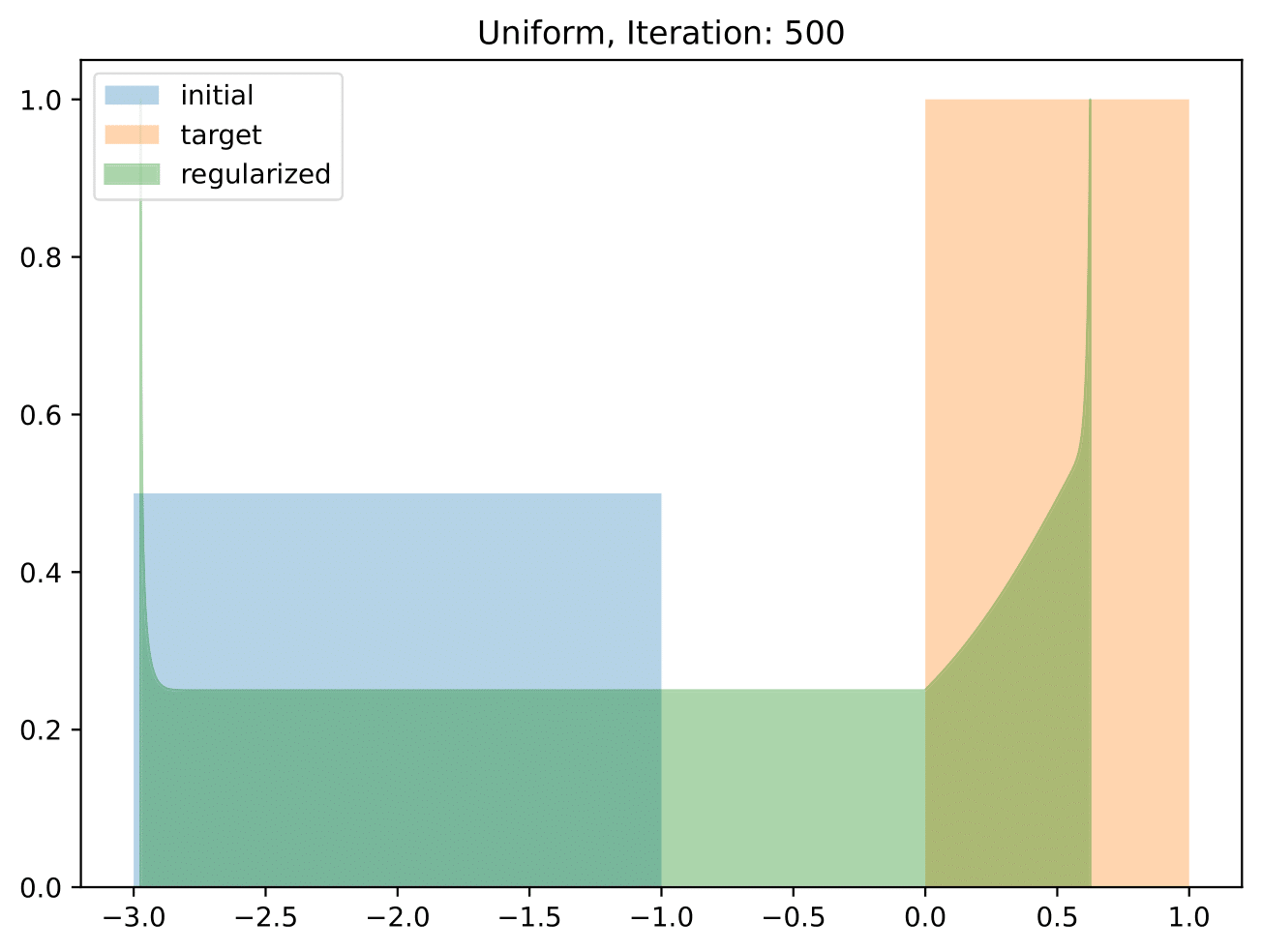}
    \includegraphics[width=.316\textwidth]{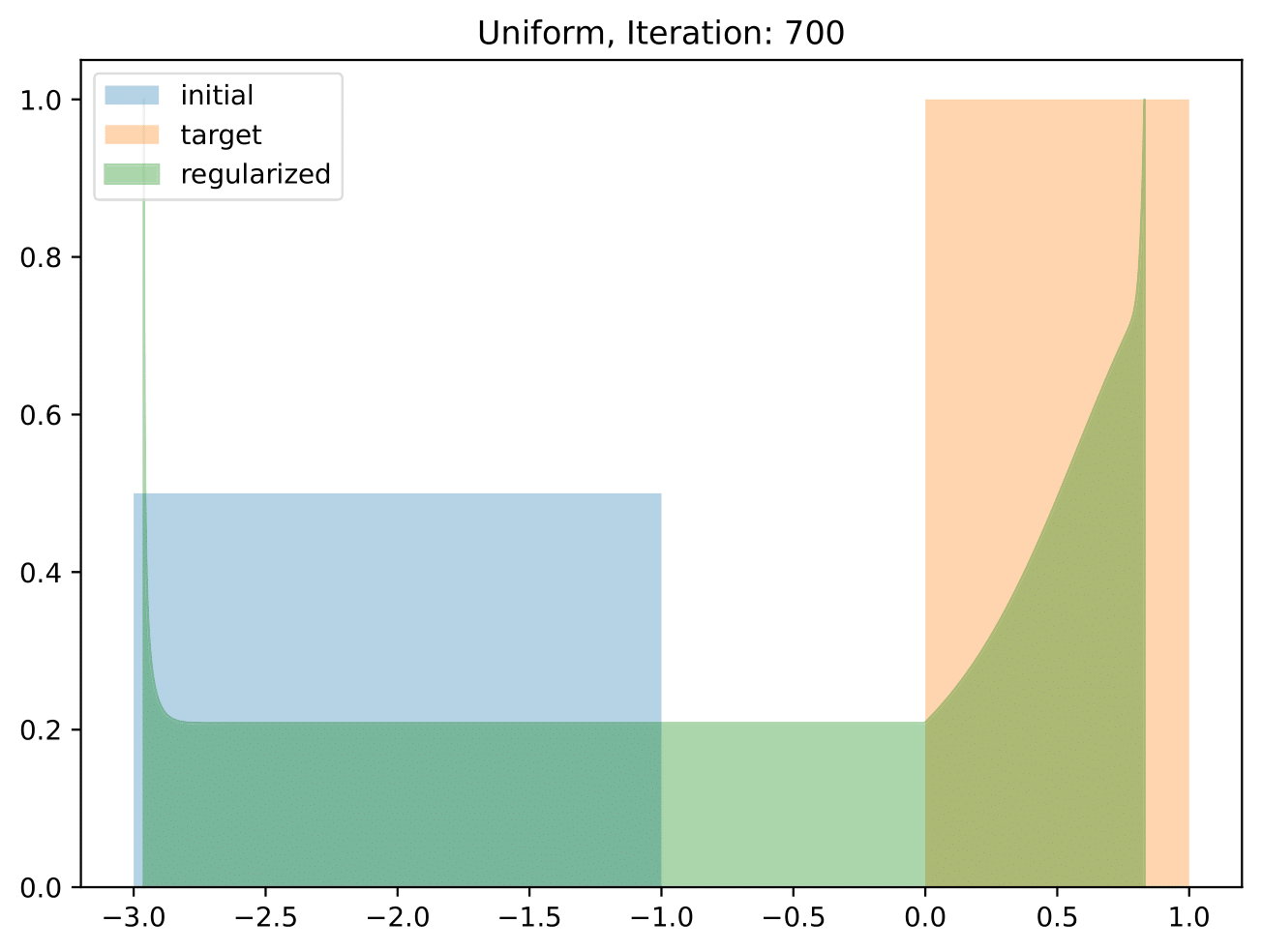}
    \\
    \includegraphics[width=.316\textwidth]{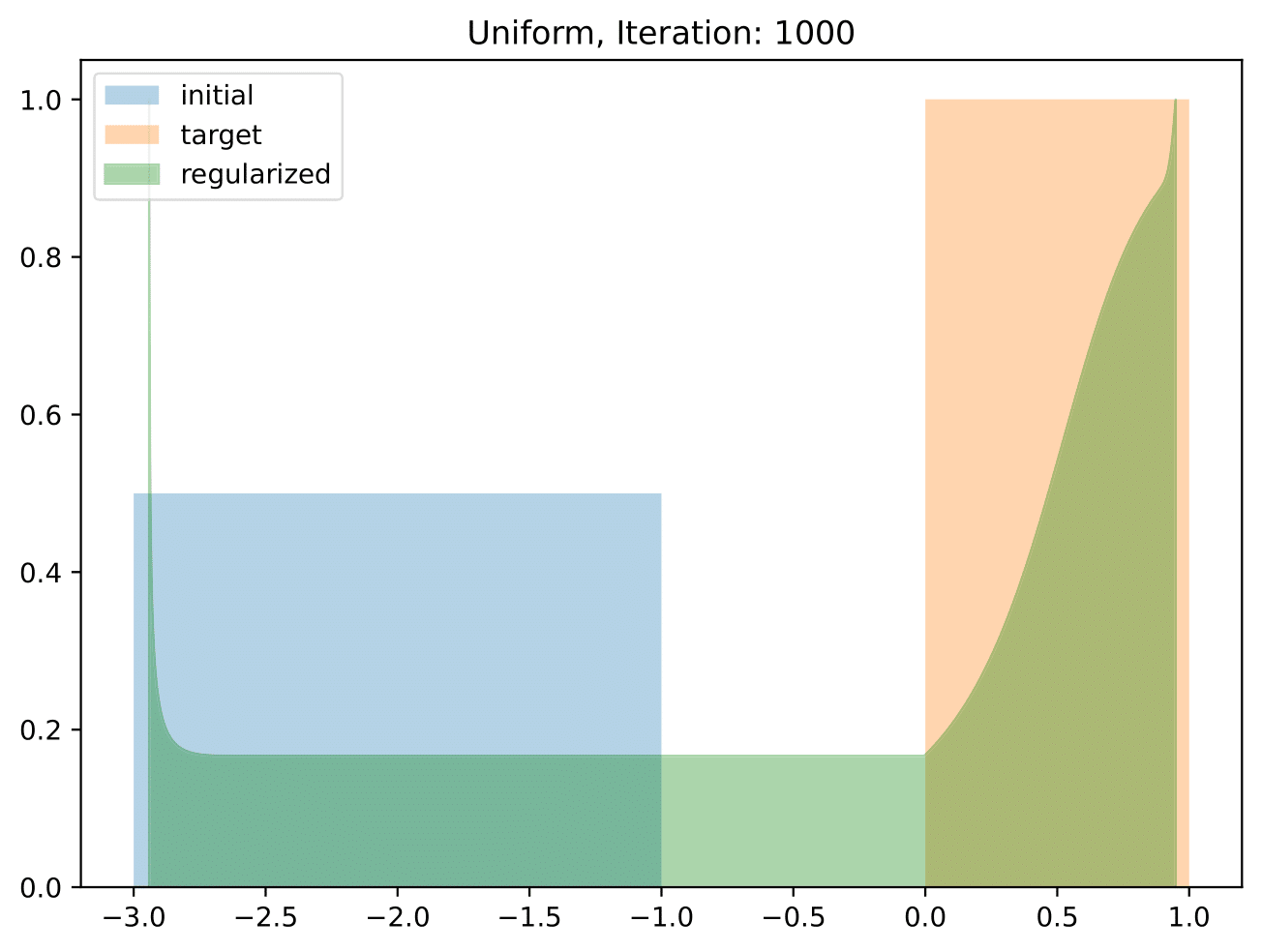}
    \includegraphics[width=.316\textwidth]{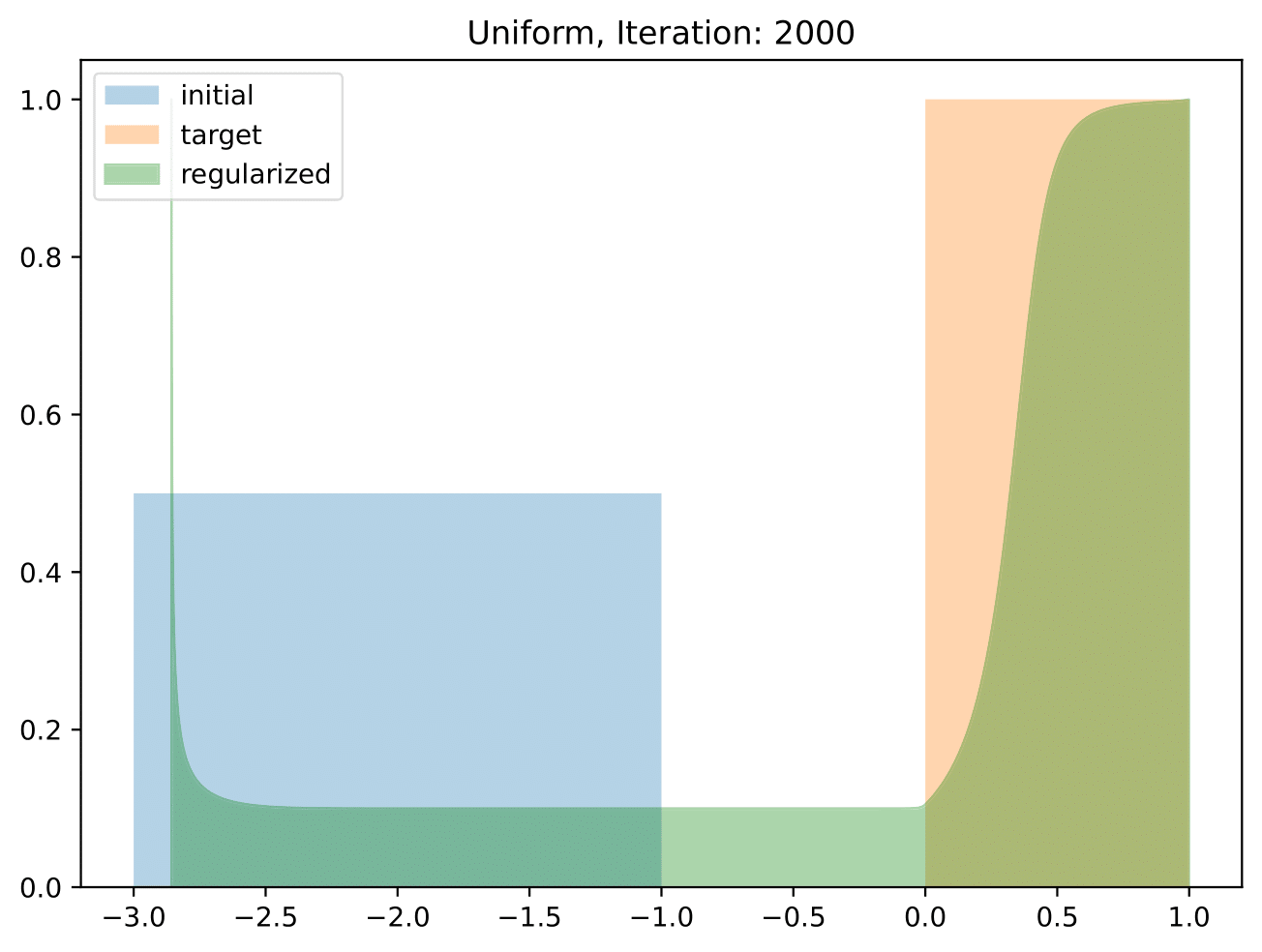}
    \includegraphics[width=.316\textwidth]{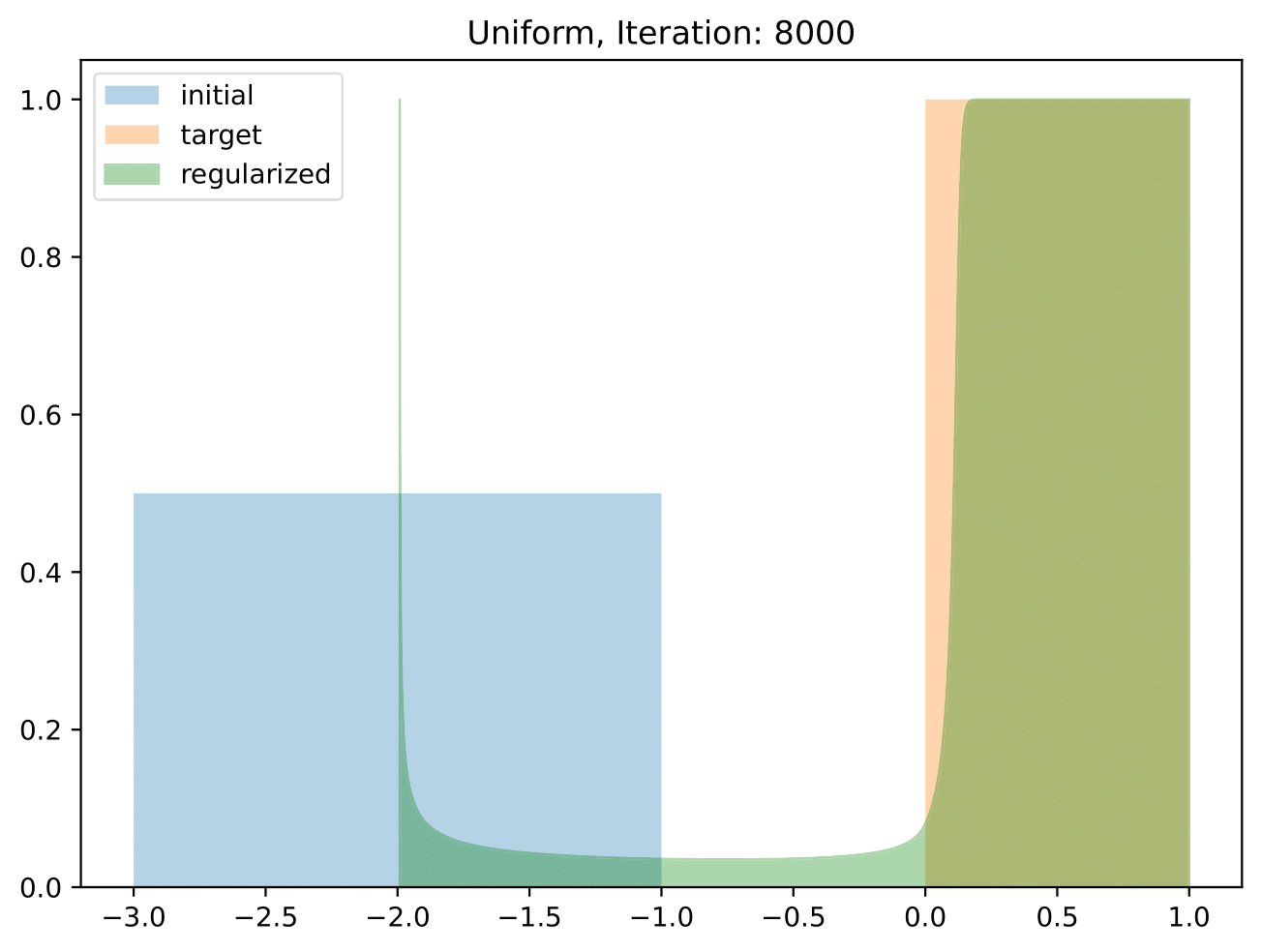}
    \caption{Regularized ${\mathcal F}_\nu^-$-flow  between the uniform measures $\gamma_0 \sim \mathcal{U}[-3,-1]$ 
		and $\nu \sim \mathcal{U}[0,1]$ for
		$\tau = 2\cdot 10^{-3}$ and "small" $\lambda = 10^{-4}$. The support shifts only slowly towards the target but expands quickly, similar to the unregularized flow for small times, see Remark \ref{rem:semigroup-convergence}.
    }
    \label{fig:Uniform_to_Uniform_small}
\end{figure}

\end{document}